\documentclass[12pt,reqno]{amsart}
\usepackage{amsmath,eucal,verbatim,amsopn,amsthm,amsfonts,amssymb,mathdots,mathrsfs,esint,mathtools,enumitem,bbm}
\usepackage[margin=1in]{geometry}
\let\mathcal=\CMcal
\allowdisplaybreaks[4]

\def\Alg{\operatorname{Alg}}
\let\Re\relax
\DeclareMathOperator*{\Re}{Re}
\let\Im\relax
\DeclareMathOperator*{\Im}{Im}

\newcommand{\setof}[2]{\ensuremath{\{ #1 : #2 \}}}
\newcommand{\isomorphic}{\ensuremath{\cong}}

\newcommand{\mathe}{\ensuremath{\mathrm{e}}}
\newcommand{\transpose}[1]{\ensuremath{\rconj{t}#1}}

\def\Hom{\operatorname{Hom}}
\def\End{\operatorname{End}}

\DeclareMathOperator*{\Dim}{dim}
\def\diag{\mathrm{diag}}
\DeclareMathOperator*{\Ad}{Ad}

\newcommand{\GL}{\ensuremath{\mathrm{GL}}}
\newcommand{\Sp}{\ensuremath{\mathrm{Sp}}}
\newcommand{\SL}{\ensuremath{\mathrm{SL}}}
\newcommand{\Pin}{\ensuremath{\mathrm{Pin}}}
\newcommand{\rmodulo}[2]{#1 / #2} 

\newcommand{\lmodulo}[2]{\ensuremath{#1 \backslash #2}}
\newcommand{\rconj}[1]{\ensuremath{{}^{#1}}}
\def\Alg{\operatorname{Alg}}
\def\Ind{\operatorname{Ind}}
\def\ind{\operatorname{ind}}
\newcommand{\C}{\ensuremath{\mathbb{C}}}
\newcommand{\R}{\ensuremath{\mathbb{R}}}
\newcommand{\N}{\ensuremath{\mathbb{N}}}
\newcommand{\Z}{\ensuremath{\mathbb{Z}}}
\newcommand{\scrS}{\ensuremath{\mathscr{S}}}
\newcommand{\Adele}{\ensuremath{\mathbb{A}}}
\newcommand{\rhs}{right-hand side} 
\newcommand{\lhs}{left-hand side} 
\newcommand{\setdifference}{\ensuremath{\mathrm{-}}}
\newcommand{\blank}{\,\cdot\,}

\newtheorem{theo}{Theorem}[section]

\newtheorem{theorem}{Theorem}[section]
\newtheorem{assumption}[theorem]{Assumption}
\newtheorem{lemma}[theorem]{Lemma}
\newtheorem{proposition}[theorem]{Proposition}
\newtheorem{corollary}[theorem]{Corollary}
\newtheorem{claim}[theorem]{Claim}
\newtheorem{remark}[theorem]{Remark}
\newtheorem{example}[theorem]{Example}
\numberwithin{equation}{section}

\begin{document}
\title[]{A Godement--Jacquet type integral and the metaplectic Shalika model}
\author{Eyal Kaplan}
\address{Department of Mathematics, The Ohio State University, Columbus, OH 43210, USA}
\email{kaplaney@gmail.com}
\author{Jan M\"{o}llers}
\address{Department Mathematik, FAU Erlangen--N\"{u}rnberg, Cauerstr. 11, 91058 Erlangen, Germany}
\email{moellers@math.fau.de}

\begin{abstract}
We present a novel integral representation for a quotient of global automorphic $L$-functions, the symmetric square over the exterior square. The pole of this integral characterizes a period of a residual representation of an Eisenstein series. As such, the integral itself constitutes a period, of an arithmetic nature. The construction involves the study of local and global aspects of a new model for double covers of general linear groups, the metaplectic Shalika model. In particular, we prove uniqueness results over $p$-adic and Archimedean fields, and a new Casselman--Shalika type formula.
\end{abstract}


\maketitle

\section*{Introduction}\label{section:introduction}
Let $\Adele$ be the ring of ad\`{e}les of a global field and let $\pi$ be a cuspidal automorphic
representation of $\GL_k(\Adele)$. One of the pillars of the Langlands program is the study of global automorphic
$L$-functions, as mediating agents in the framework of functoriality. Integral representations are an
indispensable tool in this study. Godement and Jacquet \cite{GJ} constructed the integral representation of the standard $L$-function of $\pi$.
Jacquet, Piatetski-Shapiro and Shalika developed Rankin--Selberg integrals for $\GL_k\times\GL_m$. Numerous authors considered integral representations for a product of a classical group and $\GL_m$, e.g., \cite{GPS,G,Soudry,GPSR,GRS4}, to name a few.


In several settings, Rankin--Selberg integrals offer a ``bonus", namely a period integral characterizing the pole or central value of the $L$-function. For example, the pole of the exterior square $L$-function at $s=1$ was characterized by the nonvanishing of a global Shalika period \cite{JS4}. A similar result was obtained for the symmetric square $L$-function by Bump and Ginzburg \cite{BG}. Periods of automorphic forms have been studied extensively, e.g., \cite{GRS5,GRS7,GJR3}, and they are especially interesting when they are of an arithmetic nature \cite{GJR3}.

We present a new integral representation and use it to characterize the nonvanishing of a co-period of an Eisenstein series. Let $f$ be a global matrix coefficient of $\pi$, and $\varphi$ and $\varphi'$ a pair of automorphic forms in the space of the exceptional representation of a double cover of $\GL_{n}(\Adele)$, $n=2k$. The cover and the exceptional representation were constructed by Kazhdan and Patterson \cite{KP}. We define a zeta integral $Z(f,\varphi,\varphi',s)$ (see \S~\ref{section:global theory}), where we integrate $f$ against the Fourier coefficients of $\varphi$ and $\varphi'$ along the Shalika unipotent subgroup and its (Shalika) character. Technically, it resembles the Godement--Jacquet zeta integral, where a matrix coefficient was integrated against a Schwartz--Bruhat function \cite{GJ}, but they are not to be confused - our integral is not an entire function. It is absolutely convergent in some right half-plane, and for decomposable data becomes Eulerian. The computation of the resulting local integrals with unramified data lets us identify the $L$-functions represented. We obtain the meromorphic continuation from the known properties of these $L$-functions, and a detailed analysis of the local integrals.
\begin{theo}[see Theorem~\ref{theorem:properties of the Global GJ integral}]\label{theo:Z represents L function}
The zeta integral extends to a meromorphic function on the plane, and represents the quotient
of partial $L$-functions
\begin{align*}
\frac{L^S(2s,\mathrm{Sym}^2,\pi)}{L^S(2s+1,\wedge^2,\pi)}.
\end{align*}
\end{theo}
The corresponding local integrals are studied at both $p$-adic and Archimedean places in \S~\ref{subsection:Local p-adic Godement--Jacquet integral}. In light of existing Rankin--Selberg constructions, the denominator is expected to arise from the normalization of an Eisenstein series on $\mathrm{SO}_{2n}$. This is not evident at present (see the discussion below). On the other hand, the presence of exceptional representations and a cover of $\GL_n$ in an integral representing the symmetric square $L$-function is well understood, considering the role these played in the Rankin--Selberg integral of Bump and Ginzburg \cite{BG}, and in its extension to the twisted symmetric square $L$-function by Takeda \cite{Tk}.

The symmetric square $L$-function has been studied extensively, both via the
Langlands--Shahidi method by Shahidi (e.g., \cite{Sh4,Sh7,Sh3}), and the Rankin--Selberg method by Bump and Ginzburg \cite{BG}. The present zeta integral is weaker than these works, in the sense that it does not seem to yield a global functional equation. At least at $p$-adic places, local multiplicity-one results are known (proved in \cite{me11}), then any  ``similar" type of integral may be used to define a functional equation with a proportionality factor. However, for this to be useful the local proportionality factors should also be related to a global one. See Remark~\ref{remark:local gamma factor} for more details. On the other hand, we can place the global integral in a different framework, that of periods.

The notion of co-period integrals was introduced by Ginzburg, Jiang and Rallis \cite{GJR2}, for the purpose of characterizing the zero of the symmetric cube $L$-function of a cuspidal automorphic representation of $\GL_2(\Adele)$ at $1/2$, using the spectral decomposition of a tensor of two exceptional representations of the $3$-fold cover of $\GL_2(\Adele)$. Their co-period was an integral of an Eisenstein series against two automorphic forms in the space of an exceptional representation of the $3$-fold cover of the exceptional group $\mathrm{G}_2(\Adele)$. A co-period involving an Eisenstein series and two automorphic forms in the space of an exceptional representation of a cover of $\mathrm{SO}_{2n+1}(\Adele)$ (or $\mathrm{GSpin}_{2n+1}(\Adele)$) was analyzed in \cite{me7,me8} (following \cite{GJS}).

In an ongoing work, Shunsuke Yamana and the first named author study a co-period integral on $\GL_n(\Adele)$.
Assume that $\pi$ is self-dual with a trivial central character.
Let $E(g;\rho,s)$ be an Eisenstein series corresponding to a vector $\rho$, of the representation parabolically induced from $\pi\otimes\pi^{\vee}$ to $\GL_{n}(\Adele)$, such that the residue
$E_{1/2}(g;\rho)$ of the series at $s=1/2$ is related to the pole of $L(s,\pi\times\pi)$ at $s=1$ (in fact, a more general setting is studied).
The global co-period integral $CP(E_{1/2}(\blank;\rho),\varphi,\varphi')$ is an integral of the residue against $\varphi$ and $\varphi'$ (see \S~\ref{section:The co-period integral}). The integral lends itself to a close scrutiny using the truncation operator of Arthur \cite{A1,A2}. Shunsuke Yamana and the first named author proved that
this co-period is nonvanishing if and only if
$L^S(s,\mathrm{Sym}^2,\pi)$ has a pole at $s=1$. Moreover, the co-period was unfolded to an explicit ``outer period", resembling the integral of the symmetric square $L$-function of \cite{BG}.

It turns out that, using a different analysis and under a mild assumption regarding the Archimedean places (see Assumption~\ref{Assumption:Archimedean result for co-period}), we can obtain a more precise ``outer period". Namely, we prove:
\begin{theo}[see Theorem~\ref{theorem:co-period}]\label{theo:co-period result}
$CP(E_{1/2}(\blank;\rho),\varphi,\varphi')=\int_{K}\mathrm{Res}_{s=1/2}Z(f_{k\rho},k\varphi,k\varphi',s)\,dk$,
where $K$ is the standard (global) compact subgroup and $f_{k\rho}$ is a matrix coefficient on $\pi$. Moreover,
the co-period is not identically zero if and only if the zeta integral has a nontrivial residue at $s=1/2$.
\end{theo}
The results on the co-period have a local $p$-adic counterpart, proved in \cite{me11}. Let $\tau$ be a self-dual irreducible supercuspidal representation of $\GL_k$. The results of \cite{me11} imply that if the symmetric square $L$-function of $\tau$ has a pole at $s=0$, the Langlands quotient of the representation parabolically induced from $|\det|^{1/2}\tau\otimes|\det|^{-1/2}\tau$ to $\GL_n$ has a nontrivial bilinear $\GL_n$ pairing with a tensor of exceptional representations.

Our zeta integral involves a new model for double covers of $\GL_n$, the metaplectic Shalika model, first discovered
in \cite{me11} during the computation of twisted Jacquet modules of the $p$-adic exceptional representation (of \cite{KP}). The Jacquet module of this representation along the Shalika unipotent subgroup and character is one-dimensional, and the reductive part of its stabilizer acts by a Weil factor. Therefore, this representation admits a unique embedding in a space of metaplectic Shalika functions.
The key observation underlying this model for the cover, is that the restriction of the cover to the Shalika group is simple, in the sense that the cocycle is given by the quadratic Hilbert symbol. This follows immediately from the block-compatibility of the cocycle of Banks, Levy and Sepanski \cite{BLS}. See \S~\ref{subsection:metaplectic Shalika functional and model} for a precise introduction to this model.

Roughly speaking, a model of a representation is an embedding, preferably unique, into a space of functions on the group,
which is beneficial and convenient in terms of applications. One important model is the Whittaker model, which has had a profound impact on the study of representations, with a vast number of applications, perhaps most notably the Langlands--Shahidi theory of local coefficients. This model has a natural extension to covering groups, because unipotent subgroups are split under the cover. Indeed, the metaplectic Whittaker model has been studied, e.g., in \cite{KP,BFJ,McNamara2,COf}. The main downside to using this model for covering groups, is that multiplicity one no longer holds. For instance over an $r$-fold cover of $\GL_n$, the dimension of the space of Whittaker functionals on an unramified principal series representation is essentially $r^n$.

The (non-metaplectic) Shalika model was first introduced in a global context in the aforementioned work of Jacquet and Shalika \cite{JS4}. Friedberg and Jacquet \cite{FJ} used it to characterize cuspidal automorphic representations affording linear models, in terms of the pole of the exterior square $L$-function at $s=1$ (see also \cite{BF}). Ash and Ginzburg \cite{AG} used this model to construct $p$-adic $L$-functions. Multiplicity one results were proved by
Jacquet and Rallis \cite{JR2} (over $p$-adic fields) and by Aizenbud, Gourevitch and Jacquet \cite{AGJ} over Archimedean fields. Globally, a cuspidal automorphic representation $\pi$ affords a Shalika functional (given by an integral) if and only if $L^S(s,\wedge^2,\pi)$ has a pole at $s=1$, and equivalently is a weak functorial lift from $\mathrm{SO}_{2n+1}$ (\cite{JS4,GRS7,GRS3}). A similar local result was proved for supercuspidal representations in \cite{JNQ}, and is expected to hold in greater generality. Further studies and applications of this model include \cite{GRS5,JQ,JNQ3}.

In this work we launch a study of the metaplectic Shalika model.
In sharp contrast to the Whittaker model, we conjecture that the metaplectic Shalika model does enjoy multiplicity one. We establish the following result.
\begin{theo}[see Theorems~\ref{theorem:unramified Shalika functional 1 dim} and \ref{theorem:uniqueness of metaplectic Shalika over R}]\label{theo:uniqueness}
Let $\mathrm{I}(\chi)$ be a principal series representation of a double cover of $\GL_n$, over any local field, but over a $p$-adic field assume it is unramified. Assume $\chi$ satisfies a certain regularity condition, which does not preclude reducibility. Then the space of metaplectic Shalika functionals on $\mathrm{I}(\chi)$ is at most one-dimensional.
Over a $p$-adic field and when $\mathrm{I}(\chi)$ is irreducible, this space is one-dimensional if and only if $\mathrm{I}(\chi)$ is a ``lift" from a classical group.
\end{theo}
It is perhaps premature to discuss lifting problems for covering groups; here we simply mean the natural condition one expects $\chi$ to satisfy, analogous to the case of $\GL_n$.
This result is our main tool to deduce that the (global) zeta integral is Eulerian.

Another practical starting point for the study of the metaplectic Shalika model, is the unramified setting. We prove a Casselman--Shalika formula for an unramified principal series representation of a double cover of $\GL_n$. For $\GL_n$ this formula was proved by Sakellaridis \cite{Yia}, and our proof closely follows his arguments.
\begin{theo}[see Theorem~\ref{theorem:Casselman--Shalika formula for metaplectic H} and Corollary~\ref{corollary:Casselman--Shalika formula for metaplectic Shalika}]\label{theo:Casselman--Shalika}
Let $\mathrm{I}(\chi)$ be an unramified principal series representation of a double cover of $\GL_n$, which admits a
metaplectic Shalika model. There is an explicit formula for the values of the unique unramified normalized function in the model, in terms of the character $\chi$.
\end{theo}
We apply our formula to the exceptional representation and use it to compute the local zeta integrals with unramified data.

In his proof, Sakellaridis \cite{Yia} used a variant of the Casselman--Shalika method developed by Hironaka \cite[Proposition~1.9]{Hir}.
The original approach of Casselman and Shalika \cite{CS2} was to compute functionals on torus translates of the Casselman basis element corresponding to the longest Weyl element. A somewhat more versatile approach has been introduced by Hironaka, namely to compute the projection of the functional, say, the Shalika functional, onto the Iwahori invariant subspace.
So, our first step is to extend some of the ideas from \cite{Hir} to the cover. Hironaka established a formula
for expressing a spherical function on certain spherical homogeneous spaces, which is applicable also in the absence of uniqueness results. We verify this formula for the double cover of $\GL_n$, see Lemma~\ref{lemma:analog of Hironaka theory} and Corollary~\ref{corollary:analog of Hironaka theory with phi I functions}. In fact, our proof is general and applicable to a large class of covering groups, e.g., the central extensions of unramified reductive $p$-adic groups by finite cyclic groups, studied by McNamara \cite{McNamara2,McNamara,McNamara3} (at present, with his assumption on the cardinality of the residue field). See the discussion in the end of \S~\ref{subsection:Hironaka theorem}.

Since the formula of \cite{Hir} does not depend on a uniqueness property, it may well be more adequate for covering groups, where multiplicity one results may fail (e.g., for the Whittaker model). Our analog of \cite[Proposition~1.9]{Hir} involves the verification of several results well known in the non-metaplectic setting. We greatly benefited from the work of Chinta and Offen \cite{COf}, explaining how to extend
the results of Casselman \cite{CS1} to ($r$-fold) covers of $\GL_n$.
We must point out that the study of spherical functions on spherical varieties is vigorously expanding. The aforementioned techniques (e.g., of \cite{Hir}) have been extended and generalized by Sakellaridis \cite{Yia2}. While his work does not include covering groups, his ideas are expected to apply to this context, albeit with some modifications.

As mentioned above, the Shalika model has global aspects and the same applies to its metaplectic analog.
In a global setting, as with the Whittaker model, one requires a (perhaps) stronger notion. An automorphic representation
is called globally generic if it admits a Whittaker functional given by a Fourier coefficient. Surprisingly enough, we prove:
\begin{theo}[see Theorem~\ref{theorem:Shalika model given by Fourier coefficient}]\label{theo:global Shalika model of Theta}
The global exceptional representation admits a metaplectic Shalika functional, given by a Fourier coefficient. 
\end{theo}
Using this, we prove that exceptional representations over Archimedean fields also admit a metaplectic Shalika model, which is then unique, because of Theorem~\ref{theo:uniqueness}. Theorem~\ref{theo:global Shalika model of Theta} also plays a key role in the definition of the zeta integral. Moreover, we study a family of Fourier coefficients of exceptional representations and use their properties to prove Theorem~\ref{theo:co-period result}.

Theorem~\ref{theo:global Shalika model of Theta} demonstrates a very special phenomenon, although not estranged from exceptional representations of other groups (see, e.g., \cite[\S~3.4.2]{me7}). In fact, it is intimately related to the fact that minimal, or small representations are supported on small unipotent orbits, in the sense that generic Fourier coefficients of their automorphic forms vanish on sufficiently large (with respect to the partial ordering) unipotent orbits. See \cite{Cr,CM,BFG,G2} for more details.
We mention that Shalika functionals written solely in terms of a unipotent integration already appeared in a work of Beineke and Bump \cite{BeD} in the non-metaplectic case, for a specific representation.

As mentioned above, the proof of meromorphic continuation of the global zeta integrals involves establishing meromorphic continuation of the local integrals. Over Archimedean fields, such a result is usually difficult to obtain, the analysis must be performed carefully, due to the involvement of topological vector spaces. To achieve this, we develop an asymptotic expansion of metaplectic Shalika functions over Archimedean fields, using tools and techniques from Wallach \cite{Wal83,Wal88,Wal92} and Soudry \cite{Soudry3}, see \S~\ref{section:Archimedean asympt expansion Shalika}. This result may be of independent interest (it also applies to the non-metaplectic setting). Note that as a preliminary result, we write an asymptotic expansion of smooth matrix coefficients, on a whole Weyl chamber. This result applies to any real reductive group.

The rest of this work is organized as follows. In \S~\ref{section:preliminaries} we establish some notation and preliminaries, and recall several facts about the double cover. Section~\ref{section:Metaplectic Hironaka Theory} is dedicated to $p$-adic unramified theory. 
Section~\ref{section:Metaplectic unramified Casselman--Shalika formula}
begins with a general discussion of the metaplectic Shalika model, then
we compute the Casselman--Shalika formula.

In \S~\ref{section:Metap Shalika model for theta} we discuss exceptional representations. We briefly recall
their definition in \S~\ref{subsection:The exceptional representations}. Local results are contained in
\S~\ref{subsection:The metaplectic Shalika model of theta}, where 
we apply Theorem~\ref{theo:Casselman--Shalika} to unramified exceptional representations,
and Theorem~\ref{theo:uniqueness} to exceptional representations over $\R$. In \S~\ref{section:Fourier coefficients} we study global Fourier coefficients.
The local and global theory of the zeta integral is contained in \S~\ref{section:A Godement Jacquet integral}. Section~\ref{section:The co-period} is devoted to the co-period.

\addtocontents{toc}{\protect\setcounter{tocdepth}{1}}

\subsection*{Acknowledgments}
The first named author wishes to express his gratitude to
David Ginzburg, Erez Lapid and David Soudry for many encouraging conversations.
We thank Yiannis Sakellaridis for his interest in Theorem~\ref{theo:Casselman--Shalika} and useful remarks.
We thank Dmitry Gourevitch and Nolan Wallach for helpful correspondences. Lastly, the authors are grateful to Jim Cogdell and Robert Stanton for their kind encouragement and helpful remarks.

\addtocontents{toc}{\protect\setcounter{tocdepth}{1}}

\tableofcontents

\section{Preliminaries}\label{section:preliminaries}

\subsection{The groups}\label{subsection:Notation}
Let $G_n=\GL_n$. Fix the Borel subgroup $B_n=T_n\ltimes N_n$ of upper triangular matrices in $G_n$, where $T_n$ is the diagonal torus. For any $k\geq0$, let $Q_k=M_k\ltimes U_k$ be the standard maximal parabolic subgroup of $G_n$, whose Levi part $M_k$ is isomorphic to $G_k\times G_{n-k}$. For any parabolic subgroup $Q$, $\delta_{Q}$ denotes its modulus character. If $U$ is a unipotent radical of a standard parabolic subgroup, let $U^-$ denote the unipotent radical opposite to $U$.

Let $W$ be the Weyl group of $G_n$. The longest Weyl element is denoted by $w_0$ and the identity element by $e$.
For $w\in W$, let $\ell(w)$ be the length of $w$.
Let $\Sigma_{G_n}$ be the set of roots, $\Sigma_{G_n}^+$ be the subset of positive roots and
$\Delta_{G_n}$ be the set of simple roots compatible with our choice of $B_n$. For $\alpha\in\Delta_{G_n}$, denote the simple reflection along $\alpha$ by $s_{\alpha}$. Of course $W$ is generated by the set of simple reflections.

In general for any group $G$, $C_G$ denotes its center. If $Y<G$
and $d\in\Z$, $Y^d=\setof{y^d}{y\in Y}$. Also for $x,y\in G$, $\rconj{x}y=xyx^{-1}$,
$\rconj{x}Y=\setof{\rconj{x}y}{y\in Y}$. 

Throughout, all fields will be assumed to have characteristic different from $2$. If $F$ is a local
field, we usually denote $G_n=G_n(F)$. For a global field $F$, its ring of ad\`{e}les is denoted $\Adele$.
The vector space of $k\times k$ matrices over a ring $R$ (usually a local field) is denoted $R_{k\times k}$.

In the Archimedean case we use small gothic letters like $\mathfrak{g}_n$ and $\mathfrak{u}_k$ to denote the Lie algebras of the corresponding groups $G_n$ and $U_k$. For $F=\C$ we regard the complex Lie algebras as real ones. Write $\mathcal{U}(\mathfrak{g}_n)$ for the universal enveloping algebra of the complexification of $\mathfrak{g}_n$ and $\mathcal{Z}(\mathfrak{g}_n)$ for its center.

\subsection{The local metaplectic cover}\label{subsection:The metaplectic cover G_n}
Let $F$ be a local field, $(\blank,\blank)_2$ be the Hilbert symbol of order $2$ of $F$, and $\mu_2=\{-1,1\}$. Put $G_n=G_n(F)$.
Kazhdan and Patterson \cite{KP} constructed the double cover $\widetilde{G}_n$ of $G_n$, using the double cover of
$\mathrm{SL}_{n+1}$ of Matsumoto \cite{Mats} and the embedding $g\mapsto\mathrm{diag}(g,\det{g}^{-1})$ of
$G_n$ in $\mathrm{SL}_{n+1}$.

Given a $2$-cocycle $\sigma:G_n\times G_n\rightarrow\mu_2$, let $\widetilde{G}_n$ be the associated central extension of $G_n$ by $\mu_2$, $p:\widetilde{G}_n\rightarrow G_n$ be the natural projection and
$\mathfrak{s}:G_n\rightarrow\widetilde{G}_n$ be a section such that $p(\mathfrak{s}(g))=g$ and $\sigma(g,g')=\mathfrak{s}(g)\mathfrak{s}(g')\mathfrak{s}(gg')^{-1}$ (we also require $\mathfrak{s}$ to take the identity element of $G_n$ to the identity element of $\widetilde{G}_n$). For any subset $X\subset G_n$, put $\widetilde{X}=p^{-1}(X)$. 

We use the block-compatible cocycle $\sigma=\sigma_n$ of Banks, Levi and Sepanski \cite{BLS}.
The block-compatibility formula is given by
\begin{align}\label{eq:block-compatibility}
\sigma(\mathrm{diag}(a,b),\mathrm{diag}(a',b'))=\sigma_{k}(a,a')\sigma_{n-k}(b,b')(\det{a},\det{b'})_2,\qquad (a,b),(a',b')\in M_k.
\end{align}
This cocycle is trivial when $n=1$, and for $n=2$ coincides with the cocycle of Kubota \cite{Kubota}. Specifically if
$g=\left(\begin{smallmatrix}a&b\\c&d\end{smallmatrix}\right)\in G_2$, put $\mathbf{x}(g)=c$ if $c\ne0$, otherwise $\mathbf{x}(g)=d$, then
\begin{align}\label{eq:Kubota formula}
\sigma(g,g')=\Big(\frac{\mathbf{x}(gg')}{\mathbf{x}(g)},\frac{\mathbf{x}(gg')}{\mathbf{x}(g')\det g}\Big)_2.
\end{align}

The cocycle $\sigma$ and its twist
\begin{align*}
\sigma^{(1)}(g,g')=(\det g,\det g')_2\sigma(g,g')
\end{align*}
exhaust the nontrivial cohomology classes. All of our results in this work hold for both $\sigma$ and $\sigma^{(1)}$. In fact, all of
the computations will involve $g$ and $g'$ such that $(\det g,\det g')_2=1$ (e.g., $\det g\in F^{*2}$).

We mention several properties of $\sigma$ (and $\sigma^{(1)}$), that will be applied throughout.
For any $g\in G_n$ and $v\in N_n$, $\mathfrak{s}(gv)=\mathfrak{s}(g)\mathfrak{s}(v)$ and
$\mathfrak{s}(vg)=\mathfrak{s}(v)\mathfrak{s}(g)$. If $\rconj{g}v\in N_n$, then
$\mathfrak{s}(\rconj{g}v)=\mathfrak{s}(g)\mathfrak{s}(v)\mathfrak{s}(g)^{-1}$. Consequently if
$\mathfrak{s}$ is a splitting of $M<G_n$ (i.e., a homomorphism), $U<N_n$ and $MU=M\ltimes U$, $\mathfrak{s}$ is a splitting of $MU$. Note that $C_{\widetilde{G}_n}=\widetilde{C}_{G_n}^{\mathe}$, where $\mathe=1$ if $n$ is odd, otherwise $\mathe=2$.

For computations we fix a concrete set of elements $\mathfrak{W}\subset G_n$ as in \cite{BLS}, which is in bijection with
$W$. For $\alpha\in\Sigma_{G_n}$, let $w_{\alpha}$ be the image of $\left(\begin{smallmatrix}&-1\\1\end{smallmatrix}\right)$ in the subgroup of $G_n$ corresponding to $\alpha$. Define
\begin{align*}
\mathfrak{W}=\setof{w_{\alpha_1}\cdot\ldots \cdot w_{\alpha_{\ell(w)}}}{w\in W,w=s_{\alpha_1}\cdot\ldots\cdot s_{\alpha_{\ell(w)}}}.
\end{align*}

\subsection{The global metaplectic cover}\label{subsection:The global metaplectic cover}
Let $F$ be a global field. For a finite place $v$ of $F$, let $\mathcal{O}_v$ denote the ring of integers of $F_v$ and $\varpi_v$ be a generator of the maximal ideal. The group
$G_{n}(\Adele)$ is the restricted direct product $\prod'_{\nu}G_{n}(F_{\nu})$ with respect to
the compact subgroups $\{K_{v}\}_{v}$, where $K_{v}=G_n(\mathcal{O}_v)$ for almost all $v$, and we denote
$K=\prod_vK_v$. The global double cover $\widetilde{G}_n(\Adele)$ was constructed in \cite[\S~0.2]{KP} (see also \cite{FK,Tk}). 
For all $v<\infty$ such that $q_v=|\lmodulo{\varpi_v\mathcal{O}_v}{\mathcal{O}_v}|$ is odd and $q_v>3$, there are canonical splittings $\kappa_v$ of $K_v$ (\cite[pp.~54-56]{Moore}, see also \cite[Proposition~0.1.3]{KP}). We extend them to sections of $G_n(F_v)$. Also set $\mu_2^{\times}=\setof{(\epsilon_{v})_{v}\in \prod'_{v}\widetilde{G}_n(F_{v})}{ \epsilon_{v}\in\mu_2,\prod_{v}{\epsilon_{v}=1}}$.
Then $\widetilde{G}_n(\Adele)=\lmodulo{\mu_2^{\times}}{\prod'_{\nu}\widetilde{G}_{n}(F_{\nu})}$, where the restricted direct product is taken with respect to $\{\kappa_v(K_{v})\}_{v}$.
There is a global function $\mathfrak{s}=\prod_v\mathfrak{s}_v$, which is not defined on $G_n(\Adele)$, but is a splitting of $G_n(F)$ and $N_n(\Adele)$ and is well defined on $T_n(\Adele)$.
We identify $G_n(F)$ and $N_n(\Adele)$ with their image under $\mathfrak{s}$ in $\widetilde{G}_n(\Adele)$.

\subsection{Local representations}\label{subsection:representations}
Let $\Alg{G}$ be the category of complex and smooth representations of $G$, if
$G$ is an $l$-group (in the sense of \cite{BZ1}), or smooth admissible Fr\'{e}chet representations
of moderate growth if $G$ is a real reductive group. 

For $\pi\in\Alg{G}$ we denote by $\pi^{\vee}$ the contragredient representation of $\pi$. If $V$ is the space of $\pi$,
$V^\vee$ will denote the space of $\pi^{\vee}$. For a real reductive group $G$, $\pi^\vee$ is the smooth admissible Fr\'{e}chet globalization of moderate growth, of the $K$-finite vectors in the dual representation. Then a matrix coefficient of $\pi$ is a function of the form $f(g)=\xi^\vee(\pi(g)\xi)$ where $\xi\in V$ and $\xi^\vee\in V^\vee$.

Regular induction is denoted $\Ind$, and $\ind$ is the compact induction. Induction is not normalized. Over an Archimedean field induction denotes smooth induction.


For an $l$-space or an abelian Lie group $X$, let $\mathcal{S}(X)$ be the space of Schwartz--Bruhat functions on $X$. 

Let $G$ be an $l$-group and $\pi\in\Alg{G}$. Let $U<G$ be a closed subgroup, exhausted by its compact subgroups (here $U$ will be a unipotent subgroup of $G_n$), $\psi$ be a character of $U$, and $M<G$ be the normalizer of $U$ and stabilizer of $\psi$. The Jacquet module of $\pi$ with respect to $U$ and $\psi$ is denoted $\pi_{U,\psi}$. It is a representation of $M$. The action is not normalized.
The kernel of the surjection $\pi\rightarrow\pi_{U,\psi}$ is denoted
$\pi(U,\psi)$, and occasionally we write the elements of $\pi_{U,\psi}$ in the form $\xi+\pi(U,\psi)$ ($\xi$ in the space of $\pi$). When $\psi=1$, we simply write $\pi_U$ and $\pi(U)$.

Let $\pi\in\Alg{G_n}$ be irreducible and $Q=M\ltimes U$ be a parabolic subgroup of $G_n$.
Assume that the underlying field is $p$-adic.
The normalized exponents of $\pi$ along $Q$ are the central characters of the irreducible constituents of
$\delta_Q^{-1/2}\pi_U$. For example if $\pi$ is an irreducible representation with a unitary central character, then it is square integrable if and only if its normalized exponents along any standard parabolic subgroup lie in the open cone spanned by the positive roots in $M$.

\subsection{Asymptotic expansions over Archimedean fields}\label{subsection:asymptotic expansion}
We derive an asymptotic expansion for matrix coefficients $f(g)=\xi^\vee(\pi(g)\xi)$ of $G_k$. The key ideas are due to Wallach (see \cite[Theorem 7.2]{Wal83}, \cite[Chapters 4.3, 4.4]{Wal88} and \cite[Chapter 15.2]{Wal92}), we adapt them to our setting (see also \cite[Theorem 1 in \S~4]{Soudry3}). The novelty is that we derive an expansion into an exponential power series, not only on a one-parameter group but on a whole Weyl chamber, depending continuously on $\xi$ and $\xi^\vee$. We carry out the details for $G_k$, but the proof extends to any real reductive group.

Let $F$ be an Archimedean field.
Recall the maximal parabolic subgroups $Q_\ell=M_\ell\ltimes U_\ell$ of $G_k$. The non-compact central factor of $M_\ell$ is given by $A_\ell=\exp(\mathfrak{a}_\ell)$ with $\mathfrak{a}_\ell=\R I_k+\R H_\ell$ and $H_\ell=\diag(I_\ell,0_{k-\ell})$. For $x=(x_1,\ldots,x_k)\in\R^k$ put $a_x=\exp(-\sum_{\ell=1}^kx_\ell H_\ell)$.

For a representation $\pi\in\Alg{G_k}$ on a space $V$, we write $E_1(Q_\ell,\pi)\subset(\mathfrak{a}_\ell)_\C^*$ for the finite set of generalized $\mathfrak{a}_\ell$-weights in $V/\mathfrak{u}_\ell V$. Put
\begin{align}\notag
&E_1^{(\ell)}(\pi)=\{\mu(H_\ell):\mu\in E_1(Q_\ell,\pi)\}\subset\C,\\\notag
&E^{(\ell)}(\pi)=E_1^{(\ell)}(\pi)+\N,\\
& \Lambda_{\pi,\ell} = \min\{\Re z:z\in E^{(\ell)}(\pi)\} = \min\{\Re z:z\in E_1^{(\ell)}(\pi)\},\label{eq:DefinitionLambdaPiSmoothMatrixCoefficient}\\
& \Lambda_\pi=(\Lambda_{\pi,1},\ldots,\Lambda_{\pi,k}).\notag
\end{align}
Here $\N=\{0,1,2,\ldots\}$.
By \cite[Theorem 15.2.4]{Wal92} there exist continuous seminorms $q_1$, $q_2$ on $V$, $V^\vee$ (resp.) and $d\geq0$ such that
\begin{align}
 |\xi^{\vee}(\pi(a_x)\xi)| \leq q_1(\xi)q_2(\xi^{\vee})(1+|x|)^de^{-\Lambda_\pi\cdot x}, \qquad \forall x\in\R_{>0}^k,\xi\in V,\xi^{\vee}\in V^\vee.\label{eq:BoundSmoothMatrixCoefficients}
\end{align}

For a subset $I=\{i_1,\ldots,i_r\}$ of $\{1,\ldots,k\}$ with $i_1<\ldots<i_r$, let $x_I=(x_{i_1},\ldots,x_{i_r})$ and $\overline{I}=\{1,\ldots,k\}\setdifference I$.

\begin{theorem}\label{thm:AsymptoticExpansionSmoothMatrixCoefficientsArchimedean}
Let $\pi\in\Alg{G_k}$ be an irreducible representation on a space $V$. For any $D=(D_1,\ldots,D_k)\in\R^k$ there exist finite subsets $C_\ell\subset E^{(\ell)}(\pi)$ and functions $p_{I,z}:\R^k\times V\times V^\vee\to\C$ for $I\subset\{1,\ldots,k\}$, $z\in\prod_{\ell\in I}C_\ell$, such that the following holds.
\begin{enumerate}[leftmargin=*]
\item For all $\xi\in V$, $\xi^{\vee}\in V^\vee$ and $x\in\R^k$,
\begin{align}
 \xi^{\vee}(\pi(a_x)\xi) = \sum_{I,z}p_{I,z}(x;\xi,\xi^{\vee})\cdot e^{-z\cdot x_I},\label{eq:ExpansionSmoothMatrixCoefficients}
\end{align}
where the (finite) summation is over all $I\subset\{1,\ldots,k\}$ and $z\in\prod_{\ell\in I}C_\ell$.
\item Each function $p_{I,z}(x;\xi,\xi^{\vee})$ is polynomial in $x_I$:
\begin{align}
 p_{I,z}(x;\xi,\xi^{\vee}) = \sum_{\alpha\in\N^I}c_\alpha(x_{\overline{I}};\xi,\xi^{\vee})x_I^\alpha,\label{eq:PolynomialCoefficientsSmoothMatrixCoefficients}
\end{align}
where the coefficients $c_\alpha(x_{\overline{I}};\xi,\xi^{\vee})$ are smooth in $x_{\overline{I}}$, linear in $\xi$ and $\xi^{\vee}$, and satisfy bounds of the form
\begin{align}
 \hspace{.5cm} |c_\alpha(x_{\overline{I}};\xi,\xi^{\vee})| \leq q_1(\xi)q_2(\xi^{\vee})(1+|x_{\overline{I}}|)^d e^{-D_{\overline{I}}\cdot x_{\overline{I}}}, \qquad \forall x_{\overline{I}}\in\R_{>0}^{\overline{I}},\xi\in V,\xi^{\vee}\in V^\vee,\label{eq:EstimateRemainderSmoothMatrixCoefficients}
\end{align}
with $d\geq0$ and continuous seminorms $q_1$, $q_2$ on $V$, $V^\vee$ (resp.).
\end{enumerate}
\end{theorem}

\begin{proof}
It is enough to show the above statement for arbitrarily large $D_\ell$, $\ell=1,\ldots,k$. We achieve this by first treating the case $D=\Lambda_\pi$, then proving that any $D$ can be replaced by $D+e_\ell$ for some $\ell=1,\ldots,k$. Repeatedly applying this argument yields the claim.

For $D=\Lambda_\pi$ one may choose the sum in \eqref{eq:ExpansionSmoothMatrixCoefficients} to consist of only one summand,
for the empty set $I$, then the statement follows from \eqref{eq:BoundSmoothMatrixCoefficients}. Now let $D\in\R^k$ be arbitrary and assume \eqref{eq:ExpansionSmoothMatrixCoefficients}, \eqref{eq:PolynomialCoefficientsSmoothMatrixCoefficients} and \eqref{eq:EstimateRemainderSmoothMatrixCoefficients}. We fix $1\leq\ell\leq k$ and prove that the same is true for $D$ replaced by $D+e_\ell$. For brevity, set $H=H_\ell$, $t=x_\ell$ and $x'=(x_1,\ldots,x_{\ell-1},0,x_{\ell+1},\ldots,x_k)$.

Let $\{X_p\}_p$ be a basis of $\mathfrak{u}_\ell$. In \cite[\S~15.2.4]{Wal92} Wallach constructed $E_1=1,E_2,\ldots,E_r\in\mathcal{U}(\mathfrak{g}_k)$ as well as finite sets $\{Z_{i,j}\}_{i,j}\subset\mathcal{Z}(\mathfrak{g}_k)$ and $\{U_{p,i}\}_{p,i}\subset\mathcal{U}(\mathfrak{g}_k)$ such that
\begin{align}
 H_\ell E_i = \sum_j Z_{i,j}E_j + \sum_p X_pU_{p,i}, \qquad \forall i=1,\ldots,r.\label{eq:AbsoluteActionOfHOnIdentityInU(g)}
\end{align}
Since $\pi$ is irreducible, the elements $Z_{i,j}$ act by scalars $B_{i,j}=\pi(Z_{i,j})\in\C$ and we form the matrix $B=(B_{i,j})_{i,j}$. Note that the eigenvalues of $B$ are contained in the finite set $E_1^{(\ell)}(\pi)=\{\zeta_1,\ldots,\zeta_N\}$ (the generalized eigenvalues of $\pi(H_\ell)$ on the Jacquet module $V/\mathfrak{u}_\ell V$). Define
\begin{align*}
 F(t,x';\xi,\xi^{\vee}) &= \left(\begin{array}{c}\xi^{\vee}(\pi(e^{-tH})\pi(a_{x'})\pi(E_1)\xi)\\\vdots\\\xi^{\vee}(\pi(e^{-tH})\pi(a_{x'})\pi(E_r)\xi)\end{array}\right),\\
 G(t,x';\xi,\xi^{\vee}) &= \sum_p\left(\begin{array}{c}\xi^{\vee}(\pi(e^{-tH})\pi(a_{x'})\pi(X_p)\pi(U_{p,1})\xi)\\\vdots\\\xi^{\vee}(\pi(e^{-tH})\pi(a_{x'})\pi(X_p)\pi(U_{p,r})\xi)\end{array}\right)
\end{align*}
The first coordinate of $F(t,x';\xi,\xi^{\vee})$ is equal to $\xi^{\vee}(\pi(a_x)\xi)$ and the function $F(t,x';\xi,\xi^{\vee})$ satisfies the following differential equation:
$$ \frac{d}{dt}F(t,x';\xi,\xi^{\vee}) = -BF(t,x';\xi,\xi^{\vee}) - G(t,x';\xi,\xi^{\vee}). $$
The solution of this system is given by
\begin{align}
 F(t,x';\xi,\xi^{\vee}) = e^{-tB}F(0,x';\xi,\xi^{\vee}) - e^{-tB}\int_0^t e^{\tau B}G(\tau,x';\xi,\xi^{\vee})\,d\tau.\label{eq:SolutionMatrixCoefficientsODE}
\end{align}
We claim that each coordinate on the \rhs\ is a finite sum of terms of the form
\begin{align}
 c(x_{\overline{I}};\xi,\xi^{\vee})x_I^\alpha e^{-z\cdot x_I}\label{eq:PolynomialCoefficientsSmoothMatrixCoefficientsForNewD}
\end{align}
with $I\subset\{1,\ldots,k\}$, $\alpha\in\N^I$ and $z\in\prod_{i\in I}E^{(i)}(\pi)$, and a function $c(x_{\overline{I}};\xi,\xi^{\vee})$ which is smooth in $x_{\overline{I}}$, linear in $\xi$ and $\xi^{\vee}$ and satisfies
\begin{align}
 |c(x_{\overline{I}};\xi,\xi^{\vee})| \leq q_1(\xi)q_2(\xi^{\vee})(1+|x_{\overline{I}}|)^de^{-(D+e_\ell)_{\overline{I}}\cdot x_{\overline{I}}}, \qquad \forall x_{\overline{I}}\in\R_{>0}^{\overline{I}},\xi\in V,\xi^{\vee}\in V^\vee,\label{eq:EstimateRemainderSmoothMatrixCoefficientsForNewD}
\end{align}
for some continuous seminorms $q_1,q_2$ and $d\geq0$. This will finish the proof.

Note that
\begin{align}
 e^{\tau B} = \sum_{j=1}^NP_j(\tau)e^{\zeta_j\tau}\label{eq:ExponentialOfB}
\end{align}
for some matrix-valued polynomials $P_j(\tau)$. Let us treat the two summands in \eqref{eq:SolutionMatrixCoefficientsODE} separately.

\begin{enumerate}[leftmargin=*]
\item The coordinates of $F(0,x';\xi,\xi^{\vee})$ are of the form $\xi^{\vee}(\pi(a_{x'})\pi(E_i)\xi)$, and by \eqref{eq:ExponentialOfB} multiplying with $e^{-tB}$ yields sums of terms of the form
\begin{align}
 t^me^{-\zeta_jt}\cdot\xi^{\vee}(\pi(a_{x'})\pi(E_i)\xi),\label{eq:CoordinatesFirstTermSolutionODE}
\end{align}
where $m\geq0$. The map $\xi\mapsto\pi(E_i)\xi$ is continuous, hence for every continuous seminorm $q_1'$ on $V$ there exists a continuous seminorm $q_1$ such that $q_1'(\pi(E_i)\xi)\leq q_1(\xi)$ for all $\xi\in V$. Applying the induction assumption to the matrix coefficients $\xi^{\vee}(\pi(a_{x'})\pi(E_i)\xi)$ then shows that the terms \eqref{eq:CoordinatesFirstTermSolutionODE} are finite sums of the form
\begin{align}
 t^me^{-\zeta_jt}\cdot c(x'_{\overline{J}};\xi,\xi^{\vee})(x'_J)^\alpha e^{-z'\cdot x'_J}\label{eq:Coordinates2FirstTermSolutionODE}
\end{align}
where $J\subset\{1,\ldots,k\}\setdifference\{\ell\}$, $\overline{J}=\{1,\ldots,k\}\setdifference(J\cup\{\ell\})$, $\alpha\in\N^J$, $z'\in\prod_{i\in J}E^{(i)}(\pi)$ and $c(x'_{\overline{J}};\xi,\xi^{\vee})$ is smooth in $x'_{\overline{J}}$, linear in $\xi$ and $\xi^{\vee}$, and satisfies
\begin{align*}
|c(x'_{\overline{J}};\xi,\xi^{\vee})| \leq q_1(\xi)q_2(\xi^{\vee})(1+|x'_{\overline{J}}|)^de^{-D_{\overline{J}}\cdot x'_{\overline{J}}}, \qquad \forall x'_{\overline{J}}\in\R_{>0}^{\overline{J}},\xi\in V,\xi^{\vee}\in V^\vee,
\end{align*}
for some continuous seminorms $q_1,q_2$ and $d\geq0$. Let $I=J\cup\{\ell\}$, $x=x'+te_\ell$ and $z=(\zeta_j,z')\in E^{(\ell)}(\pi)\times\prod_{i\in J}E^{(i)}(\pi)$, then \eqref{eq:Coordinates2FirstTermSolutionODE} is clearly of the form \eqref{eq:PolynomialCoefficientsSmoothMatrixCoefficientsForNewD} and $c(x_{\overline{I}};\xi,\xi^{\vee})$ satisfies \eqref{eq:EstimateRemainderSmoothMatrixCoefficientsForNewD} since $\ell\notin \overline{I}$.

\item The coordinates of $G(t,x';\xi,\xi^{\vee})$ are sums of the terms
\begin{align*}
\xi^{\vee}(\pi(e^{-tH})\pi(a_{x'})\pi(X_p)\pi(U_{p,i})\xi) = e^{-t}\xi^{\vee}(\pi(a_{x'})\pi(X_p)\pi(e^{-tH})\pi(U_{p,i})\xi).
\end{align*}
Note that $\pi(a_{x'})\pi(X_p)$ is a sum of terms of the form $e^{-\beta\cdot x'}\pi(X)\pi(a_{x'})$ with $X\in\mathfrak{u}_\ell$ and $\beta\in\N^{k-1}$. Hence, each coordinate of $G(t,x';\xi,\xi^{\vee})$ is a sum of terms of the form
\begin{align*}
e^{-t}e^{-\beta\cdot x'}\cdot(\pi^\vee(X)\xi^{\vee})\Big(\pi(e^{-tH})\pi(a_{x'})\pi(U)\xi\Big), \qquad U,X\in\mathcal{U}(\mathfrak{g}_k),\beta\in\N^{k-1}.
\end{align*}
Since $\xi\mapsto\pi(U)\xi$ and $\xi^{\vee}\mapsto\pi^\vee(X)\xi^{\vee}$ are continuous linear operators we can apply the induction assumption to the matrix coefficients $(\pi^\vee(X)\xi^{\vee})(\pi(e^{-tH})\pi(a_{x'})\pi(U)\xi)$ and find that the coordinates of $G(t,x';\xi,\xi^{\vee})$ are linear combinations of terms of the form
\begin{align*}
c(x'_{\overline{J}};\xi,\xi^{\vee})(x'_J)^\alpha e^{-z'\cdot x'_J}\cdot t^me^{-(z_\ell+1)\cdot t} \qquad \text{and} \qquad c(t,x'_{\overline{J}};\xi,\xi^{\vee})(x'_J)^\alpha e^{-z'\cdot x'_J}\cdot e^{-t}
\end{align*}
with $J\subset\{1,\ldots,k\}\setdifference\{\ell\}$, $\overline{J}=\{1,\ldots,k\}\setdifference(J\cup\{\ell\})$, $\alpha\in\N^J$, $m\in\N$, $z'\in\prod_{i\in J}E^{(i)}(\pi)$, $z_\ell\in E^{(\ell)}(\pi)$, where the coefficients $c(x'_{\overline{J}};\xi,\xi^{\vee})$ and $c(t,x'_{\overline{J}};\xi,\xi^{\vee})$ are smooth in $x'_{\overline{J}}$ and also $t$ in the second case, linear in $\xi$ and $\xi^{\vee}$, and satisfy estimates of the form
\begin{align}
 |c(x'_{\overline{J}};\xi,\xi^{\vee})| &\leq q_1(\xi)q_2(\xi^{\vee})(1+|x'_{\overline{J}}|)^de^{-D_{\overline{J}}\cdot x'_{\overline{J}}}\label{eq:EstimateRemainderSmoothMatrixCoefficientsForGcoordinates1},\\
 |c(t,x'_{\overline{J}};\xi,\xi^{\vee})| &\leq q_1(\xi)q_2(\xi^{\vee})(1+|x'_{\overline{J}}|)^de^{-D_{\overline{J}}\cdot x'_{\overline{J}}}(1+t)^de^{-D_\ell t}\label{eq:EstimateRemainderSmoothMatrixCoefficientsForGcoordinates2}
\end{align}
for all $x'_{\overline{J}}\in\R_{>0}^{\overline{J}}$, $t\geq0$ and $\xi\in V$, $\xi^{\vee}\in V^\vee$. Here we have attributed the factors $e^{-\beta_ix_i}$ either to $e^{-z'\cdot x'_J}$ in case $i\in J$, or to $c(x'_{\overline{J}};\xi,\xi^{\vee})$ and $c(t,x'_{\overline{J}};\xi,\xi^{\vee})$ in case $i\notin J$. Then by \eqref{eq:ExponentialOfB} the components of $e^{(\tau-t)B}G(\tau,x';\xi,\xi^{\vee})$ are linear combinations of terms of the form
\begin{align*}
 I_1 &= t^ne^{-\zeta_jt}\cdot c(x'_{\overline{J}};\xi,\xi^{\vee})(x'_J)^\alpha e^{-z'\cdot x'_J}\cdot\tau^me^{(\zeta_j-z_\ell-1)\cdot\tau} \qquad \mbox{and}\\
 I_2 &= t^ne^{-\zeta_jt}\cdot c(\tau,x'_{\overline{J}};\xi,\xi^{\vee})(x'_J)^\alpha e^{-z'\cdot x'_J}\cdot\tau^me^{(\zeta_j-1)\tau}.
\end{align*}
We treat the integration over $\tau$ for these two types separately.

For $I_1$ we know that
\begin{align*}
\int_0^t\tau^me^{(\zeta_j-z_\ell-1)\cdot\tau}\,d\tau = p(t)e^{(\zeta_j-z_\ell-1)t}+C
\end{align*}
for some polynomial $p(t)$ and a constant $C$. Hence
\begin{align*}
 \int_0^tI_1\,d\tau &= p(t)t^ne^{-(z_\ell+1)t}c(x'_{\overline{J}};\xi,\xi^{\vee})(x'_J)^\alpha e^{-z'\cdot x'_J}\\&
\quad + Ct^ne^{-\zeta_jt}c(x'_{\overline{J}};\xi,\xi^{\vee})(x'_J)^\alpha e^{-z'\cdot x'_J}.
\end{align*}
Both summands are linear combinations of terms of the form \eqref{eq:PolynomialCoefficientsSmoothMatrixCoefficientsForNewD} for $I=J\cup\{\ell\}$ since both $z_\ell+1,\zeta_j\in E^{(\ell)}(\pi)$, and the corresponding estimates \eqref{eq:EstimateRemainderSmoothMatrixCoefficientsForNewD} follow from \eqref{eq:EstimateRemainderSmoothMatrixCoefficientsForGcoordinates1} since $\ell\notin \overline{I}$.

For $I_2$ we first assume $D_\ell>\Re\zeta_j-1$. Then rewrite
$$ \int_0^tI_2\,d\tau = \int_0^\infty I_2\,d\tau - \int_t^\infty I_2\,d\tau. $$
The first integral
$$ \int_0^\infty I_2\,d\tau = t^ne^{-\zeta_jt}(x'_J)^\alpha e^{-z'\cdot x'_J} \int_0^\infty c(\tau,x'_{\overline{J}};\xi,\xi^{\vee})\tau^me^{(\zeta_j-1)\tau}\,d\tau $$
is of the form \eqref{eq:PolynomialCoefficientsSmoothMatrixCoefficientsForNewD} for $I=J\cup\{\ell\}$. The estimate \eqref{eq:EstimateRemainderSmoothMatrixCoefficientsForNewD} follows from \eqref{eq:EstimateRemainderSmoothMatrixCoefficientsForGcoordinates2} since
$|c(\tau,x'_{\overline{J}};\xi,\xi^{\vee})\tau^me^{(\zeta_j-1)\tau}|$ is bounded by
\begin{align*}
 q_1(\xi)q_2(\xi^{\vee})(1+|x'_{\overline{J}}|)^de^{-D_{\overline{J}}\cdot x'_{\overline{J}}}\cdot \tau^m(1+\tau)^de^{(\Re\zeta_j-D_\ell-1)\tau},
\end{align*}
which is integrable over $\R_{>0}$ since $D_\ell>\Re\zeta_j-1$.

The second integral
$$ \int_t^\infty I_2\,d\tau = (x'_J)^\alpha e^{-z'\cdot x'_J}\cdot t^ne^{-\zeta_jt}\int_t^\infty c(\tau,x'_{\overline{J}};\xi,\xi^{\vee})\tau^me^{(\zeta_j-1)\tau}\,d\tau $$
is of the form \eqref{eq:PolynomialCoefficientsSmoothMatrixCoefficientsForNewD} for $I=J$. The integral is by the previous estimation smooth in $t$ and $x'_{\overline{J}}$, and the estimate \eqref{eq:EstimateRemainderSmoothMatrixCoefficientsForNewD} follows from
\begin{align*}
 & \left|t^ne^{-\zeta_jt}\int_t^\infty c(\tau,x'_{\overline{J}};\xi,\xi^{\vee})\tau^me^{(\zeta_j-1)\tau}\,d\tau\right|\\
 &\leq{} q_1(\xi)q_2(\xi^{\vee})(1+|x'_{\overline{J}}|)^de^{-D_{\overline{J}}\cdot x'_{\overline{J}}}\cdot t^ne^{-\Re\zeta_jt}\int_t^\infty\tau^m(1+\tau)^de^{(\Re\zeta_j-D_\ell-1)\tau}\,d\tau\\
 &\leq{} \text{constant}\times q_1(\xi)q_2(\xi^{\vee})(1+|x'_{\overline{J}}|)^de^{-D_{\overline{J}}\cdot x'_{\overline{J}}}\cdot (1+t)^{m+d+n}e^{-(D_\ell+1)t}.
\end{align*}

It remains to consider $I_2$ when $D_\ell\leq\Re\zeta_j-1$. Then
$$ \int_0^tI_2\,d\tau = (x'_J)^\alpha e^{-z'\cdot x'_J}\cdot t^ne^{-\zeta_jt}\int_0^t c(\tau,x'_{\overline{J}};\xi,\xi^{\vee})\tau^me^{(\zeta_j-1)\tau}\,d\tau $$
is of the form \eqref{eq:PolynomialCoefficientsSmoothMatrixCoefficientsForNewD} for $I=J$. The estimate \eqref{eq:EstimateRemainderSmoothMatrixCoefficientsForNewD} now follows from \eqref{eq:EstimateRemainderSmoothMatrixCoefficientsForGcoordinates2} since
\begin{align*}
 & \left|t^ne^{-\zeta_jt}\int_0^tc(\tau,x'_{\overline{J}};\xi,\xi^{\vee})\tau^me^{(\zeta_j-1)\tau}\,d\tau\right|\\
 &\leq{} q_1(\xi)q_2(\xi^{\vee})(1+|x'_{\overline{J}}|)^de^{-D_{\overline{J}}\cdot x'_{\overline{J}}}\cdot t^ne^{-(\Re\zeta_j)t}\int_0^t\tau^m(1+\tau)^de^{(\Re\zeta_j-D_\ell-1)\tau}\,d\tau\\
 &\leq{} q_1(\xi)q_2(\xi^{\vee})(1+|x'_{\overline{J}}|)^de^{-D_{\overline{J}}\cdot x'_{\overline{J}}}\cdot t^{n+1}(1+t)^de^{-(D_\ell+1)t}.
\end{align*}
\end{enumerate}
The proof is complete.
\end{proof}

\subsection{The Weil symbol}
Let $F$ be a local field.
We usually denote by $\psi$ a nontrivial additive character of $F$. Then $\gamma_{\psi}$ is the normalized Weil factor (\cite[\S~14]{We}), $\gamma_{\psi}(\blank)^4=1$. For $a\in F^*$, $\psi_a$ denotes the character $\psi_a(x)=\psi(ax)$ and
$\gamma_{\psi,a}=\gamma_{\psi_a}$. We have the following formulas (see the appendix of \cite{Rao},
$\gamma_{\psi}(a)$ is $\gamma_{F}(a,\psi)$ in his notation):
\begin{align}\label{eq:Weil factor identities}
\gamma_{\psi}(xy)=\gamma_{\psi}(x)\gamma_{\psi}(y)(x,y)_2 ,\quad \gamma_{\psi}(x^2)=1,\quad \gamma_{\psi}^{-1}=\gamma_{\psi^{-1}},\quad\gamma_{\psi,a}(x)=(a,x)_2\gamma_{\psi}(x).
\end{align}

\section{Metaplectic Hironaka Theory}\label{section:Metaplectic Hironaka Theory}

\subsection{Preliminaries}\label{subsection:unramified representations}
In this section $F$ is a local non-Archimedean field and $G_n=G_n(F)$. Then $\mathcal{O}$, $\varpi$, $K$ and $q$
are defined as in \S~\ref{subsection:The global metaplectic cover}. Assume $q>3$ and $|2|=1$ in the field. We normalize the Haar measure $dg$ on $G_n$ such that $\mathrm{vol}(K)=1$.
Let $\kappa:K\rightarrow\widetilde{G}_n$ be the canonical splitting of $K$. It coincides with $\mathfrak{s}$ on
$T_n\cap K$, $\mathfrak{W}$ and $N_n\cap K$ ($\mathfrak{W}$ was defined
in \S~\ref{subsection:The metaplectic cover G_n}, and is a subset of $K$). We use this freely, e.g., it implies that $\mathfrak{s}$ is a splitting of $\mathfrak{W}$.

\subsection{Metaplectic review of Casselman's theory}\label{section:Casselman's basis and a result of Hironaka}
In this section we describe a metaplectic analog of several results of Casselman \cite{CS1}, needed in the following section to prove a result of Hironaka \cite{Hir}. We define an analog of Casselman's basis for the Iwahori fixed subspace of an unramified principal series. Following \cite{Hir}, we use this basis to deduce a general formula for functionals on this space. The bulk part of this section has already appeared in the works of Chinta and Offen \cite{COf} and McNamara \cite{McNamara}, specific references are given below.

In order to apply the ideas of \cite{CS1} we need several auxiliary results.
Let $\mathcal{I}<K$ be the Iwahori subgroup which is compatible with $B_n$. The section $\kappa$ is in particular a splitting of $\mathcal{I}$.
The Iwahori factorization of $\mathcal{I}$ reads $\mathcal{I}=\mathcal{I}^+\mathcal{I}^0\mathcal{I}^-$ where
$\mathcal{I}^+=\mathcal{I}\cap N_n$, $\mathcal{I}^0=\mathcal{I}\cap T_n$ and $\mathcal{I}^-=\mathcal{I}\cap N_n^-$.
Note that
\begin{align}\label{eq:reformulation of Iwahori}
\mathcal{I}=(\mathcal{I}\cap B_n)\mathcal{I}^-=(K\cap B_n)\mathcal{I}^-.
\end{align}

For a representation $\pi\in\mathrm{Alg}\widetilde{G}_n$ and a compact open subgroup $K_0<K$, the space of $\kappa(K_0)$-invariants is denoted $\pi^{K_0}$ and the projection on
$\pi^{K_0}$ is denoted $\mathscr{P}_{K_0}$. For any vector $\xi$ in the space of $\pi$,
\begin{align*}
\mathscr{P}_{K_0}(\xi)=\int_{K_0}\pi(\kappa(k))\xi\, dk,
\end{align*}
where the measure is normalized such that $\mathrm{vol}(K_0)=1$.

Let
\begin{align*}
T_n^-=\{t\in T_n:|\alpha(t)|\leq1,\quad\forall\alpha\in\Delta_{G_n}\}.
\end{align*}
The following is a version of Jacquet's First Lemma.
\begin{claim}\label{claim:Jacquet's First Lemma}
Let $\pi\in \mathrm{Alg}\widetilde{G}_n$ be admissible. For any $\xi$ in the space of $\pi^{\mathcal{I}^0\mathcal{I}^-}$ and
$t\in C_{\widetilde{T}_n}\cap\widetilde{T}_n^-$, 
\begin{align*}
&\mathscr{P}_{\mathcal{I}}(\pi(t)\xi)=\mathscr{P}_{\mathcal{I}^+}(\pi(t)\xi),\qquad
\mathscr{P}_{\mathcal{I}}(\pi(t)\xi)-\pi(t)\xi\in\pi(N_n).
\end{align*}
\end{claim}
\begin{proof}
Write the integral $\mathscr{P}_{\mathcal{I}}$ using the Iwahori decomposition $\mathcal{I}=\mathcal{I}^+\mathcal{I}^0\mathcal{I}^-$. Then
\begin{align*}
\mathscr{P}_{\mathcal{I}}(\pi(t)\xi)&=\int_{\mathcal{I}^+}\int_{\mathcal{I}^0}\int_{\mathcal{I}^-}\pi(\kappa(x)\kappa(y)\kappa(z)t)\xi\, dz\, dy \, dx\\
&=\int_{\mathcal{I}^+}\int_{\mathcal{I}^0}\int_{\mathcal{I}^-}\pi(\kappa(x)t\kappa(y)\ \rconj{t^{-1}}\kappa(z))\xi\, dz\, dy \, dx\\
&=\mathscr{P}_{\mathcal{I}^+}(\pi(t)\xi).
\end{align*}
Here we used the facts that
 $\rconj{t^{-1}}\kappa(\mathcal{I}^-)<\kappa(\mathcal{I}^-)$, $t$ commutes with $\kappa(\mathcal{I}^0)$ and $\xi\in\pi^{\mathcal{I}^0\mathcal{I}^-}$.
Then the Jacquet-Langlands lemma (see e.g., \cite[2.33]{BZ1}) implies that
$\mathscr{P}_{\mathcal{I}^+}(\pi(t)\xi)$ and $\pi(t)\xi$ have the same image in $\pi_{N_n}$.
\end{proof}

The following claim is a weak analog of \cite[Proposition~2.4]{CS1}.
\begin{claim}\label{claim:Iwahori invariants and Jacquet kernel}
Let $\pi\in\mathrm{Alg}{\widetilde{G}_n}$ be admissible. The representation $\pi^{\mathcal{I}}$ injects as a vector space into $\pi_{N_n}^{\mathcal{I}_0}$.
\end{claim}
\begin{proof}
We have to prove $\pi^{\mathcal{I}}\cap \pi(N_n)=0$.
Let $\xi$ belong to the space of $\pi^{\mathcal{I}}\cap \pi(N_n)$. We prove $\xi=0$.
According to the Jacquet-Langlands lemma (\cite[2.33]{BZ1}) there is a compact open subgroup $\mathcal{N}<N_n$ with
\begin{align*}
\int_{\mathcal{N}}\pi(v)\xi\, dv=0.
\end{align*}
Let $t\in C_{\widetilde{T}_n}\cap \widetilde{T}_n^-$ be such that $\mathcal{N}<\rconj{p(t)^{-1}}\mathcal{I}^+$. Then by Claim~\ref{claim:Jacquet's First Lemma},
\begin{align*}
\mathscr{P}_{\mathcal{I}}(\pi(t)\xi)&=\int_{\mathcal{I}^+}\pi(\kappa(x)t)\xi\, dx=
\pi(t)\int_{\mathcal{I}^+}\pi(\mathfrak{s}(\rconj{p(t)^{-1}}x))\xi\, dx=0,
\end{align*}
by our choice of $t$. However, by \cite[Proposition~3.1.4]{Savin4} the mapping
$\xi\mapsto\mathscr{P}_{\mathcal{I}}(\pi(t)\xi)$ on the finite-dimensional space $\pi^{\mathcal{I}}$ is invertible
(this mapping is essentially $\pi(T_{p(t)})$ in the notation of \textit{loc. cit.}).
\end{proof}
\begin{remark}\label{remark:analog of Casselman 2.4 is weaker}
The original claim in \cite{CS1} also included surjectivity. In the metaplectic case this does not hold in general, e.g.,
when $\pi$ is a genuine unramified principal series. 
\end{remark}

We recall the parametrization of genuine unramified principal series representations of $\widetilde{G}_n$.
Let $\widetilde{T}_{n,*}$ be the centralizer of $\widetilde{T}_n\cap\kappa(K)$ in $\widetilde{T}_n$,
it is a maximal abelian subgroup. Explicitly,
\begin{align*}
T_{n,*}=\{\diag(t_1z,\ldots,t_nz):t_i\in F^{*2}\mathcal{O}^*,z\in F^{*\mathe}\}
\end{align*}
($\mathe$ was defined in \S~\ref{subsection:The metaplectic cover G_n}). Assume
$\underline{s}\in\C^n$. Let $\gamma:F^*\rightarrow\C$ be a function satisfying $\gamma(zz')=\gamma(z)\gamma(z')\sigma(zI_n,z'I_n)$ for $z,z'\in F^{*\mathe}$, and
$\gamma|_{F^{*2}}=\gamma|_{\mathcal{O}^*}=1$. The genuine unramified character $\chi=\chi_{\underline{s},\gamma}$ of $\widetilde{T}_{n,*}$ is defined by
\begin{align*}
\chi(\epsilon\mathfrak{s}(\diag(t_1,\ldots,t_n))\mathfrak{s}(zI_n))=
\epsilon\gamma(z)|z|^{s_1+\ldots+s_n}\prod_{i=1}^{n}|t_i|^{s_i},\quad \forall \epsilon\in\mu_2, t_i\in F^{*2}\mathcal{O}^*,z\in F^{*\mathe}.
\end{align*}
Put $B_{n,*}=T_{n,*}N_n$, it is a closed subgroup of $G_n$, open and normal in $B_n$. Then
\begin{align*}
\mathrm{I}(\chi)=\Ind_{\widetilde{B}_{n,*}}^{\widetilde{G}_n}(\delta^{1/2}_{B_n}\chi)
\end{align*}
is a genuine unramified principal series representation, i.e., $\mathrm{I}(\chi)^{K}$ is one dimensional, and any such representation takes this form, where $s_i$ is unique modulo $\frac{\pi i}{\log q}\Z$ (as opposed to $\frac{2\pi i}{\log q}\Z$ in $G_n$). The dependency on $\gamma$ will not appear in the formulas and in fact when $n$ is even, $\gamma$ is irrelevant. For a general discussion of these representation for $r$-fold covers of $G_n$ see \cite{COf}.

We regard the elements in the space of $\mathrm{I}(\chi)$ as complex-valued functions. 
The unramified normalized element of $\mathrm{I}(\chi)$ is the unique function $f_{\chi}\in\mathrm{I}(\chi)$ such that
$f_{\chi}(g\kappa(k))=f_{\chi}(g)$ for any $g\in\widetilde{G}_n$ and $k\in K$, and $f_{\chi}(\mathfrak{s}(I_n))=1$. We will use the observation
of \cite{KP} (Lemma~I.1.3, see also \cite{McNamara} Lemma~2) that
for any unramified $f_{\chi}$,
\begin{align}\label{eq:unramified vanishes on torus not in T_n*}
f_{\chi}(\mathfrak{s}(t))=0,\qquad\forall t\in T_n\setdifference T_{n,*}.
\end{align}


Let $\mathcal{S}^{\mathrm{gen}}(\widetilde{G}_n)$ denote the subspace of $\mathcal{S}(\widetilde{G}_n)$ consisting of genuine functions. Given an open subset $K_0\subset K$, one can lift any $f\in\mathcal{S}(K_0)$ to $f^{\mathrm{gen}}\in \mathcal{S}^{\mathrm{gen}}(\widetilde{G}_n)$ by putting
\begin{align*}
f^{\mathrm{gen}}(g)=\begin{cases}\epsilon f(k_0)&g=\epsilon\kappa(k_0),\epsilon\in \mu_2, k_0\in K_0,\\0&g\notin\widetilde{K}_0.
\end{cases}
\end{align*}
For example, for any subset $X\subset G_n$ let $\mathrm{ch}_X$ be the characteristic function of $X$, then $\mathrm{ch}^{\mathrm{gen}}_{K}$ vanishes off $\widetilde{K}$ and is $1$ on $\kappa(K)$.

The group $\widetilde{G}_n$ acts on the right on $\mathcal{S}^{\mathrm{gen}}(\widetilde{G}_n)$.
We have a surjection $P_{\chi}:\mathcal{S}^{\mathrm{gen}}(\widetilde{G}_n)\rightarrow\mathrm{I}(\chi)$ of
$\widetilde{G}_n$-representations
\begin{align}\label{eq:P chi surjection}
P_{\chi}(f)(g)&=
\int_{B_{n,*}}f(\mathfrak{s}(b_*)g)\delta^{1/2}_{B_n}(b_*)\chi^{-1}(\mathfrak{s}(b_*))\, d_lb_*
\\\nonumber&=\int_{T_{n,*}}\int_{N_{n}}f(\mathfrak{s}(tn)g)\delta^{1/2}_{B_n}(t)\chi^{-1}(\mathfrak{s}(t))\, dn\, dt.
\end{align}
Here $d_lb_*$ is the left Haar measure normalized by requiring $\mathrm{vol}(B_{n,*}\cap K)=1$, this is consistent with the normalization in the non-metaplectic setting because $B_{n,*}\cap K=B_n\cap K$. Note that $\mathfrak{s}$ is a splitting of $B_{n,*}$.

The unramified normalized function of $\mathrm{I}(\chi)$ is then
$\varphi_{K,\chi}=P_{\chi}(\mathrm{ch}^{\mathrm{gen}}_{K})$.

For $w\in W$, let $\mathfrak{w}\in \mathfrak{W}$ be a representative of $w$ and define $\varphi_{w,\chi}=P_{\chi}(\mathrm{ch}^{\mathrm{gen}}_{\mathcal{I}\mathfrak{w}\mathcal{I}})$. This is the function supported on
$\widetilde{B}_{n,*}\kappa(\mathfrak{w}\mathcal{I})=\widetilde{B}_{n,*}\mathfrak{s}(\mathfrak{w})\kappa(\mathcal{I})$, which is right-invariant by $\kappa(\mathcal{I})$ and
$\varphi_{w,\chi}(\mathfrak{s}(\mathfrak{w}))=1$. Then $\{\varphi_{w,\chi}\}_{w\in W}$ is a basis of
$\mathrm{I}(\chi)^{\mathcal{I}}$. In particular
\begin{align}\label{eq:expansion of phi K using phi w}
\varphi_{K,\chi}=\sum_w\varphi_{w,\chi}.
\end{align}

We introduce the intertwining operators. If $\alpha\in\Sigma_{G_n}$,
let $N_{\alpha}<G_n$ denote the root subgroup of $\alpha$.
For any $w\in W$, let
\begin{align*}
T_{w}=T_{w,\chi}:\mathrm{I}(\chi)\rightarrow\mathrm{I}(\rconj{w}\chi)
\end{align*}
be the standard intertwining operator defined by
\begin{align*}
T_wf_{\chi}(g)=\int_{N_n(w)}f_{\chi}(\mathfrak{s}(\mathfrak{w}^{-1}v)g)\, dv,
\end{align*}
where $\mathfrak{w}\in\mathfrak{W}$ is the representative of $w$ and
\begin{align*}
N_n(w)=\prod_{\setof{\alpha>0}{w^{-1}\alpha<0}}N_{\alpha}<N_n.
\end{align*}
Given that $\chi$ lies in a certain cone, the integral is absolutely convergent, otherwise it is
defined by meromorphic continuation in $\chi$ (i.e., in the parameter $\underline{s}$).

Henceforth assume $\chi$ is regular, i.e., $\rconj{w}\chi\ne\chi$ for any $e\ne w\in W$. In this case (see \cite[Theorem~4]{McNamara})
\begin{align}\label{eq:one dim for regular hom}
\Dim\Hom_{\widetilde{G}_n}(\mathrm{I}(\chi),\mathrm{I}(\rconj{w}\chi))=1,\qquad\forall w\in W.
\end{align}

We recall the Gindikin-Karpelevich formula (\cite{CS1} Theorem~3.1), whose extension to $\widetilde{G}_n$ was proved in \cite{KP}. For any root $\alpha$, let
$\mathcal{R}_{\alpha}$ be the group generated by $N_{\alpha}$ and $N_{-\alpha}$ ($\mathcal{R}_{\alpha}\isomorphic\SL_2$), $a_{\alpha}'$ be the embedding of $\diag(\varpi^2,\varpi^{-2})$ in $\mathcal{R}_{\alpha}$
and put $a_{\alpha}=\mathfrak{s}(a_{\alpha}')$.
Define for $\alpha$ and $w\in W$,
\begin{align*}
&c_{\alpha}(\chi)=\frac{1-q^{-1}\chi(a_{\alpha})}{1-\chi(a_{\alpha})}, \qquad c_w(\chi)=\prod_{\alpha>0:w\alpha<0}c_{\alpha}(\chi).
\end{align*}
Then
\begin{align}\label{eq:Gindikin-Karpelevich formula}
T_w\varphi_{K,\chi}=c_w(\chi)\varphi_{K,\rconj{w}\chi}.
\end{align}
This implies the analog of \cite[Proposition~3.5]{Cs}:
\begin{claim}\label{claim:T_w isomorphism}
$T_{w}(=T_{w,\chi})$ is an isomorphism if and only if $c_{w^{-1}}(\rconj{w}\chi)c_w(\chi)\ne0$. Furthermore,
$\mathrm{I}(\chi)$ is irreducible if and only if $T_{w_0}$ is an isomorphism.
\end{claim}
\begin{proof}
The first part follows exactly as in \cite{Cs}. For the second part one can verify that the arguments in the case of
$G_n$ apply here as well (see \cite{Cs} Theorem~6.6.2). Briefly, if $T_{w_0}$ is not an isomorphism, either $I(\chi)$ or $I(\rconj{w_0}\chi)$ is reducible, but these representations have the same irreducible constituents (\cite[Theorem~2.9]{BZ2}). Suppose $T_{w_0}$ is an isomorphism and $\pi$ is a proper irreducible
quotient of $\mathrm{I}(\chi)$. Then $\pi$ can be embedded in $\mathrm{I}(\rconj{w}\chi)$ for some $w\in W$
(\cite[Theorem~2.9 and Theorem~2.4b]{BZ2}). This embedding is proportional to $T_{w}$ by \eqref{eq:one dim for regular hom}, and it follows that
$T_w$ is not an isomorphism, thus $c_w(\chi)$ or $c_{w^{-1}}(\rconj{w}\chi)$ vanish. But then
$c_{w_0}(\chi)$ or $c_{w_0}(\chi^{-1})$ is zero, contradiction.
\end{proof}
We mention that when $\ell(w)+\ell(w')=\ell(ww')$, $T_{w',\rconj{w}\chi}T_{w,\chi}=T_{w'w,\chi}$, and in general
for any $w,w'$ (see e.g., \cite[(3.17)]{COf}),
\begin{align}\label{eq:formula for composing Tw and Tw'}
T_{w',\rconj{w}\chi}T_{w,\chi}=\frac{c_{w'}(\rconj{w}\chi)c_{w}(\chi)}{c_{w'w}(\chi)}T_{w'w,\chi}.
\end{align}

In contrast with the non-metaplectic case, here $\mathrm{I}(\chi)_{N_n}$ is not linearly isomorphic with
$\mathrm{I}(\chi)^{\mathcal{I}}$ (compare to \cite[Proposition~2.4]{CS1}). Indeed, the dimension of the former is $|W|\cdot|\lmodulo{T_{n,*}}{T_n}|$. However, we can still define an analog of the Casselman basis.

For each $w\in W$ let $\Upsilon_w$ be the functional on $\mathrm{I}(\chi)$ given by
\begin{align*}
\Upsilon_w(f_{\chi})=T_wf_{\chi}(\mathfrak{s}(I_n)).
\end{align*}
Clearly
\begin{align}\label{eq:eigenvalue of lambda w on t}
\Upsilon_{w}(tf_{\chi})=\delta^{1/2}_{B_n}\ \rconj{w}\chi(t)\Upsilon_{w}(f_{\chi}),\qquad\forall f_{\chi}\in\mathrm{I}(\chi),\quad t\in C_{\widetilde{T}_{n}}.
\end{align}

For any genuine character $\omega$ of $\widetilde{T}_{n,*}$, let $\rho(\omega)=\ind_{\widetilde{T}_{n,*}}^{\widetilde{T}_n}(\omega)$. In fact $\rho(\omega)$ depends only on
the restriction of $\omega$ to $C_{\widetilde{T}_n}$.
According to the Geometric Lemma of Bernstein and Zelevinsky \cite[Theorem~5.2]{BZ2}, and because $\chi$ is regular,
\begin{align}\label{eq:geometric lemma on I(chi)}
\mathrm{I}(\chi)_{N_n}=\bigoplus_{w\in W}\delta^{1/2}_{B_n}\rho(\rconj{w}\chi).
\end{align}
Since $\Upsilon_{w}$ factors through $\mathrm{I}(\chi)_{N_n}$, we can
consider it as a functional on this finite direct sum, of finite-dimensional vector spaces. For any $w'\ne w$, if $f_{\chi}+\mathrm{I}(\chi)(N_n)$ belongs to the space of $\rho(\rconj{w'}\chi)$,
\begin{align}\label{eq:eigenvalue of lambda w on t 2}
\Upsilon_{w}(tf_{\chi})=\Upsilon_{w}(\delta^{1/2}_{B_n}\ \rconj{w'}\chi(t)f_{\chi})=\delta^{1/2}_{B_n}\ \rconj{w'}\chi(t)\Upsilon_{w'}(f_{\chi}),
\qquad\forall t\in C_{\widetilde{T}_{n}}.
\end{align}
Combining this with \eqref{eq:eigenvalue of lambda w on t} implies that $\Upsilon_{w}$ vanishes on $\bigoplus_{w'\ne w}\delta^{1/2}_{B_n}\rho(\rconj{w'}\chi)$.

\begin{claim}\label{claim:analog of Casselman basis - def of functionals}
The functionals $\{\Upsilon_w\}_{w}$ restricted to $\mathrm{I}(\chi)^{\mathcal{I}}$ are linearly independent.
\end{claim}
\begin{proof}
According to \cite[Remark~1]{COf}, if we put
\begin{align*}
f_{w,\chi}=T_{w^{-1}w_0}\varphi_{w_0,\rconj{(w^{-1}w_0)^{-1}}\chi}\in\mathrm{I}(\chi)^{\mathcal{I}},
\end{align*}
\begin{align*}
\Upsilon_{w'}(f_{w,\chi})=\delta_{w',w},
\end{align*}
where $\delta_{w',w}$ is the Kronecker delta.
\end{proof}

Let $\{f_{w,\chi}\}_w$ be the basis of $\mathrm{I}(\chi)^{\mathcal{I}}$ dual to $\{\Upsilon_w\}_w$, given
in the proof of Claim~\ref{claim:analog of Casselman basis - def of functionals}. By definition 
\begin{align*}
f_{w_0,\chi}=\varphi_{w_0,\chi},
\end{align*}
as in the non-metaplectic setting (\cite[Proposition~3.7]{Cs}) the computation there can also be performed here, to obtain a direct proof).

Next we compute the image of $f_{w,\chi}$ in $\mathrm{I}(\chi)_{N_n}$ and relate it to the direct sum of \eqref{eq:geometric lemma on I(chi)}.
\begin{claim}\label{claim:action of t on f w in Jacquet module}
The image of $f_{w,\chi}$ in $\mathrm{I}(\chi)_{N_n}$ belongs to the space of $\rho(\rconj{w}\chi)$. Consequently, for any
$t\in C_{\widetilde{T}_n}$, the image of $tf_{w,\chi}$ in $\mathrm{I}(\chi)_{N_n}$ is $\delta^{1/2}_{B_n}\ \rconj{w}\chi(t)f_{w,\chi}+\mathrm{I}(\chi)(N_n)$.
\end{claim}
\begin{proof}
Looking at \eqref{eq:geometric lemma on I(chi)}, we can write the image of $f_{w,\chi}$ in $\mathrm{I}(\chi)_{N_n}$
in the form $\sum_{w\in W}\xi_w$, $\xi_w\in\rho(\rconj{w}\chi)$. Fix $w'\ne w$.
We prove $\xi_{w'}=0$, by showing that any functional on $\rho(\rconj{w'}\chi)$ vanishes on $\xi_{w'}$.

For a representative $b$ of $\lmodulo{T_{n,*}}{T_n}$, consider the linear functional on $\mathrm{I}(\chi)$ given by
\begin{align*}
\Upsilon_{w',b}(f_{\chi})=T_{w'}f_{\chi}(\mathfrak{s}(b)).
\end{align*}
It is defined on $\mathrm{I}(\chi)_{N_n}$ and satisfies equalities similar to \eqref{eq:eigenvalue of lambda w on t} and \eqref{eq:eigenvalue of lambda w on t 2}, i.e.,
\begin{align*}
&\Upsilon_{w',b}(tf_{\chi})=\delta^{1/2}_{B_n}\ \rconj{w'}\chi(t)\Upsilon_{w',b}(f_{\chi}),\qquad\forall t\in C_{\widetilde{T}_n},\\
&\Upsilon_{w',b}(tf_{\chi})=\delta^{1/2}_{B_n}\ \rconj{w''}\chi(t)\Upsilon_{w',b}(f_{\chi}),
\qquad\forall t\in C_{\widetilde{T}_n},\qquad f_{\chi}+\mathrm{I}(\chi)(N_n)\in \rho(\rconj{w''}\chi).
\end{align*}
Hence $\Upsilon_{w',b}$ vanishes on $\rho(\rconj{w''}\chi)$ for any $w''\ne w'$.

The set $\{\Upsilon_{w',b}\}_{b}$ is linearly independent, since for $b_1\in T_n$, 
\begin{align*}
\Upsilon_{w',b}(\mathfrak{s}(b_1)f_{w',\chi})&=
T_{w_0}\varphi_{w_0,\rconj{({w'}^{-1}w_0)^{-1}}\chi}(\mathfrak{s}(b)\mathfrak{s}(b_1))
=\begin{cases}
\delta_{B_n}^{1/2}\ \rconj{w'}\chi(\mathfrak{s}(b)\mathfrak{s}(b_1))&bb_1\in T_{n,*},\\
0&\text{otherwise}.
\end{cases}
\end{align*}
This is because for $t\in T_n$, if $\mathfrak{w}_0N_nt$ intersects the projection of the support of
$\varphi_{w_0,\ldots}$, then $t\in T_{n,*}$ (use \eqref{eq:reformulation of Iwahori} and the fact that $\rconj{\mathfrak{w}_0}t\in B_{n,*}N_n^-$ implies $t\in T_{n,*}$).

Thus $\{\Upsilon_{w',b}\}_{b}$ is a basis of the linear dual of $\rho(\rconj{w'}\chi)$ and moreover,
\begin{align*}
\Upsilon_{w',b}(\xi_{w'})=\Upsilon_{w',b}(f_{w,\chi}).
\end{align*}
Recall that we have to show $\xi_{w'}=0$, we prove $\Upsilon_{w',b}(\xi_{w'})=0$ for all $b\in T_n$.
For $b\in T_{n,*}$, $\Upsilon_{w',b}$ is a scalar multiple of $\Upsilon_{w'}$, hence vanishes on $f_{w,\chi}$ (by definition).
It remains to show
$\Upsilon_{w',b}(f_{w,\chi})=0$ for all $b\notin T_{n,*}$. Using \eqref{eq:formula for composing Tw and Tw'},
\begin{align*}
\Upsilon_{w',b}(f_{w,\chi})
&=\frac{c_{w'}(\rconj{w^{-1}w_0}\chi)c_{w^{-1}w_0}(\chi)}{c_{w'w^{-1}w_0}(\chi)}
T_{w'w^{-1}w_0}\varphi_{w_0,\rconj{(w^{-1}w_0)^{-1}}\chi}(\mathfrak{s}(b))=0,
\end{align*}
since for any $\mathfrak{w}\in\mathfrak{W}$ and $v\in N_n$, $\mathfrak{w}vb\in B_{n,*}\mathfrak{w}_0\mathcal{I}$ implies $b\in T_{n,*}$ (use \eqref{eq:reformulation of Iwahori}, first deduce $\mathfrak{w}=\mathfrak{w}_0$).
\end{proof}
\begin{remark}
This claim will be used to prove Claim~\ref{claim:Iwahori projection and the action of t- on f w} below, which is the analog of \cite[Proposition~3.9]{CS1}.
In \cite{CS1}, the fact that the image of $f_{w,\chi}$ in $\mathrm{I}(\chi)_{N_n}$ belongs to $\rho(\rconj{w}\chi)$ followed
directly from the definition of the basis $\{f_{w,\chi}\}$, because in the non-metaplectic setting the
spaces $\rconj{w}\rho(\chi)$ are one-dimensional, then $\Upsilon_{w'}(f_{w,\chi})=0$ immediately implies $\xi_{w'}=0$.
\end{remark}

\begin{claim}\label{claim:Iwahori projection and the action of t- on f w}(\cite[Proposition~3.9]{CS1})
For any $w\in W$ and $t\in C_{\widetilde{T}_n}\cap\widetilde{T}_n^-$,
\begin{align*}
\mathscr{P}_{\mathcal{I}}(tf_{w,\chi})=\delta^{1/2}_{B_n}\ \rconj{w}\chi(t)f_{w,\chi}.
\end{align*}
\end{claim}
\begin{proof}
By Claims~\ref{claim:Jacquet's First Lemma} and \ref{claim:action of t on f w in Jacquet module},
$\mathscr{P}_{\mathcal{I}}(tf_{w,\chi})$ and $\delta^{1/2}_{B_n}\ \rconj{w}\chi(t)f_{w,\chi}$ have the same image in $\mathrm{I}(\chi)_{N_n}$. Because $f_{w,\chi}=\mathscr{P}_{\mathcal{I}}(f_{w,\chi})$, this implies
\begin{align*}
\mathscr{P}_{\mathcal{I}}(tf_{w,\chi}-\delta^{1/2}_{B_n}\ \rconj{w}\chi(t)f_{w,\chi})\in\mathrm{I}(\chi)(N_n)\cap\mathrm{I}(\chi)^{\mathcal{I}}=0,
\end{align*}
where we used Claim~\ref{claim:Iwahori invariants and Jacquet kernel}. The result follows.
\end{proof}

For each $w\in W$, if $\mathfrak{w}$ is the representative of $w$, put
$q(w)=[\mathcal{I}\mathfrak{w}\mathcal{I}:\mathcal{I}]$.
E.g., $q(s_{\alpha})=q$ for $\alpha\in\Delta_{G_n}$. Then
$[K:\mathcal{I}]=\sum_{w\in W}q(w)$ and since $\mathrm{vol}(K)=1$,
\begin{align}\label{eq:vol of I when vol K = 1}
\mathrm{vol}(\mathcal{I})=(\sum_{w\in W}q(w))^{-1}.
\end{align}

We have the following analog of \cite[Theorem~3.4]{Cs}:
\begin{claim}\label{claim:analog of theorem 3.4 of Casselman}
If $\alpha\in\Delta_{G_n}$, $w\in W$ and $\ell(s_{\alpha}w)>\ell(w)$,
\begin{align*}
&T_{s_{\alpha}}\varphi_{s_{\alpha}w,\chi}=\varphi_{w,\rconj{s_{\alpha}}\chi}+(c_{\alpha}(\chi)-q^{-1})\varphi_{s_{\alpha}w,\rconj{s_{\alpha}}\chi},\\
&T_{s_{\alpha}}\varphi_{w,\chi}=(c_{\alpha}(\chi)-1)\varphi_{w,\rconj{s_{\alpha}}\chi}+q^{-1}\varphi_{s_{\alpha}w,\rconj{s_{\alpha}}\chi}.
\end{align*}
\end{claim}
\begin{proof}
Let $\mathfrak{w}_{\alpha}$ be the representative of $s_{\alpha}$. In the domain of absolute convergence
\begin{align*}
T_{s_{\alpha}}\varphi_{e,\chi}(\mathfrak{s}(\mathfrak{w}_{\alpha}))=\int_{N_{\alpha}}\varphi_{e,\chi}(\mathfrak{s}(\mathfrak{w}_{\alpha})^{-1}\mathfrak{s}(v)\mathfrak{s}(\mathfrak{w}_{\alpha}))\, dv.
\end{align*}
The integrand vanishes unless $v\in N_{\alpha}(\varpi \mathcal{O})<K$, then
$\mathfrak{s}(\mathfrak{w}_{\alpha})^{-1}\mathfrak{s}(v)\mathfrak{s}(\mathfrak{w}_{\alpha})=\kappa(\rconj{\mathfrak{w}_{\alpha}^{-1}}v)$ and the integral equals $\mathrm{vol}(\varpi\mathcal{O})$. Hence
\begin{align*}
T_{s_{\alpha}}\varphi_{e,\chi}(\mathfrak{s}(\mathfrak{w}_{\alpha}))=q^{-1}.
\end{align*}
This is true for all $\chi$ by meromorphic continuation.
A similar computation gives
\begin{align*}
T_{s_{\alpha}}\varphi_{s_{\alpha},\chi}(\mathfrak{s}(I_n))=\varphi_{s_{\alpha},\chi}(\mathfrak{s}(\mathfrak{w}_{\alpha})^{-1})
=1.
\end{align*}
Here we used the fact that
$\mathfrak{s}(\mathfrak{w}_{\alpha})^{-1}=\mathfrak{s}(\mathfrak{w}_{\alpha})\mathfrak{s}(t_0)$ for some
$t_0\in T_n\cap K$.
Exactly as in \cite{Cs}, the rank one case of  \eqref{eq:Gindikin-Karpelevich formula} implies
\begin{align*}
c_{\alpha}(\chi)\varphi_{K,\rconj{s_{\alpha}}\chi}=T_{s_{\alpha}}\varphi_{e,\chi}+T_{s_{\alpha}}\varphi_{s_{\alpha},\chi}
\end{align*}
and we deduce
\begin{align*}
&T_{s_{\alpha}}\varphi_{s_{\alpha},\chi}(\mathfrak{s}(\mathfrak{w}_{\alpha}))=c_{\alpha}(\chi)-q^{-1},\\
&T_{s_{\alpha}}\varphi_{e,\chi}(\mathfrak{s}(I_n))=c_{\alpha}(\chi)-1.
\end{align*}

If $w'\in W\setdifference\{e,s_{\alpha}\}$ and $\mathfrak{w}'$ is the representative of $w'$,
\begin{align*}
T_{s_{\alpha}}\varphi_{e,\chi}(\mathfrak{s}(\mathfrak{w}'))=
T_{s_{\alpha}}\varphi_{s_{\alpha},\chi}(\mathfrak{s}(\mathfrak{w}'))=0,
\end{align*}
because $\mathfrak{w}_{\alpha}^{-1}N_{\alpha}\mathfrak{w}'$ does not intersect with the support of the integrands.
Combining this with the previous equalities,
\begin{align*}
&T_{s_{\alpha}}\varphi_{s_{\alpha},\chi}=\varphi_{e,\rconj{s_{\alpha}}\chi}+(c_{\alpha}(\chi)-q^{-1})\varphi_{s_{\alpha},\rconj{s_{\alpha}}\chi},\\
&T_{s_{\alpha}}\varphi_{e,\chi}=(c_{\alpha}(\chi)-1)\varphi_{e,\rconj{s_{\alpha}}\chi}+q^{-1}\varphi_{s_{\alpha},\rconj{s_{\alpha}}\chi}.
\end{align*}
Now the assertions of the claim follow because if $\mathfrak{w}$ is the representative of
$w$,
\begin{align*}
&q(w)\mathscr{P}_{\mathcal{I}}(\mathfrak{s}(\mathfrak{w}^{-1})\varphi_{e,\chi})=\varphi_{w,\chi},\qquad q(w)\mathscr{P}_{\mathcal{I}}(\mathfrak{s}(\mathfrak{w}^{-1})\varphi_{s_{\alpha},\chi})=\varphi_{s_{\alpha}w,\chi}.
\end{align*}
To see these formulas one can compute the integrals directly at $\mathfrak{s}(\mathfrak{w})$ (resp.
$\mathfrak{s}(\mathfrak{w}_{\alpha}\mathfrak{w})$), then verify that for
$\mathfrak{w}'\ne \mathfrak{w}$ (resp. $\mathfrak{w}'\ne \mathfrak{w}_{\alpha}\mathfrak{w}$),
$\mathfrak{w}'\mathcal{I}\mathfrak{w}^{-1}$ does not intersect $B_{n,*}\mathcal{I}$ (resp. $B_{n,*}\mathfrak{w}_{\alpha}\mathfrak{w}\mathcal{I}$).
\end{proof}

Combining \eqref{eq:expansion of phi K using phi w} with \eqref{eq:Gindikin-Karpelevich formula} and the definition of
the Casselman basis implies
\begin{align}\label{eq:expansion of phi K using Casselman basis}
\varphi_{K,\chi}=\sum_{w\in W}c_w(\chi)f_{w,\chi}.
\end{align}

The $\widetilde{G}_n$-pairing on $\mathrm{I}(\chi)\times\mathrm{I}(\chi^{-1})$ is given by (see \cite[(10.2)]{COf})
\begin{align}\label{eq:metaplectic pairing}
\langle f_{\chi},f_{\chi^{-1}}\rangle=\sum_{b\in\lmodulo{B_{n,*}}{B_n}}\delta^{-1}_{B_n}(b)\int_Kf_{\chi}(\mathfrak{s}(b)\kappa(k))f_{\chi^{-1}}(\mathfrak{s}(b)\kappa(k))\, dk.
\end{align}

The following result was proved in \cite{CS1} in the course of computing Macdonald's formula (\cite[\S~4]{CS1}).
\begin{claim}\label{claim:phi K on Casselman basis}
For any $w\in W$,
\begin{align*}
\mathscr{P}_K(f_{w,\chi})(1)=\frac{c_{w_0}(\rconj{w_0w}\chi)}{Qc_w(\chi)},\qquad Q=\sum_{w\in W}q(w)^{-1}.
\end{align*}
\end{claim}
\begin{proof}
We prove the identity using the formula for the zonal spherical function. Put
\begin{align*}
\Gamma_{\chi}(g)=\langle g\varphi_{K,\chi},\varphi_{K,\chi^{-1}}\rangle,\qquad g\in\widetilde{G}_n.
\end{align*}
In the non-metaplectic setting a formula for $\Gamma_{\chi}(g)$ was proved by MacDonald \cite[Chapter~V, \S~3]{M}, then generalized to reductive groups in \cite{CS1}, as an application of the basis $\{f_{w,\chi}\}$. In the metaplectic case the formula was computed in \cite[\S~10]{COf} and in particular
for any $t\in C_{\widetilde{T}_n}\cap\widetilde{T}_n^-$,
\begin{align}\label{eq:reinterpretation of COf}
\Gamma_{\chi}(t)=Q^{-1}\sum_{w\in W}c_{w_0}(\rconj{w_0w}\chi)\delta^{1/2}_{B_n}\ \rconj{w}\chi(t).
\end{align}
On the other hand using \eqref{eq:metaplectic pairing} and the fact that $\varphi_{K,\chi^{-1}}(\mathfrak{s}(b))=0$ unless
$b\in B_{n,*}$ (see \eqref{eq:unramified vanishes on torus not in T_n*}),
\begin{align*}
\Gamma_{\chi}(g)
&=\sum_{b\in\lmodulo{B_{n,*}}{B_n}}\delta^{-1}_{B_n}(b)\int_K\varphi_{K,\chi}(\mathfrak{s}(b)\kappa(k)g)\varphi_{K,\chi^{-1}}(\mathfrak{s}(b)\kappa(k))\, dk\\
&=\int_K\varphi_{K,\chi}(\kappa(k)g)\, dk=\mathscr{P}_K(g\varphi_{K,\chi})(1).
\end{align*}
Putting \eqref{eq:expansion of phi K using Casselman basis} into this equality shows
\begin{align*}
\Gamma_{\chi}(t)=\sum_{w\in W}c_w(\chi)\mathscr{P}_K(tf_{w,\chi})(1).
\end{align*}
Since $\mathscr{P}_K(tf_{w,\chi})(1)=\mathscr{P}_K(\mathscr{P}_{\mathcal{I}}(tf_{w,\chi}))(1)$, Claim~\ref{claim:Iwahori projection and the action of t- on f w} implies
\begin{align*}
\Gamma_{\chi}(t)=\sum_{w\in W}c_w(\chi)\mathscr{P}_K(f_{w,\chi})(1)\delta^{1/2}_{B_n}\ \rconj{w}\chi(t).
\end{align*}
Comparing this with \eqref{eq:reinterpretation of COf} and since the characters $\chi$ are linearly independent as functions on
$t\in C_{\widetilde{T}_n}\cap\widetilde{T}_n^-$, we obtain the result.
\end{proof}
We described the tools of \cite{CS1} that we need and proceed to formulate an analog of Hironaka's theorem \cite[Proposition~1.9]{Hir}, for
computing functionals on $\mathrm{I}(\chi)$.

\subsection{Metaplectic version of Hironaka's theorem}\label{subsection:Hironaka theorem}
Let $\mathrm{I}(\chi)^*$ denote the linear dual of $\mathrm{I}(\chi)$.
For any $w\in W$, by definition
\begin{align*}
T_{w^{-1},\rconj{w}\chi^{-1}}^*:\mathrm{I}(\chi^{-1})^*\rightarrow\mathrm{I}(\rconj{w}\chi^{-1})^*.
\end{align*}
Since $\mathrm{I}(\chi)$ is the subspace of $\mathrm{I}(\chi^{-1})^*$ of smooth functionals and $T_{w^{-1},\rconj{w}\chi^{-1}}^*$ preserves smoothness,
\begin{align*}
T_{w^{-1},\rconj{w}\chi^{-1}}^{\vee}=T_{w^{-1},\rconj{w}\chi^{-1}}^*|_{\mathrm{I}(\chi)}:\mathrm{I}(\chi)\rightarrow\mathrm{I}(\rconj{w}\chi).
\end{align*}
\begin{claim}\label{claim:defining twisted intertwining functional}(\cite[Lemma~1.5, Proposition~1.6]{Hir})
If $c_w(\chi)\ne0$,
\begin{align*}
T_{w^{-1},\rconj{w}\chi^{-1}}^{\vee}=\frac{c_{w^{-1}}(\rconj{w}\chi^{-1})}{c_w(\chi)}T_{w,\chi}.
\end{align*}
If also $c_{w^{-1}}(\rconj{w}\chi^{-1})\ne0$, the map
\begin{align*}
\mathcal{T}_{w^{-1},\rconj{w}\chi^{-1}}=\frac{c_w(\chi)}{c_{w^{-1}}(\rconj{w}\chi^{-1})}T_{w^{-1},\rconj{w}\chi^{-1}}^*:\mathrm{I}(\chi^{-1})^*\rightarrow\mathrm{I}(\rconj{w}\chi^{-1})^*
\end{align*}
extends $T_{w,\chi}$.
\end{claim}
\begin{proof}
By \eqref{eq:one dim for regular hom} the mappings $T_{w^{-1},\rconj{w}\chi^{-1}}^{\vee}$ and $T_{w,\chi}$ are proportional. To compute the proportionality
factor, exactly as in \cite{Hir} we compare
\begin{align*}
\langle T_{w,\chi}(\varphi_{K,\chi}),\varphi_{K,\rconj{w}\chi^{-1}}\rangle\text{ to }
\langle T_{w^{-1},\rconj{w}\chi^{-1}}^{\vee}(\varphi_{K,\chi}),\varphi_{K,\rconj{w}\chi^{-1}}\rangle
\end{align*}
using
\eqref{eq:Gindikin-Karpelevich formula} and \eqref{eq:metaplectic pairing}. The $\widetilde{G}_n$-pairing given by \eqref{eq:metaplectic pairing}
involves a finite summation, but by \eqref{eq:unramified vanishes on torus not in T_n*} applied to $\varphi_{K,\rconj{w}\chi^{-1}}$, it reduces to one summand corresponding to the trivial coset.
\end{proof}
As a corollary we obtain the following result, explaining the normalization factor used to define
$\mathcal{T}_{w^{-1},\rconj{w}\chi^{-1}}$.
\begin{corollary}\label{corollary:intertwining commputes with projection}(\cite[Proposition~1.7]{Hir})
If $c_w(\chi)c_{w^{-1}}(\rconj{w}\chi^{-1})\ne0$, then for any compact open $K_0<K$,
\begin{align*}
\mathscr{P}_{K_0}\circ\mathcal{T}_{w^{-1},\rconj{w}\chi^{-1}}=T_{w,\chi}\circ\mathscr{P}_{K_0}.
\end{align*}
\end{corollary}
\begin{proof}
Since $\mathscr{P}_{K_0}$ takes $\mathrm{I}(\chi^{-1})^*$ to
$(\mathrm{I}(\chi^{-1})^*)^{K_0}\subset \mathrm{I}(\chi)$, both
sides are mappings $\mathrm{I}(\chi^{-1})^*\rightarrow \mathrm{I}(\rconj{w}\chi)$. Consider the \lhs.
Since $T_{w^{-1},\rconj{w}\chi^{-1}}$ and $\mathscr{P}_{K_0}$ commute,
so do $\mathcal{T}_{w^{-1},\rconj{w}\chi^{-1}}$ and $\mathscr{P}_{K_0}$,
but on the image of $\mathscr{P}_{K_0}$, according to Claim~\ref{claim:defining twisted intertwining functional},  $\mathcal{T}_{w^{-1},\rconj{w}\chi^{-1}}$ extends $T_{w,\chi}$.
\end{proof}

Asssume we have a family $\{\Lambda_{a,\chi}\}_{a\in\mathcal{A}}$ of functionals $\Lambda_{a,\chi}\in \mathrm{I}(\chi^{-1})^*$ indexed by a finite set $\mathcal{A}$, and such that for each $a$, $\Lambda_{a,\chi}$ is a meromorphic function in $\chi$.
We fix an arbitrary ordering on $\mathcal{A}$. Furthermore, assume that for each $w\in W$
there is an invertible matrix $A(w,\chi)=(A(w,\chi)_{a,a'})$ such that
\begin{align}\label{eq:matrix relation between family of functionals}
(T_{w^{-1},\rconj{w}\chi^{-1}}^*\Lambda_{a,\chi})_{a\in\mathcal{A}}=A(w,\chi)\ (\Lambda_{a,\rconj{w}\chi})_{a\in\mathcal{A}}.
\end{align}
The entries of $A(w,\chi)$ are by definition meromorphic in $\chi$.
\begin{remark}
In \cite{Hir} the matrix is defined using $\mathcal{T}_{w^{-1},\rconj{w}\chi^{-1}}$, we followed the convention of  \cite{KP,COf}.
\end{remark}

For example, one can consider the family $\{\Upsilon_{w,b}\}$ of functionals introduced in the proof
of Claim~\ref{claim:action of t on f w in Jacquet module}, then $A(w,\chi)$ exists because these form a basis of
$(\mathrm{I}(\chi)_{N_n})^*$. Another example is the Whittaker functionals, then $\mathcal{A}$ can be a set of
representatives of $\lmodulo{T_{n,*}}{T_n}$ (see \cite{KP,COf}).

We are interested in obtaining a formula for $\Lambda_{a,\chi}(g\varphi_{K,\chi^{-1}})$, where $g\in\widetilde{G}_n$. Since $\mathscr{P}_{\mathcal{I}}(g^{-1}\Lambda_{a,\chi})\in(\mathrm{I}(\chi^{-1})^*)^{\mathcal{I}}\subset\mathrm{I}(\chi)^{\mathcal{I}}$, there is
a function $\lambda_{a,\chi,g}\in\mathrm{I}(\chi)^{\mathcal{I}}$ satisfying
\begin{align}\label{eq:correspondingg function from functional}
\mathscr{P}_I(g^{-1}\Lambda_{a,\chi})(f_{\chi^{-1}}')=\langle \lambda_{a,\chi,g},f_{\chi^{-1}}'\rangle,\qquad\forall f_{\chi^{-1}}'\in\mathrm{I}(\chi^{-1}).
\end{align}
In other words $\lambda_{a,\chi,g}$ is the image of $\mathscr{P}_I(g^{-1}\Lambda_{a,\chi})$ in $\mathrm{I}(\chi)$.
\begin{claim}\label{claim:meaning of Gamma a chi and relation to functional}
$\Lambda_{a,\chi}(g\varphi_{K,\chi^{-1}})=\mathscr{P}_K(\lambda_{a,\chi,g})(1)$.
\end{claim}
\begin{proof}
Equality~\eqref{eq:correspondingg function from functional} together with \eqref{eq:metaplectic pairing}
and \eqref{eq:unramified vanishes on torus not in T_n*} imply
\begin{align*}
\mathscr{P}_I(g^{-1}\Lambda_{a,\chi})(\varphi_{K,\chi^{-1}})=\mathscr{P}_K(\lambda_{a,\chi,g})(1).
\end{align*}
Then since
\begin{align*}
\Lambda_{a,\chi}(g\varphi_{K,\chi^{-1}})=g^{-1}\Lambda_{a,\chi}(\varphi_{K,\chi^{-1}})=
\mathscr{P}_I(g^{-1}\Lambda_{a,\chi})(\varphi_{K,\chi^{-1}}),
\end{align*}
the result holds.
\end{proof}
Here is the analog of \cite[Proposition~1.9]{Hir}.
\begin{lemma}\label{lemma:analog of Hironaka theory}
\begin{align*}
(\Lambda_{a,\chi}(g\varphi_{K,\chi^{-1}}))_{a\in\mathcal{A}}
=\sum_{w\in W}\frac{c_{w_0}(\rconj{w_0w}\chi)}{Qc_{w^{-1}}(\rconj{w}\chi^{-1})}
A(w,\chi)(\lambda_{a,\rconj{w}\chi,g}(1))_{a\in\mathcal{A}}.
\end{align*}
\end{lemma}
\begin{proof}
The formal steps in \cite{Hir} can be carried out immediately, because we established the
machinery used in \textit{loc. cit.} for the metaplectic setting. We repeat the steps, for readability.
For the proof we may assume $c_w(\chi)c_{w^{-1}}(\rconj{w}\chi^{-1})\ne0$.
Since $\lambda_{a,\chi,g}\in\mathrm{I}(\chi)^{\mathcal{I}}$, we can write
\begin{align}\label{eq:expression of projection onto Iwahori of functional using Casselman basis}
\lambda_{a,\chi,g}=\sum_{w\in W}d_{a,\chi}(g)f_{w,\chi},
\end{align}
for some complex coefficients $d_{a,\chi}(g)$. The coefficients can be computed by applying $\Upsilon_{w,\chi}$,
\begin{align*}
d_{a,\chi}(g)=T_{w,\chi}\lambda_{a,\chi,g}(1).
\end{align*}
By the definition of $\lambda_{a,\chi,g}$ and using Corollary~\ref{corollary:intertwining commputes with projection},
\begin{align*}
T_{w,\chi}\lambda_{a,\chi,g}=T_{w,\chi}\mathscr{P}_I(g^{-1}\Lambda_{a,\chi})
=\mathscr{P}_{\mathcal{I}}(\mathcal{T}_{w^{-1},\rconj{w}\chi^{-1}}g^{-1}\Lambda_{a,\chi}).
\end{align*}
Then Claim~\ref{claim:defining twisted intertwining functional} implies
\begin{align*}
d_{a,\chi}(g)=\mathscr{P}_{\mathcal{I}}(\mathcal{T}_{w^{-1},\rconj{w}\chi^{-1}}g^{-1}\Lambda_{a,\chi})(1)
=\frac{c_w(\chi)}{c_{w^{-1}}(\rconj{w}\chi^{-1})}\mathscr{P}_{\mathcal{I}}(g^{-1}T_{w^{-1},\rconj{w}\chi^{-1}}^*\Lambda_{a,\chi})(1).
\end{align*}
By Claim~\ref{claim:meaning of Gamma a chi and relation to functional},
to compute $\Lambda_{a,\chi}(g\varphi_{K,\chi^{-1}})$ we apply $\mathscr{P}_{K}$ to
\eqref{eq:expression of projection onto Iwahori of functional using Casselman basis} and evaluate at $1$.
Claim~\ref{claim:phi K on Casselman basis} gives the values of $\mathscr{P}_{K}(f_{w,\chi})(1)$. We obtain
\begin{align*}
\Lambda_{a,\chi}(g\varphi_{K,\chi^{-1}})=\sum_{w\in W}\mathscr{P}_{K}(f_{w,\chi})(1)d_{a,\chi}(g)
=\sum_{w\in W}\frac{c_{w_0}(\rconj{w_0w}\chi)}{Qc_{w^{-1}}(\rconj{w}\chi^{-1})}\mathscr{P}_{\mathcal{I}}(g^{-1}T_{w^{-1},\rconj{w}\chi^{-1}}^*\Lambda_{a,\chi})(1).
\end{align*}
Now since
\begin{align*}
\mathscr{P}_{\mathcal{I}}(g^{-1}T_{w^{-1},\rconj{w}\chi^{-1}}^*\Lambda_{a,\chi})
=T_{w^{-1},\rconj{w}\chi^{-1}}^*\mathscr{P}_{\mathcal{I}}(g^{-1}\Lambda_{a,\chi}),
\end{align*}
\eqref{eq:matrix relation between family of functionals} implies
\begin{align*}
(\Lambda_{a,\chi}(g\varphi_{K,\chi^{-1}}))_{a\in\mathcal{A}}
=\sum_{w\in W}\frac{c_{w_0}(\rconj{w_0w}\chi)}{Qc_{w^{-1}}(\rconj{w}\chi^{-1})}A(w,\chi)(
\lambda_{a,\rconj{w}\chi,g}(1))_{a\in\mathcal{A}},
\end{align*}
as claimed.
\end{proof}
As observed in \cite{Yia}, the value of $\lambda_{a,\chi,g}(1)$ corresponds to the evaluation of the functional
at $g\varphi_{e,\chi^{-1}}$. Namely,
\begin{claim}
$\lambda_{a,\chi,g}(1)=\mathrm{vol}(\mathcal{I})^{-1}\Lambda_{a,\chi}(g\varphi_{e,\chi^{-1}})$.
\end{claim}
\begin{proof}
We show
\begin{align*}
\lambda_{a,\chi,g}(1)=\mathrm{vol}(\mathcal{I})^{-1}\langle \lambda_{a,\chi,g},\varphi_{e,\chi^{-1}}\rangle.
\end{align*}
Then the result follows from \eqref{eq:correspondingg function from functional} because
\begin{align*}
\mathscr{P}_I(g^{-1}\Lambda_{a,\chi})(\varphi_{e,\chi^{-1}})=
\Lambda_{a,\chi}(g\varphi_{e,\chi^{-1}}).
\end{align*}
For $b\in B_n$ and $k\in K$, if $\mathfrak{s}(b)\kappa(k)$ belongs to the support of $\varphi_{e,\chi^{-1}}$,
then $b\in B_{n,*}$ and $k\in\mathcal{I}$. This is because when $tT_{n,*}\ne t'T_{n,*}$, the sets
$tB_{n,*}K$ and $t'B_{n,*}K$ are disjoint and $K\cap B_{n,*}\mathcal{I}=\mathcal{I}$ (use \eqref{eq:reformulation of Iwahori}). Hence looking at \eqref{eq:metaplectic pairing} we obtain
\begin{align*}
\langle \lambda_{a,\chi,g},\varphi_{e,\chi^{-1}}\rangle=\int_{\mathcal{I}}\lambda_{a,\chi,g}(\kappa(k))\, dk=\mathrm{vol}(\mathcal{I})\lambda_{a,\chi,g}(1),
\end{align*}
as required.
\end{proof}
\begin{corollary}\label{corollary:analog of Hironaka theory with phi I functions}
\begin{align*}
(\Lambda_{a,\chi}(g\varphi_{K,\chi^{-1}}))_{a\in\mathcal{A}}=\frac1{Q\mathrm{vol}(\mathcal{I})}\sum_{w\in W}\frac{c_{w_0}(\rconj{w_0w}\chi)}{c_{w^{-1}}(\rconj{w}\chi^{-1})}
A(w,\chi)(\Lambda_{a,\rconj{w}\chi}(g\varphi_{e,\chi^{-1}}))_{a\in\mathcal{A}}.
\end{align*}
\end{corollary}
As mentioned in the introduction, the proofs in \S~\ref{section:Casselman's basis and a result of Hironaka}-\ref{subsection:Hironaka theorem} apply to a more general setting of covering groups. The preliminary results we used are essentially
the Gindikin-Karpelevich formula \eqref{eq:Gindikin-Karpelevich formula}, the description of Casselman's basis using intertwining operators (in the
proof of Claim~\ref{claim:analog of Casselman basis - def of functionals}) and a formula for the zonal spherical function (used for the proof
Claim~\ref{claim:phi K on Casselman basis}). These results were established in \cite{COf} for $r$-fold covers of $G_n$, hence our results apply to these groups as well.

McNamara has generalized substantially several results of \cite{COf}, to any central extension
of $G$ by a finite cyclic group, where $G$ is an unramified reductive group over a $p$-adic field.
See \cite[Theorem~6.4]{McNamara2} regarding the Gindikin-Karpelevich formula and \cite[Proposition~5.2]{McNamara3} (Casselman's basis). As mentioned in
\cite{McNamara3}, the zonal spherical formula can also be proved using his tools. To be cautious, we refrain from stating our results in this generality.

\section{The metaplectic Shalika model}\label{section:Metaplectic unramified Casselman--Shalika formula}
\subsection{Definition}\label{subsection:metaplectic Shalika functional and model}
Let $F$ be a local field and $\psi$ be a nontrivial additive character of $F$. First we recall the notion of the Shalika model, introduced by Jacquet and Shalika \cite{JS4}.

Put $n=2k$. Consider the unipotent radical $U_k$ of the parabolic
subgroup $Q_k<G_n$, and the (Shalika) character $\psi$ of $U_k$ given by
\begin{align*}
\psi(\left(\begin{smallmatrix}I_k&u\\&I_k\end{smallmatrix}\right))=\psi(\mathrm{tr}(u)).
\end{align*}
The normalizer of $U_k$ and stabilizer of $\psi$ is
\begin{align*}
G_k^{\triangle}=\{c^{\triangle}:c\in G_k\},\qquad c^{\triangle}=\diag(c,c).
\end{align*}
The group $G_k^{\triangle}\ltimes U_k$ is called the Shalika group.
Let $\pi\in\Alg{G_n}$. A $\psi$-Shalika functional on the space of $\pi$ is a functional $l$ such that for any vector $\xi$ in the space of $\pi$, $c\in G_k$ and $u\in U$,
\begin{align*}
l(\pi(c^{\triangle}u)\xi)=\psi(u)l(\xi).
\end{align*}
Here on the \lhs\ $u=\left(\begin{smallmatrix}I_k&u\\&I_k\end{smallmatrix}\right)$.
If the field is Archimedean we also demand that $l$ is continuous.

Jacquet and Rallis \cite{JR2} proved that over a non-Archimedean field, when $\pi$ is irreducible,
the space of such functionals is at most one-dimensional.
First they proved that $(G_n,M_k)$ is a Gelfand pair, then they used analytic techniques similar to \cite{FJ} to show that a Shalika functional induces an $M_k$-invariant functional. Their result was extended to Archimedean fields in \cite{AGJ}.

We mention that Ash and Ginzburg \cite{AG} already proved uniqueness results for a certain class of irreducible
unramified representations, we will adapt their proof to the metaplectic case (see
Theorem~\ref{theorem:unramified Shalika functional 1 dim} below).

Granted that $l$ exists, one may define the Shalika model $\mathscr{S}(\pi,\psi)$ of $\pi$ as the space of functions $\mathscr{S}_{\xi}(g)=l(\pi(g)\xi)$, as $\xi$ varies in the space of $\pi$.

We explain the metaplectic analog (following \cite{me11}). According to \eqref{eq:block-compatibility}, restriction of the cover of $G_n$ to $G_k^{\triangle}$ is a ``simple cover" in the sense that $\sigma(c^{\triangle},{c'}^{\triangle})=(\det{c},\det{c'})_2$. Hence
\begin{align}\label{eq:formula for s on triangle}
\mathfrak{s}((cc')^{\triangle})=\sigma(c^{\triangle},{c'}^{\triangle})\mathfrak{s}(c^{\triangle})\mathfrak{s}({c'}^{\triangle})=
(\det c,\det c')_2\mathfrak{s}(c^{\triangle})\mathfrak{s}({c'}^{\triangle}).
\end{align}
Then if $\psi'$ is another nontrivial additive character of $F$, any genuine representation of $\widetilde{G}_k^{\triangle}$ takes the form $\pi\otimes\gamma_{\psi'}$, where $\pi$ is a representation of $G_k$ and
\begin{align*}
\pi\otimes\gamma_{\psi'}(\epsilon\mathfrak{s}(c^{\triangle}))=\epsilon\gamma_{\psi'}(\det c)\pi(c),\qquad \epsilon\in\mu_2.
\end{align*}
For a genuine $\pi\in\Alg{\widetilde{G}_n}$, we call $l$ a metaplectic $(\psi',\psi)$-Shalika functional on $\pi$ if
\begin{align}\label{eq:metaplectic Shalika def}
l(\pi(\mathfrak{s}(c^{\triangle}u))\xi)=\gamma_{\psi'}(\det c)\psi(u)l(\xi).
\end{align}
The terminology is reasonable because $\gamma_{\psi'}$ corresponds to the trivial character of $G_k$.
If $\pi$ is irreducible, we denote the corresponding metaplectic Shalika model by $\mathscr{S}(\pi,\psi',\psi)$. It is a subrepresentation of $\Ind_{\widetilde{G}_k^{\triangle}U_k}^{\widetilde{G}_n}(\gamma_{\psi'}\otimes\psi)$, where
$\gamma_{\psi'}\otimes\psi(\epsilon\mathfrak{s}(c^{\triangle}u))=
\epsilon\gamma_{\psi'}(\det c)\psi(u)$.
\begin{remark}\label{remark:metaplectic Shalika relevant to both covers}
Since $\det c^{\triangle}\in F^{*2}$, we also have
$\sigma^{(1)}(c^{\triangle},{c'}^{\triangle})=(\det{c},\det{c'})_2$ and \eqref{eq:formula for s on triangle} also holds
for the section corresponding to $\sigma^{(1)}$. (The cocycle $\sigma^{(1)}$ was defined in \S~\ref{subsection:The metaplectic cover G_n}, as we mentioned there our work applies to both covers.)
\end{remark}
We have the following bound for Shalika functions.
\begin{claim}\label{claim:general bound on Shalika function}
Assume that $\pi\in\Alg{\widetilde{G}_n}$ is genuine irreducible and admits a metaplectic Shalika model. Then there exists $\alpha>0$ such that for any $\xi$ in the space of $\pi$, there exists a positive $\phi\in\mathcal{S}(F_{k\times k})$ such that
\begin{align*}
|\mathscr{S}_{\xi}(\mathfrak{s}(\diag(g,I_k)))|\leq|\det g|^{-\alpha}\phi(g),\qquad\forall g\in G_k.
\end{align*}
\end{claim}
\begin{proof}
Over a $p$-adic field, the proof is a simple adaptation of Lemma~6.1 of Jacquet and Rallis \cite{JR2} for the non-metaplectic setting. The main difference is that the asymptotic expansion is written for $T_k^2$ instead of $T_k$, the group $T_k^2$ is abelian and splits under the cover.

If $F$ is Archimedean, by the Dixmier--Malliavin Lemma applied to the representation $\pi|_{U_k}$ there exist $\xi_j$ in the representation space of $\pi$ and $\phi_j\in\mathcal{S}(U_k)$ such that $\xi = \sum_j \theta(\phi_j)\xi_j$. This implies that for $d_g=\mathfrak{s}(\diag(g,I_k))$ we have
\begin{align*}
 \mathscr{S}_\xi(d_g)&=\sum_j\int_{U_k}\phi_j(u)\mathscr{S}_{\xi_j}(d_g\mathfrak{s}(u))\, du
=\sum_j \mathscr{S}_{\xi_j}(d_g) \int_{U_k}\phi_j(u)\psi(gu)\, du
=\sum_j \widehat{\phi}_j(g)\mathscr{S}_{\xi_j}(d_g).
\end{align*}
It remains to show that each term $\mathscr{S}_{\xi_j}(d_g)$ is bounded by a constant times $|\det g|^{-\alpha}$ for some $\alpha$ independent of $\xi_j$. This follows from \cite[Lemma 3.2]{AGJ}, which also holds in the metaplectic setting since the proof merely uses the Lie algebra action.
\end{proof}

The next claim is a simple analog of the relation between Whittaker and Kirillov models (see \cite[Proposition~2]{GK}).
\begin{claim}\label{claim:richness of restriction of Shalika model to Q}
Assume $F$ is non-Archimedean. Let $\pi$ be a genuine irreducible subrepresentation of
$\Ind_{\widetilde{G}_k^{\triangle}U_k}^{\widetilde{G}_n}(\gamma_{\psi'}\otimes\psi)$.
The restriction of functions in the space of $\pi$ to $\widetilde{Q}_n$ contains any element of
$\ind_{\widetilde{G}_k^{\triangle}U_k}^{\widetilde{Q}_n}(\gamma_{\psi'}\otimes\psi)$.
\end{claim}
\begin{proof}
For any $\xi$ in the space of $\pi$ and $\phi\in\mathcal{S}(U_k)$, the function $\phi(\xi)$ defined by
\begin{align*}
\phi(\xi)(x)=\int_{U_k}\phi(u)\xi(x\mathfrak{s}(u))du,\qquad x\in\widetilde{G}_n,
\end{align*}
belongs to the space of $\pi$.

Given $d_g=\mathfrak{s}(\diag(g,I_k))\in\widetilde{Q}_n$, one can find $\xi$ such that $\xi(d_g)\ne0$ and we also have
$\phi(\xi)(d_g)=\widehat{\phi}(g)\xi(d_g)$, where $\widehat{\phi}$ denotes the Fourier transform of $\phi$ with respect to $\psi$. Note that the group of characters of $U_k$ is isomorphic to the direct product of $k^2$ copies of $F$. We can select $\phi$ such that the support of $\widehat{\phi}$ is contained in $g\mathcal{V}$, where $\mathcal{V}<G_k$ is a small compact open neighborhood of the identity. Then
$\phi(\xi)|_{\widetilde{Q}_n}$ belongs to $\ind_{\widetilde{G}_k^{\triangle}U_k}^{\widetilde{Q}_n}(\gamma_{\psi'}\otimes\psi)$. The latter space is spanned by these functions, as $g$ and $\mathcal{V}$ vary.
\end{proof}

\subsection{Asymptotic expansion of metaplectic Shalika functions}\label{section:Archimedean asympt expansion Shalika}
As with Whittaker functions, one can write an asymptotic expansion for Shalika functions, as functions on the torus.
Over $p$-adic fields, such an expansion was obtained by Jacquet and Rallis \cite[\S~6.2]{JR2} and their result can be easily extended to the metaplectic setting. Over Archimedean fields, as a rule, the problem is more difficult, because one has to deal with delicate continuity properties. Here we provide an asymptotic expansion for metaplectic Shalika functions over $F=\R$ or $\C$. In particular, the same result holds in the non-metaplectic setting for the Shalika functional.

\begin{remark}The results of this section will be used in \S~\ref{section:A Godement Jacquet integral} to prove the meromorphic continuation of the local integrals (see the proof of Theorem~\ref{theorem:local props GJ integral archimedean}).
\end{remark}
Let $n=2k$. We use the notation and definitions of \S~\ref{subsection:asymptotic expansion}, e.g.,
$A_\ell$, $a_x$, $E_1(Q_\ell,\pi)$ and $\Lambda_{\pi,\ell}$, where now $\pi\in\Alg{\widetilde{G}_n}$ is a genuine representation. Note that $a_x$ is viewed as an element in $G_n$ rather than $G_k$ by embedding $G_k$ into $G_n$ via $g\mapsto\diag(g,I_k)$.
\begin{theorem}\label{thm:AsymptoticExpansionShalikaFunctionsArchimedean}
Let $\pi\in\Alg{\widetilde{G}_n}$ be a genuine irreducible representation on a space $V$, that admits a metaplectic Shalika functional $l$. Then for all $D=(D_1,\ldots,D_k)\in\R^k$ there exist finite subsets $C_\ell\subset E^{(\ell)}(\pi)$ and functions $p_{I,z}:\R^k\times V\to\C$ for $I\subset\{1,\ldots,k\}$, $z\in\prod_{\ell\in I}C_\ell$, such that the following holds.
\begin{enumerate}
\item For all $\xi\in V$ and $x\in\R^k$,
\begin{align}
 l(\pi(\mathfrak{s}(a_x))\xi) = \sum_{I,z}p_{I,z}(x;\xi)\cdot e^{-z\cdot x_I},\label{eq:ExpansionShalikaFunctionals}
\end{align}
where the (finite) summation is over all $I\subset\{1,\ldots,k\}$ and $z\in\prod_{\ell\in I}C_\ell$.
\item Each function $p_{I,z}(x;\xi)$ is polynomial in $x_I$:
\begin{align}
 p_{I,z}(x;\xi) = \sum_{\alpha\in\N^I}c_\alpha(x_{\overline{I}};\xi)x_I^\alpha,\label{eq:PolynomialCoefficientsShalikaFunctionals}
\end{align}
where the coefficients $c_\alpha(x_{\overline{I}};\xi)$ are smooth in $x_{\overline{I}}$, linear in $\xi$, and satisfy bounds of the form
\begin{align}
 |c_\alpha(x_{\overline{I}};\xi)| \leq q(\xi)(1+|x_{\overline{I}}|)^d e^{-D_{\overline{I}}\cdot x_{\overline{I}}}, \qquad \forall x_{\overline{I}}\in\R_{>0}^{\overline{I}},\xi\in V,\label{eq:EstimateRemainderShalikaFunctionals}
\end{align}
with $d\geq0$ and a continuous seminorm $q$ on $V$.
\end{enumerate}
\end{theorem}

\begin{proof}
The proof is similar to the proof of Theorem~\ref{thm:AsymptoticExpansionSmoothMatrixCoefficientsArchimedean}, we explain the necessary modifications. The main difference is, of course, that $l$ is not a smooth functional.

In contrast to the situation in Theorem~\ref{thm:AsymptoticExpansionSmoothMatrixCoefficientsArchimedean} we cannot replace $l$ by $\pi^\vee(X)l$ for $X\in\mathcal{U}(\mathfrak{g}_n)$ since in general $\pi^\vee(X)l$ is not a metaplectic Shalika functional. This replacing was used in Theorem~\ref{thm:AsymptoticExpansionSmoothMatrixCoefficientsArchimedean} for elements $X\in\mathfrak{u}_\ell$, $1\leq\ell\leq k$ (in Theorem~\ref{thm:AsymptoticExpansionSmoothMatrixCoefficientsArchimedean} everything depended continuously on the funtional). However, each $X\in\mathfrak{u}_\ell$ can be written as $X=X_1+X_2$ with $X_1\in\mathfrak{u}_k$ and $X_2=\diag(Y,0)$ for some $Y\in\mathfrak{g}_k$ (here $0$ denotes the $k\times k$ zero matrix).
By \eqref{eq:metaplectic Shalika def} we obtain
\begin{align*}
 l(\pi(X)\pi(\mathfrak{s}(a_x))\xi) &= l(\pi(X_1)\pi(\mathfrak{s}(a_x))\xi) + l(\pi(\diag(Y,0))\pi(\mathfrak{s}(a_x))\xi)\\
 &= \psi(X_1)l(\pi(\mathfrak{s}(a_x))\xi) - l(\pi(\diag(0,Y))\pi(\mathfrak{s}(a_x))\xi)\\
 &= \psi(X_1)l(\pi(\mathfrak{s}(a_x))\xi) - l(\pi(\mathfrak{s}(a_x))\pi(\diag(0,Y))\xi).
\end{align*}
In the second term we can replace $\pi(\diag(0,Y))\xi$ by $\xi$ since the map $\xi\mapsto\pi(\diag(0,Y))\xi$ is continuous and all data in the statement depends continuously on $\xi$.
This shows that the recursive argument in the proof of Theorem~\ref{thm:AsymptoticExpansionSmoothMatrixCoefficientsArchimedean} extends to our setting.

It remains to show that an estimate of the form \eqref{eq:BoundSmoothMatrixCoefficients} also holds for $l$. More precisely, we show that there exists a continuous seminorm $q$ on $V$ and $d\geq0$ such that
\begin{align}
 |l(\pi(\mathfrak{s}(a_x))\xi)| \leq q(\xi)(1+|x|)^de^{-\Lambda_\pi\cdot x}, \qquad \forall x\in\R_{>0}^k,\xi\in V.\label{eq:BoundShalikaFunctional}
\end{align}
First, $l$ is a continuous linear functional on $V$, so there exists a continuous seminorm $q_1$ on $V$ such that
\begin{align*}
|l(\xi)| \leq q_1(\xi), \qquad \forall\xi\in V.
\end{align*}
Since $\pi$ is of moderate growth there exists $N\geq0$ and a continuous seminorm $q_2$ on $V$ such that
\begin{align*}
q_1(\pi(g)\xi) \leq q_2(\xi)\|g\|^N, \qquad \forall g\in\widetilde{G}_n,\xi\in V,
\end{align*}
where $\|\blank\|$ is the pullback to $\widetilde{G}_n$ of the usual norm function on $G_n$, i.e.,
\begin{align*}
\|h\| = \sum_{i,j=1}^n(|h_{i,j}|^2+|(h^{-1})_{i,j}|^2),\qquad h\in G_n.
\end{align*}
For $x\in\R_{>0}^k$ we clearly have
\begin{align*}
\|a_x\|^N \leq Ce^{D\cdot x}, \qquad \forall x\in\R_{>0}^k,
\end{align*}
for some $C>0$ and $D=(D_1,\ldots,D_k)\in\R^k$. For $q=Cq_2$,
\begin{align*}
|l(\pi(\mathfrak{s}(a_x))\xi)| \leq q(\xi)e^{D\cdot x}, \qquad \forall x\in\R_{>0}^k,\xi\in V.
\end{align*}
We now argue as in \cite[Section 4.3.5]{Wal88}. More precisely, applying the recursive argument in the proof of Theorem~\ref{thm:AsymptoticExpansionSmoothMatrixCoefficientsArchimedean} shows that after possibly multiplying with a power of $(1+|x|)$, we can replace $D$ by $D-e_\ell$ for any $1\leq\ell\leq k$ as long as $D_\ell-1\geq-\Lambda_{\pi,\ell}$. Repeating this shows that $D$ can actually be replaced by $-\Lambda_\pi$ after multiplying with a certain power of $(1+|x|)$, and hence we obtain \eqref{eq:BoundShalikaFunctional}.
\end{proof}
\begin{remark}\label{rem:ExpansionsForXkleq0}
As stated, the estimates for $l(\pi(\mathfrak{s}(a_x))\xi)$ in Theorems~\ref{thm:AsymptoticExpansionSmoothMatrixCoefficientsArchimedean} and \ref{thm:AsymptoticExpansionShalikaFunctionsArchimedean} only hold for $x\in\R_{>0}^k$. However, one can also expand $l(\pi(\mathfrak{s}(a_x))\xi)$ for $x_1,\ldots,x_{k-1}\geq0$ and $x_k\leq0$ into an exponential power series, and the remainder terms are at least bounded by $e^{-Nx_k}$ for some $N\geq0$. To see this, first note that
$l(\pi(\mathfrak{s}(a_x))\xi) = l(\pi(\mathfrak{s}(a_{x'}))\pi(\mathfrak{s}(e^{-x_kH_k}))\xi)$ where $x'=(x_1,\ldots,x_{k-1},0)$. Then apply the expansion with $\xi$ replaced by $\pi(\mathfrak{s}(e^{-x_kH_k}))\xi$,
and use the fact that for each seminorm $q$ there exists another seminorm $q'$ and $N\geq0$ such that
$q(\pi(\mathfrak{s}(e^{-x_kH_k}))\xi)\leq q'(\xi)e^{-Nx_k}$ for all $x_k\leq0$ (see the proof of Theorem~\ref{thm:AsymptoticExpansionShalikaFunctionsArchimedean}).
\end{remark}

\subsection{The unramified metaplectic Shalika function}\label{section:The metaplectic unramified Shalika formula}
In this section we develop an explicit formula for the unramified metaplectic Shalika function.
Our arguments closely follow those of Sakellaridis \cite{Yia}, who established this formula
in the non-metaplectic setting. We use the definitions and results of \S~\ref{section:Metaplectic Hironaka Theory},
in particular $|2|=1$ and we have the canonical splitting $\kappa$ of $K$ as in \S~\ref{subsection:unramified representations}.

Assume $n=2k$. Let $\mathrm{I}(\chi)$ be a genuine unramified principal series representation,
$\chi=\chi_{\underline{s}}$ (since $n$ is even, $\gamma$ can be ignored). Let $\psi$ and $\psi'$ be a pair of nontrivial additive characters of $F$, and assume they are both unramified, i.e., their conductor is $\mathcal{O}$.
Assuming that $\underline{s}$ is in a ``suitable general position", we construct a metaplectic $(\psi',\psi)$-Shalika functional (see \S~\ref{subsection:metaplectic Shalika functional and model} for the definition) on $\mathrm{I}(\chi)$, and provide an explicit formula for the value of such a functional on the normalized unramified element of $\mathrm{I}(\chi)$.

Identify $\Delta_{G_n}$ with the pairs $(i,i+1)$, $1\leq i<n$, and put $\alpha_i=(i,i+1)$, then $\Sigma_{G_n}$ consists of the pairs $(i,j)$ with $1\leq i\ne j\leq n$.
We will need a few simple auxiliary results for computing conjugations.

\begin{claim}\label{claim:conjugation W and T}
Let $\mathfrak{w}\in\mathfrak{W}$ and $t\in T_n$. Assume that $\mathfrak{w}$ is the representative of $w\in W$. Then
$\rconj{\mathfrak{s}(\mathfrak{w})}\mathfrak{s}(t)=\sigma(\mathfrak{w},t)\mathfrak{s}(\rconj{\mathfrak{w}}t)$, where
\begin{align*}
\sigma(\mathfrak{w},t)=\prod_{\setof{(i,j)=\alpha>0}{w\alpha<0}}(-t_j,t_i).
\end{align*}
\end{claim}
\begin{proof}
Since $\mathfrak{s}(\mathfrak{w})\mathfrak{s}(t)\mathfrak{s}(\mathfrak{w}t)^{-1}=\sigma(\mathfrak{w},t)$,
\begin{align*}
\mathfrak{s}(\mathfrak{w})\mathfrak{s}(t)\mathfrak{s}(\mathfrak{w})^{-1}=\sigma(\mathfrak{w},t)\mathfrak{s}(\mathfrak{w}t)\mathfrak{s}(\mathfrak{w})^{-1}
=\sigma(\mathfrak{w},t)\mathfrak{s}(\rconj{\mathfrak{w}}t\mathfrak{w})\mathfrak{s}(\mathfrak{w})^{-1}=
\sigma(\mathfrak{w},t)\mathfrak{s}(\rconj{\mathfrak{w}}t).
\end{align*}
For the last equality we used the fact that $\sigma(t',\mathfrak{w})=1$ for any $t'\in T_n$ (\cite[\S~3, Theorem~7]{BLS}).
The value of $\sigma(\mathfrak{w},t)$ was given in \cite[\S~3, Lemma~3]{BLS}.
\end{proof}
For any $t,t'\in T_n$, according to the computation of $\sigma(t,t')$ in \cite[\S~3, Lemma~1]{BLS},
\begin{align}\label{eq:commutator on torus}
[\mathfrak{s}(t),\mathfrak{s}(t')]=\mathfrak{s}(t)\mathfrak{s}(t')\mathfrak{s}(t)^{-1}\mathfrak{s}(t')^{-1}=\prod_{i<j}(t_i,t'_j)_2(t'_i,t_j)_2.
\end{align}
\begin{remark}\label{remark:computing conjugation t by w}
Let $w$ be a permutation matrix. One can find unique $\mathfrak{w}\in\mathfrak{W}$ and $t_0\in T_n$ such that $w=t_0\mathfrak{w}$, where the diagonal coordinates of $t_0$ are $\pm1$. Then for any $t\in T_n$,
$\rconj{\mathfrak{s}(w)}\mathfrak{s}(t)=\rconj{\mathfrak{s}(t_0)}(\rconj{\mathfrak{s}(\mathfrak{w})}\mathfrak{s}(t))$,
which can be computed using Claim~\ref{claim:conjugation W and T} and \eqref{eq:commutator on torus}.
\end{remark}
\begin{corollary}\label{corollary:computing conjugation t by w for t in maximal abelian or in the center}
Let $w$ be a permutation matrix and write $w=t_0\mathfrak{w}$ as in Remark~\ref{remark:computing conjugation t by w}.
For any $t\in T_{n,*}$, $\rconj{\mathfrak{s}(w)}\mathfrak{s}(t)=\rconj{\mathfrak{s}(\mathfrak{w})}\mathfrak{s}(t)$. If
$\mathfrak{s}(t)\in C_{\widetilde{T}_n}$,
$\rconj{\mathfrak{s}(w)}\mathfrak{s}(t)=\mathfrak{s}(\rconj{\mathfrak{w}}t)$.
\end{corollary}
The following theorem shows that when the inducing data is in a certain general position, which does not preclude reducibility, the space of metaplectic Shalika functionals is at most one-dimensional. Ash and Ginzburg \cite[Lemma~1.7]{AG} proved this for $G_n$. Our proof follows theirs, with certain modifications.

For $1\leq i<j\leq n$ and $x\in F^*$, let $t_{i,j}(x)=\mathfrak{s}(\diag(I_{i-1},x^2,I_{j-i-1},x^2,I_{n-j}))\in C_{\widetilde{T}_n}$.

\begin{theorem}\label{theorem:unramified Shalika functional 1 dim}
Assume that $\chi$ satisfies the following conditions.
\begin{enumerate}[leftmargin=*]
\item\label{item:uniqueness item 1}
$\chi(t_{i,j})\ne1$ for all $1\leq i<j\leq k$.
\item\label{item:uniqueness item 2}
There exists a permutation $\tau$ on $\{1,\ldots,k\}$ such that
$\chi(t_{i,k+j})\ne1$ for all $1\leq i, j\leq k$ such that $j\ne\tau(i)$.
\end{enumerate}
The space of metaplectic Shalika functionals on $\mathrm{I}(\chi)$ is at most one-dimensional. Furthermore, it is zero dimensional unless $\chi(t_{i,k+\tau(i)})=1$ for all $1\leq i\leq k$.
\end{theorem}
\begin{proof}
For the first assertion we have to show
\begin{align}\label{eq:dimension of homspace to bound}
\Dim\Hom_{\widetilde{G}_n}(\mathrm{I}(\chi),\Ind_{\widetilde{G}_k^{\triangle}U_k}^{\widetilde{G}_n}(\gamma_{\psi'}\otimes\psi))\leq1.
\end{align}
For $g\in\rmodulo{\lmodulo{G_k^{\triangle}U_k}{G_n}}{B_{n,*}}$,
denote
\begin{align*}
\mathcal{H}(g)=\Hom_{(\widetilde{G}_k^{\triangle}U_k)^{\mathfrak{s}(g)^{-1}}}(\rconj{\mathfrak{s}(g)}(\gamma_{\psi'}^{-1}\otimes\psi^{-1})\otimes\chi,\delta),
\end{align*}
where $(\widetilde{G}_k^{\triangle}U_k)^{\mathfrak{s}(g)^{-1}}=
\rconj{\mathfrak{s}(g)^{-1}}(\widetilde{G}_k^{\triangle}U_k)\cap \widetilde{B}_{n,*}$;
for a representation $\pi$ of
$\widetilde{G}_k^{\triangle}U_k$, $\rconj{\mathfrak{s}(g)}\pi$ is the
representation of $(\widetilde{G}_k^{\triangle}U_k)^{\mathfrak{s}(g)^{-1}}$ on
the space of $\pi$ given by
$\rconj{\mathfrak{s}(g)}\pi(x)=\pi(\rconj{\mathfrak{s}(g)}x)$;
$\delta(x)=\delta_{B_n}^{-1/2}(x)\delta_{g}(x)$
is a representation of
$(G_k^{\triangle}U_k)^{g^{-1}}=\rconj{g^{-1}}(G_k^{\triangle}U_k)\cap B_{n,*}$ on $\C$, lifted to a nongenuine representation of the cover, where
$\delta_{g}$ is the modulus character of $(G_k^{\triangle}U_k)^{g^{-1}}$.
We recall that $\chi$ is a genuine character of $\widetilde{T}_{n,*}$ which is trivial on $\widetilde{T}_n\cap\kappa(K)$ (see the discussion following
Remark~\ref{remark:analog of Casselman 2.4 is weaker}). Also
$\rconj{\mathfrak{s}(g)}(\gamma_{\psi'}^{-1}\otimes\psi^{-1})\otimes\chi$ acts trivially on $\mu_2$.

By \cite{AG} (proof of Proposition~1.3) $G_n=\bigcup_wG_k^{\triangle}U_kwB_n$ where $w$ varies over the permutation matrices. Hence we may assume $g=wt$ for a permutation matrix $w$ and $t\in T_n$.
Set $\omega_0^{-1}=(\delta_{i,\tau(j)})_{1\leq i,j\leq k}$,
$\omega^{-1}=\left(\begin{smallmatrix}&\omega_0^{-1}\\I_k\end{smallmatrix}\right)$.

By definition $\Dim\mathcal{H}(g)\leq1$ for any $g$.
According to the Bruhat Theory (see e.g., \cite[Theorems~1.9.4 and 1.9.5]{Silb}), to prove \eqref{eq:dimension of homspace to bound} it is enough to show
\begin{align}\label{eq:Hom(g) tw with w not w_k}
&\mathcal{H}(g)=0,\qquad\forall g=wt\text{ such that }G_k^{\triangle}U_kwB_n\ne G_k^{\triangle}U_k\omega^{-1}B_n, \\
\label{eq:Hom(g) tw_k with t not 1}
&\mathcal{H}(\omega^{-1}t)=0,\qquad\forall t\text{ such that }G_k^{\triangle}U_k\omega^{-1}tB_{n,*}\ne G_k^{\triangle}U_k\omega^{-1}B_{n,*}.
\end{align}
\begin{claim}\label{claim:Hom(g) tw with w not w_k}
Equality~\eqref{eq:Hom(g) tw with w not w_k} holds.
\end{claim}
\begin{remark}
In \cite{AG} this claim already completed the proof of uniqueness, because in $G_n$ one could a priori take
$g=w$.
\end{remark}
Before proving the claim, let us complete the proof of the theorem. We establish \eqref{eq:Hom(g) tw_k with t not 1}. Put $g=\omega^{-1}t$.
The condition on $t$ implies that we can assume $t=\diag(t_1,\ldots,t_k,I_k)$ and for some $1\leq i\leq k$, $t_i\notin F^{*2}\mathcal{O}^*$. Let $i$ be minimal with this property. Note that for any $y\in F^*\setdifference F^{*2}\mathcal{O}^*$, there is $x\in\mathcal{O}^*$ such that $(y,x)_2\ne1$. Then by \eqref{eq:commutator on torus} there is $x\in\mathcal{O}^*$ such that for $d=\diag(I_{i-1},x,,I_{k-i})^{\triangle}$,
\begin{align*}
\rconj{\mathfrak{s}(t)^{-1}}\mathfrak{s}(\rconj{\omega}d)=(t_i,x)_2\mathfrak{s}(\rconj{\omega}d)=-\mathfrak{s}(\rconj{\omega}d).
\end{align*}
Then we obtain
\begin{align*}
\rconj{\mathfrak{s}(g)^{-1}}\mathfrak{s}(d)
=\rconj{\mathfrak{s}(t)^{-1}}(\rconj{\mathfrak{s}(\omega)}\mathfrak{s}(d))=
\rconj{\mathfrak{s}(t)^{-1}}(\mathfrak{s}(\rconj{\omega}d))=-\mathfrak{s}(\rconj{\omega}d).
\end{align*}
Here for the second equality we used Corollary~\ref{corollary:computing conjugation t by w for t in maximal abelian or in the center} and Claim~\ref{claim:conjugation W and T}. Clearly
$\rconj{\mathfrak{s}(g)^{-1}}\mathfrak{s}(d)\in(\widetilde{G}_k^{\triangle}U_k)^{\mathfrak{s}(g)^{-1}}$ and we see that
$(\gamma_{\psi'}\otimes\psi)(\mathfrak{s}(d))=1$ (because $\gamma_{\psi'}$ is trivial on $\mathcal{O}^*$) and
$\chi(\rconj{\mathfrak{s}(g)^{-1}}\mathfrak{s}(d))=-\chi(\mathfrak{s}(\rconj{\omega}d))=-1$, because $\chi$ is trivial on $\mathfrak{s}(T_n\cap K)$. But $\delta_{B_n}$ and $\delta_g$ are modulus characters and as such, are trivial on $\rconj{g^{-1}}d$. Thus $\mathcal{H}(\omega^{-1}t)=0$.

Therefore the space of metaplectic Shalika functionals on $\mathrm{I}(\chi)$ is at most one-dimensional. For the second assertion of the theorem
consider $\mathcal{H}(\omega^{-1})$. 
In this case
\begin{align*}
(G_k^{\triangle}U_k)^{\omega}=\iota((T_k\cap K)T_k^2)\ltimes \prod_{\setof{\alpha\in\Sigma_{G_k}^+}{\tau^{-1}\alpha>0}}\iota(N_{\alpha}),
\end{align*}
where $N_{\alpha}$ on the \rhs\ is a subgroup of $G_k$ and for any $c\in G_k$, $\iota(c)=\diag(c,\rconj{w_0}c)$.
For $d=x^{\triangle}$ with $x\in T_k$ such that its coordinates lie in $F^{*2}\mathcal{O}^*$,
$(\gamma_{\psi'}\otimes\psi)(\mathfrak{s}(d))=1$. Also $\mathfrak{s}$ is a homomorphism of $\iota((T_k\cap K)T_k^2)$. Hence
\begin{align*}
\mathcal{H}(\omega^{-1})=\Hom_{\mathfrak{s}(\iota((T_k\cap K)T_k^2))}(\chi,\delta_{B_n}^{-1/2}\delta_{\omega^{-1}}).
\end{align*}
For all $1\leq i\leq k$, $t_{i,k+\tau(i)}\in \mathfrak{s}(\iota((T_k\cap K)T_k^2))$ and a direct computation of $\delta_{\omega^{-1}}$, exactly as in the proof
of Claim~\ref{claim:Hom(g) tw with w not w_k} below (only roots of the form $(l,i)$ with $1\leq l<i$, and $(i,l)$ with $i<l\leq k$, may contribute), shows
\begin{align*}
\delta_{\omega^{-1}}(t_{i,k+\tau(i)})=|x|^{2(-|\setof{1\leq l<i}{\tau(l)<\tau(i)}|+k-\tau(i)-|\setof{1\leq l<i}{\tau(l)>\tau(i)}|)}
=|x|^{2(k-\tau(i)-i+1)}=\delta^{1/2}_{B_n}(t_{i,k+\tau(i)}).
\end{align*}
Hence there is no metaplectic Shalika functional on $\mathrm{I}(\chi)$, unless $\chi(t_{i,k+\tau(i)})=1$.
\end{proof}
\begin{proof}[Proof of Claim~\ref{claim:Hom(g) tw with w not w_k}]
As mentioned above, the arguments are similar to \cite[Lemma~1.7]{AG}. One difference is that we consider a general $\tau$, whereas in \cite{AG} the permutation was taken (for simplicity) to be the identity. This is only a minor complication.

Let $g=wt$ be as prescribed and write $w=(\delta_{i,j_i})_{1\leq i,j\leq n}$. 
For each $1\leq l\leq k$, if $j_l<j_{k+l}$, then for $U=N_{(l,k+l)}$ (the root subgroup of $(l,k+l)$),
$\rconj{w^{-1}}U=N_{j_l,j_{k+l}}<B_{n,*}$ and $\psi|_{U}\ne1$, and since $\delta|_U=1$, $\mathcal{H}(g)=0$. Hence
$j_{k+l}<j_l$ for all $1\leq l\leq k$.

We may assume $j_{\tau(1)}<\ldots<j_{\tau(k)}$.
Suppose $j_{\tau(1)}\leq k$ and put $m=\tau(1)$. For $x\in F^*$ set
\begin{align*}
d=\diag(I_{m-1},x^2,I_{k-m})^{\triangle}\in G_k^{\triangle}\cap B_{n,*}.
\end{align*}
Then $\rconj{\mathfrak{s}(g)^{-1}}\mathfrak{s}(d)\in (\widetilde{G}_k^{\triangle}U_k)^{\mathfrak{s}(g)^{-1}}$. Since $\mathfrak{s}(d)\in C_{\widetilde{T}_n}$,
by Corollary~\ref{corollary:computing conjugation t by w for t in maximal abelian or in the center}
\begin{align*}
\rconj{\mathfrak{s}(g)^{-1}}\mathfrak{s}(d)=\rconj{\mathfrak{s}(w)^{-1}}\mathfrak{s}(d)=\mathfrak{s}(\rconj{w^{-1}}d)=t_{j_{k+m},j_{m}}.
\end{align*}
On the one hand $(\gamma_{\psi'}\otimes\psi)(\mathfrak{s}(d))=1$ and 
$\chi(\rconj{\mathfrak{s}(g)^{-1}}\mathfrak{s}(d))=\chi(t_{j_{k+m},j_{m}})\ne1$ 
by assumption \eqref{item:uniqueness item 1} (because $j_{k+m}<j_{m}\leq k$).
On the other hand, we show $\delta(\rconj{g^{-1}}d)=1$. 

To compute $\delta(\rconj{g^{-1}}d)$, we count the roots from $U$ and $G_k^{\triangle}$ that are conjugated by $w$ to $N_n$. In the conjugation row $l$ moves into row $j_l$ and similarly for the columns. For the computation we only care about rows and columns $m$ and $k+m$. Our assumption on $\tau$ implies $\setof{l}{j_l<j_{m}}\subset\setof{l}{l>k}$.

First consider roots from $U$. Precisely $j_m-1$ roots from row $m$ are conjugated to columns smaller than $j_{m}$. The other $k-j_m+1$ roots are conjugated to columns bigger than $j_{m}$, so they do not contribute. Row $m+k$ and column $m$ do not contain any root of $U$, so do not contribute. Regarding column $m+k$, since $j_{m+k}<j_m$, the last $j_{m+k}-1$ rows from column $m+k$ are conjugated into the first $j_{m+k}-1$ rows of column $j_{m+k}$, and these
do not contain roots. Summarizing, the roots from $U$ contribute $|x|^{2(k-j_m+1)}$.

Next consider $G_k^{\triangle}$. In general, a root $(i,j)$ with $1\leq i\ne j\leq k$ contributes only if both the $(i,j)$-th and $(i+k,j+k)$-th coordinates are conjugated into $N_n$. The roots from row $m$ ($k-1$ roots) are conjugated to columns greater than $j_{m}$. The last $j_{m}-1$ columns are conjugated to the first $j_{m}-1$ columns, so only $k-1-(j_{m}-1)$ roots from row $m$ contribute. Regarding row $m+k$, the first $k-j_{m}$ roots were already accounted for. But $j_{m}-j_{m+k}$ roots from the last $j_{m}$ columns are conjugated into columns $j_{m+k}+1,\ldots,j_{m}$, and also contribute because their copies from row $k$ were conjugated into positive roots. The rows conjugated into the first $j_{m}-1$ rows are from the last $k$ rows. Therefore, roots from column $m$ are not conjugated into positive roots in column $j_{m}$. Hence their copies in column $m+k$ cannot contribute as well.
Altogether the roots from $G_k^{\triangle}$ contribute $|x|^{2(k-j_{m})+2(j_{m}-j_{m+k})}=
|x|^{2(k-j_{m+k})}$.

In total, $\delta(\rconj{g^{-1}}d)=|x|^{2n-2(j_{m}+j_{k+m}-1)}=\delta^{1/2}_{B_n}(\rconj{g^{-1}}d)$, hence
$\delta(\rconj{g^{-1}}d)=1$.

It follows that $\mathcal{H}(g)=0$, unless $j_{\tau(l)}>k$ for all $1\leq l\leq k$. In the latter case we can assume $w=\left(\begin{smallmatrix}&w'\\I_k\end{smallmatrix}\right)$ (so, the assumption
$j_{\tau(1)}<\ldots<j_{\tau(k)}$ is no longer valid). Let $1\leq i_1,\ldots,i_k\leq k$ be such that
$\tau(i_l)=l$ for all $1\leq l\leq k$. We prove by induction on $l$, that $j_{i_l}=k+l$.

Suppose otherwise $j_{i}=k+1$ with $i\ne i_1$. Proceeding as above take
\begin{align}\label{eq:d in proof of claim last}
d=\diag(I_{i-1},x^2,I_{k-1},x^2,I_{k-i})\in G_k^{\triangle}\cap B_{n,*},
\end{align}
then
$\rconj{\mathfrak{s}(g)^{-1}}\mathfrak{s}(d)
=\mathfrak{s}(\rconj{w^{-1}}d)
=t_{i,k+1}$.
We see that $(\gamma_{\psi'}\otimes\psi)(\mathfrak{s}(d))=1$, $\chi(\mathfrak{s}(\rconj{w^{-1}}d))\ne1$ by assumption \eqref{item:uniqueness item 2} (because $\tau(i)\ne\tau(i_1)=1$), while $\delta(\rconj{g^{-1}}d)=1$. In fact, we compute $\delta$ as above. Here only roots from $G_k^{\triangle}$ may contribute. On row $i$, the first $i-1$ roots do not contribute, because their copies are mapped into negative roots. On column $i$, the positive roots do not contribute because they are mapped into rows greater than $k+1$ of column $k+1$. This leaves the last $k-i$ roots of row $i$, which are mapped into row $k+1$ and columns greaater than $k+1$, and their copies are mapped into positive roots. Altogether $\delta_g(\rconj{g^{-1}}d)=|x|^{2(k-i)}=\delta^{1/2}_{B_n}(\rconj{g^{-1}}d)$. Thus $\mathcal{H}(g)=0$, unless $j_{i_1}=k+1$.

Assuming for all $1\leq l'<l$, $j_{i_{l'}}=k+l'$, we show $j_{i_{l}}=k+l$. Suppose $j_{i}=k+l$ for $i\ne i_{l}$. Then
$i\ne i_{l'}$ for $1\leq l'\leq l$. We take $d$ as in \eqref{eq:d in proof of claim last}. Then
$\rconj{\mathfrak{s}(g)^{-1}}\mathfrak{s}(d)=t_{i,k+l}$, $\chi(\mathfrak{s}(\rconj{w^{-1}}d))\ne1$ and we have to show
$\delta(\rconj{g^{-1}}d)=1$. Indeed, as in the case $l=1$, we need only consider the positive roots from
$G_k^{\triangle}$, on the $i$-th row and column. Their copies are conjugated to positive roots. According to the induction hypothesis, for each $1\leq l'<l$, the $(i_{l'},i)$-th (resp. $(i,i_{l'})$-th) entry is conjugated into a positive root if and only if $i_{l'}<i$ (resp. $i_{l'}>i$). This implies that the contribution consists of $|\setof{1\leq l'<l}{i_{l'}<i}|$ roots from column $i$ (contributing negative powers of $|x|$) and $k-i-|\setof{1\leq l'<l}{i_{l'}>i}|$ from row $i$ (contributing positive powers). Thus
$\delta_g(\rconj{g^{-1}}d)=|x|^{2(k-i-l+1)}=\delta^{1/2}_{B_n}(\rconj{g^{-1}}d)$ and $\mathcal{H}(g)=0$, unless $j_{i_l}=k+l$. The proof is complete, since $j_l=k+\tau(l)$ for all $1\leq l\leq k$, i.e.,
$w'=\omega_0^{-1}$ and $w=\omega^{-1}$.
\end{proof}
\begin{remark}\label{remark:permuting last k coordinates}
Recall that $\chi=\chi_{\underline{s}}$, $\underline{s}\in\C^n$
(see \S~\ref{section:Casselman's basis and a result of Hironaka}). Assume $q^{2s_i}\ne q^{\pm2s_j}$ for all
$1\leq i<j\leq k$ and $q^{2s_i}=q^{-2s_{k+i}}$ for all $1\leq i\leq k$, perhaps after a reordering of $(s_{k+1},\ldots,s_n)$. 
By Theorem~\ref{theorem:unramified Shalika functional 1 dim}, $\mathrm{I}(\chi)$ and $\mathrm{I}(\chi^{-1})$ admit at most one metaplectic Shalika functional (up to scalar multiplication). 
\end{remark}
\begin{remark}\label{remark:why we do not fix tau}
If $\mathrm{I}(\chi)$ is irreducible, then we may have simply fixed one permutation $\tau$ in the statement of the theorem. However, we will also apply the theorem to reducible representations.
\end{remark}

We define a functional satisfying the prescribed equivariance properties \eqref{eq:metaplectic Shalika def}. It will be defined using an integral, which is absolutely convergent in a certain cone in $\chi$. Then we extend it to a meromorphic function, thereby obtaining a functional for all $\chi$.

Following Sakellaridis \cite{Yia}, we realize the metaplectic Shalika functional using a similar functional on a group $H$, which is a twist
of the Shalika group. Working with $H$ instead of the Shalika group was advantageous in \cite{Yia} for computational reasons. Let $J_k$ be the permutation matrix
\begin{align*}
J_k=\left(\begin{array}{ccc}0&&1\\&\iddots&\\1&&0\end{array}\right)\in G_k.
\end{align*}
Put
\begin{align*}
H=\rconj{\omega}(G_k^{\triangle}U_k)=\left\{\left(\begin{array}{cc}c&0\\J_kcu&J_kcJ_k\end{array}\right):c\in G_k\right\},\qquad \omega=\left(\begin{array}{cc}&I_k\\J_k\end{array}\right).
\end{align*}
We form a section $\mathfrak{h}$ of $H$ by
\begin{align*}
\mathfrak{h}(h)=\mathfrak{s}(\omega)\mathfrak{s}(\rconj{\omega^{-1}}h)\mathfrak{s}(\omega)^{-1}.
\end{align*}
For $h\in H$, write
$\rconj{\omega^{-1}}h=c_h^{\triangle}u_h$ with $c_h\in G_k$ and $u_h\in U_k$. The mapping $h\mapsto c_h$ is a group epimorphism,
the mapping $h\mapsto u_h$ is onto $U$, and $h\mapsto\psi(u_h)$ is a character of $H$ (trivial on $\rconj{\omega}G_k^{\triangle}$).
According to \eqref{eq:formula for s on triangle}, the section $\mathfrak{h}$ satisfies
 \begin{align}\label{eq:the section h is almost a splitting}
\mathfrak{h}(hh')=(\det c_h,\det c_{h'})_2\mathfrak{h}(h)\mathfrak{h}(h').
\end{align}
A functional $l$
on a genuine $\pi\in\Alg{\widetilde{G}_n}$ is called a (metaplectic) $(H,\psi',\psi)$-functional if for
any $\xi$ in the space of $\pi$ and $h\in H$,
\begin{align}\label{eq:equivariance props of a H psi' psi functional}
l(\pi(\mathfrak{h}(h))\xi)=\gamma_{\psi'}(\det c_{h})\psi(u_h)l(\xi).
\end{align}
The definitions of $H$ and $\mathfrak{h}$ imply the following relation between these functionals and
metaplectic Shalika functionals.
\begin{claim}\label{claim:H and Shalika functionals}
Let $\pi\in\Alg{\widetilde{G}}_n$ be genuine.
The dimension of the space of $(H,\psi',\psi)$-functionals on $\pi$ is equal to the dimension of
metaplectic $(\psi',\psi)$-Shalika functionals.
\end{claim}
\begin{proof}
If $l$ is an $(H,\psi',\psi)$-functional on $\pi$,
$\xi\mapsto l(\pi(\mathfrak{s}(\omega))\xi)$ is a $(\psi',\psi)$-Shalika functional. If
$l$ is a $(\psi',\psi)$-Shalika functional,
$\xi\mapsto l(\pi(\mathfrak{s}(\omega)^{-1})\xi)$ is an $(H,\psi',\psi)$-functional.
\end{proof}

With Claim~\ref{claim:H and Shalika functionals} in mind, we construct $(H,\psi',\psi)$-functionals.
Our main tool will be the metaplectic analog of Hironaka's theorem, Corollary~\ref{corollary:analog of Hironaka theory with phi I functions}. To preserve the notation, we construct the functional on $\mathrm{I}(\chi^{-1})$.

The group $H$ is closed and unimodular.
Sakellaridis proved \cite[Lemma~3.1]{Yia} that $B_nH$ is Zariski open in $G_n$. Therefore $B_nH$ is dense and open in $G_n$, making it convenient for integration formulas. Here we are forced to work with $B_{n,*}H$, which is
in particular not dense. Next we describe certain structural results involving $B_{n,*}$ and $H$, which will
enable us to adapt the integration formulas from \cite{Yia} to the metaplectic setting.
\begin{claim}\label{claim:coset decomposition B*H in BnH}
$B_nH=\coprod_{t\in \lmodulo{T_{k,*}}{T_k}}\diag(t,I_k)B_{n,*}H$.
\end{claim}
\begin{proof}
Since $H$ contains the subgroup $\{\diag(x_1,\ldots,x_k,x_k,\ldots,x_1):x_i\in F^*\}$, it is clear that any $bh\in B_nH$ is contained
in one of the cosets on the \rhs. The union is disjoint because if $\diag(t,I_k)B_{n,*}H=\diag(t',I_k)B_{n,*}H$, using the fact that
$B_{n,*}$ is normal in $B_n$ leads to $\diag(t{t'}^{-1},I_k)b_*\in H$ for some $b_*\in B_{n,*}$. Let the diagonal of $b_*$ be
$\diag(c_1,\ldots,c_n)$, $c_i\in F^{*2}\mathcal{O}^*$. Then the diagonal of $\diag(t{t'}^{-1},I_k)b_*$ takes the form
$\diag(x_1,\ldots,x_k,x_k,\ldots,x_1)$ and we must have $t_i{t'_i}^{-1}c_i=x_i=c_{n-i+1}$ for all $1\leq i\leq k$,
whence $t_i{t'_i}^{-1}\in F^{*2}\mathcal{O}^*$ for all $i$ and $tT_{k,*}=t'T_{k,*}$.
\end{proof}
We will use the following geometric results from \cite[Lemma~5.1]{Yia}:
\begin{align}\label{eq:Iwahori- without one root contained in BcapK times HcapK}
&N_{\alpha}^-(\varpi\mathcal{O})\subset(T_{n}\cap K)N_{\alpha}(\mathcal{O})N_{\alpha'}(\mathcal{O})(H\cap K),\qquad\forall\alpha>0,\\
\nonumber
&\mathcal{I}^-\subset(B_{n}\cap K)(H\cap K).
\end{align}
Here for $(i,j)=\alpha>0$, $\alpha'=\alpha$ if $i\leq k<j$ otherwise $\alpha'=(n-j+1,n-i+1)$.
The latter inclusion along with \eqref{eq:reformulation of Iwahori} and the fact that $B_{n}\cap K<B_{n,*}$ implies
\begin{align}\label{eq:Iwahori contained in B*H}
&\mathcal{I}\subset(B_{n}\cap K)(H\cap K)\subset B_{n,*}H.
\end{align}
\begin{claim}\label{claim:B_{n,*}H is open}
The set $B_{n,*}H$ is open (in $G_n$).
\end{claim}
\begin{proof}
According to \eqref{eq:Iwahori contained in B*H}, for any $b_*\in B_{n,*}$ and $h\in H$, $b_*\mathcal{I}h\subset B_{n,*}H$. The set $b_*\mathcal{I}h$ is open (because $\mathcal{I}$ is)
and contains $b_*h$. Therefore $B_{n,*}H$ is open.
\end{proof}

The following two claims relate the sections $\mathfrak{h}$, $\mathfrak{s}$ and $\kappa$.
\begin{claim}\label{claim:h and kappa agree on H cap K}
The sections $\mathfrak{s}$ and $\kappa$ agree on $(G_k^{\triangle}U_k)\cap K$. Consequently,
$\mathfrak{h}$ and $\kappa$ agree on $H\cap K$.
\end{claim}
\begin{proof}
First we compare $\mathfrak{s}$ and $\kappa$. We have $(G_k^{\triangle}U_k)\cap K=G_k(\mathcal{O})^{\triangle}U_k(\mathcal{O})$,
where $G_k(\mathcal{O})^{\triangle}=\{c^{\triangle}:c\in G_k(\mathcal{O})\}$.
According to \eqref{eq:formula for s on triangle} and because $(\mathcal{O}^*,\mathcal{O}^*)_2=1$, $\mathfrak{s}$ is a splitting of $G_k(\mathcal{O})^{\triangle}$ and thereby of
$G_k(\mathcal{O})^{\triangle}U_k(\mathcal{O})$. The section $\kappa$ is also a splitting of the latter. Since
$\mathfrak{s}|_{N_n\cap K}=\kappa|_{N_n\cap K}$, it is enough to show that $\mathfrak{s}$ and $\kappa$ agree on $G_k(\mathcal{O})^{\triangle}$. The mapping $c\mapsto\mathfrak{s}(c^{\triangle})\kappa(c^{\triangle})^{-1}$ is
a homomorphism $G_k(\mathcal{O})\rightarrow\mu_2$, which is trivial on $T_k\cap G_k(\mathcal{O})$ because
$\mathfrak{s}|_{T_n\cap K}=\kappa|_{T_n\cap K}$. Moreover, there is only one homomorphism $\mathrm{SL}_k(\mathcal{O})\rightarrow\mu_2$, the trivial one (when $|2|=1$ and $q>3$).
Thus $c\mapsto\mathfrak{s}(c^{\triangle})\kappa(c^{\triangle})^{-1}$ is the identity.

Next we have $H\cap K=\rconj{\omega}(G_k(\mathcal{O})^{\triangle}U_k(\mathcal{O}))$. The section $\mathfrak{h}$ is a splitting of $H\cap K$ by definition and because $\mathfrak{s}$ is a splitting of $(G_k^{\triangle}U_k)\cap K$. The assertion for $\mathfrak{h}$ now follows from the previous result because $\kappa(\omega)=\mathfrak{s}(\omega)$.
\end{proof}
\begin{remark}
One can check that for $u\ne0$, $\mathfrak{h}\left(\begin{smallmatrix}1&0\\u&1\end{smallmatrix}\right)=(u,u)_2\mathfrak{s}\left(\begin{smallmatrix}1&0\\u&1\end{smallmatrix}\right)$, whence
$\mathfrak{h}$ and $\mathfrak{s}$ do not coincide, even on $\rconj{\omega}U_k(\mathcal{O})$.
\end{remark}
Next we consider

\begin{align*}
H\cap B_{n,*}=\{\diag(t_1,\ldots,t_k,t_k,\ldots,t_1):t_i\in F^{*2}\mathcal{O}^*\}.
\end{align*}
It is closed and unimodular.
\begin{claim}\label{claim:h and s agree on H cap B_n*}
The sections $\mathfrak{h}$ and $\mathfrak{s}$ agree on
$H\cap B_{n,*}$.
\end{claim}
\begin{proof}
Let $t=\diag(t_1,\ldots,t_k,t_k,\ldots,t_1)\in H\cap B_{n,*}$. Then $\mathfrak{h}(t)=\rconj{\mathfrak{s}(\omega)}\mathfrak{s}(\rconj{\omega^{-1}}t)$. We compute this conjugation as explained in
Remark~\ref{remark:computing conjugation t by w}. Put $\omega=t_0\mathfrak{w}$, where $t_0\in T_n\cap K$ and $\mathfrak{w}\in\mathfrak{W}$. Then $\sigma(\mathfrak{w},t)=1$ because $(\blank,\blank)_2$ is trivial on $F^{*2}\mathcal{O}^*\times F^{*2}\mathcal{O}^*$.
Also $\mathfrak{s}(t_0)$ and $\mathfrak{s}(t)$ commute because $t_0,t\in T_{n,*}$. Thus
$\rconj{\mathfrak{s}(\omega)}\mathfrak{s}(\rconj{\omega^{-1}}t)=\mathfrak{s}(t)$.
\end{proof}

Let $dh$ be a Haar measure on $H$. Consider the functional $l_H$ on
$\mathcal{S}^{\mathrm{gen}}(\widetilde{H})$ given by
\begin{align*}
l_H(f)=\int_Hf(\mathfrak{h}(h))\gamma_{\psi'}^{-1}(\det c_h)\psi^{-1}(u_h)\, dh.
\end{align*}
This integral is absolutely convergent because $|f|\in \mathcal{S}^{\mathrm{gen}}(H)$ ($|f|$ is non-genuine and $p:\widetilde{H}\rightarrow H$ is continuous).
Since $f$ is genuine, \eqref{eq:the section h is almost a splitting} and \eqref{eq:Weil factor identities} imply
\begin{align*}
l_H(\mathfrak{h}(h)f)=\gamma_{\psi'}(\det c_{h})\psi(u_h)l_H(f).
\end{align*}
Hence $l_H$ satisfies \eqref{eq:equivariance props of a H psi' psi functional} and is an $(H,\psi',\psi)$-functional
on the space of $\mathcal{S}^{\mathrm{gen}}(\widetilde{H})$, where $\widetilde{H}$ acts by right-translations.

Next consider the space of $\mathcal{S}^{\mathrm{gen}}(\lmodulo{\mathfrak{s}(H\cap B_{n,*})}{\widetilde{H}})$. There is a surjection
\begin{align*}
\mathcal{S}^{\mathrm{gen}}(p^{-1}(B_{n,*}H))\rightarrow \mathcal{S}^{\mathrm{gen}}(\lmodulo{\mathfrak{s}(H\cap B_{n,*})}{\widetilde{H}})
\end{align*}
given by $f\mapsto f^{B_{n,*}\cap H}$, where
\begin{align*}
f^{B_{n,*}\cap H}(h)=\int_{B_{n,*}\cap H}f(\mathfrak{s}(b_*)h)\, db_*,\qquad h\in \widetilde{H}.
\end{align*}
Note that $\mathfrak{h}(b_*)=\mathfrak{s}(b_*)$ (Claim~\ref{claim:h and s agree on H cap B_n*}). Then
we have a functional $l_{\lmodulo{H\cap B_{n,*}}{H}}$ on $\mathcal{S}^{\mathrm{gen}}(\lmodulo{\mathfrak{s}(H\cap B_{n,*})}{\widetilde{H}})$ defined by (see \cite[(44)]{Yia})
\begin{align}\label{int:integral on quotient of H}
l_{\lmodulo{H\cap B_{n,*}}{H}}(f)=\int_{\lmodulo{H\cap B_{n,*}}{H}}f(\mathfrak{h}(h))\gamma_{\psi'}^{-1}(\det c_h)\psi^{-1}(u_h)\, dh.
\end{align}
The integrand is invariant on the left by $H\cap B_{n,*}$: for $b_*\in H\cap B_{n,*}$ and $h\in H$, by \eqref{eq:the section h is almost a splitting}, Claim~\ref{claim:h and s agree on H cap B_n*},
\eqref{eq:Weil factor identities} and the fact that $h\mapsto\gamma_{\psi'}^{-1}(\det c_h)$ is trivial on $H\cap B_{n,*}$,
\begin{align*}
\gamma_{\psi'}^{-1}(\det c_{b_*h})\mathfrak{h}(b_*h)&=\gamma_{\psi'}^{-1}(\det c_{b_*}\det c_{h})(\det c_{b_*},\det c_h)_2\mathfrak{h}(b_*)\mathfrak{h}(h)
\\&=\gamma_{\psi'}^{-1}(\det c_{b_*})\gamma_{\psi'}^{-1}(\det c_{h})\mathfrak{s}(b_*)\mathfrak{h}(h)
\\&=\gamma_{\psi'}^{-1}(\det c_{h})\mathfrak{s}(b_*)\mathfrak{h}(h).
\end{align*}
Hence
\begin{align*}
f(\mathfrak{h}(b_*h))\gamma_{\psi'}^{-1}(\det c_{b_*h})=f(\mathfrak{h}(h))\gamma_{\psi'}^{-1}(\det c_h).
\end{align*}
Also the character $\psi^{-1}$ satisfies $\psi^{-1}(u_{b_*h})=\psi^{-1}(u_{h})$ because $u_{b_*}=I_{n}$.
Of course, $l_{\lmodulo{H\cap B_{n,*}}{H}}$ satisfies the same equivariance properties as $l_H$. 

Henceforth assume
\begin{align}\label{eq:assumptions on character for H functional}
q^{2s_i}\ne q^{\pm 2s_j}, \quad \forall1\leq i<j\leq k,\qquad q^{2s_i}=q^{-2s_{n-i+1}},\quad\forall1\leq i\leq k.
\end{align}
By Theorem~\ref{theorem:unramified Shalika functional 1 dim} and Claim~\ref{claim:H and Shalika functionals}
the space of $(H,\psi',\psi)$-functionals on $\mathrm{I}(\chi^{-1})$ is at most one-dimensional.

We chose $\chi$ in this form so that it is trivial on $\mathfrak{s}(H\cap B_{n,*})$.
The modulus character $\delta_{B_n}$ is also trivial on $H\cap B_{n,*}$. We may then regard
the elements of $\mathrm{I}(\chi^{-1})$ as genuine locally constant functions  on
$\lmodulo{\mathfrak{s}(H\cap B_{n,*})}{\widetilde{H}}$. Define
a functional $\Lambda_{\chi}$ on $\mathrm{I}(\chi^{-1})$ by
\begin{align*}
\Lambda_{\chi}(f_{\chi^{-1}})=l_{\lmodulo{H\cap B_{n,*}}{H}}(f_{\chi^{-1}}).
\end{align*}
First we prove that this makes sense, for sufficiently many characters $\chi$. Later we extend it to all relevant characters
by meromorphic continuation. The equivariance properties of $l_H$ imply that
$\Lambda_{\chi}$ is an $(H,\psi',\psi)$-functional. 
\begin{claim}\label{claim:functional on quotient of H is abs conv}
Assuming $\Re(s_i)<\Re(s_{i+1})$ for all $1\leq i <k$ and $\Re(s_{k})<0$, the integral $l_{\lmodulo{H\cap B_{n,*}}{H}}(f_{\chi^{-1}})$ given by \eqref{int:integral on quotient of H} is absolutely convergent for all $f_{\chi^{-1}}$.
\end{claim}
\begin{proof}
Since $|f_{\chi^{-1}}|$ can be regarded as a function on $G_n$, the proof in \cite[Proposition~7.1]{Yia} applies to our setting as well. Briefly, writing $dh$ on $\lmodulo{A_k}{T_k}\times N_k \times G_k(\mathcal{O})$ with $A_k=(T_k\cap G_k(\mathcal{O}))T_k^2$,
\begin{align*}
\int_{\lmodulo{H\cap B_{n,*}}{H}}|f_{\chi^{-1}}|(h)\, dh=
\int_{U_k}
\sum_{t\in\lmodulo{A_k}{T_k}}
\int_{N_k}
\int_{G_k(\mathcal{O})}|f_{\chi^{-1}}|(\omega t^{\triangle}v^{\triangle}k^{\triangle}u\omega^{-1})\, dk\, dv\, du.
\end{align*}
We may assume $f_{\chi^{-1}}$ is unramified. Then the integral is majorized by a finite sum of integrals
\begin{align*}
\int_{U_k}\int_{N_k}|f_{\chi^{-1}}|(\omega\left(\begin{array}{cc}v&u\\&I_k\end{array}\right)\omega^{-1}\ \rconj{\omega}(t^{\triangle}))\, dv\, du
\end{align*}
(we omitted a constant depending on the change of measure $dvdu$, caused by the conjugation by $t^{\triangle}$),
which are known to converge given our assumption on $\chi$.
\end{proof}
Recall the surjection $P_{\chi}:\mathcal{S}^{\mathrm{gen}}(\widetilde{G}_n)\rightarrow\mathrm{I}(\chi)$ given by
\eqref{eq:P chi surjection}. We combine $P_{\chi^{-1}}$ with \eqref{int:integral on quotient of H} to obtain a formula convenient for computations (as in \cite[(45)]{Yia}).
For $g\in B_{n,*}H$, one can write $g=b_*h$ with $b_*\in B_{n,*}$ and $h\in H$. This writing is not unique. We still put $b_g=b_*$ and $h_g=h$. The arguments after \eqref{int:integral on quotient of H} imply that the mappings $g\mapsto\psi^{-1}(u_{h_g})$ and
$g\mapsto \mathfrak{s}(b_g)\mathfrak{h}(h_g)\gamma_{\psi'}^{-1}(\det c_{h_g})$ are well defined.
Also $g\mapsto\delta^{1/2}_{B_n}(b_g)$ is independent of the choice of $b_g$ and the assumption
$\chi|_{\mathfrak{s}(H\cap B_{n,*})}=1$ implies $g\mapsto\chi(\mathfrak{s}(b_g))$ is independent of the specific writing.
Now we formally obtain
\begin{align}\label{int:integral on B_n*H}
\Lambda_{\chi}(P_{\chi^{-1}}(f))=
\int_{B_{n,*}H}\delta^{1/2}_{B_n}(b_g)\chi(\mathfrak{s}(b_g))f(\mathfrak{s}(b_g)\mathfrak{h}(h_g))\gamma_{\psi'}^{-1}(\det c_{h_g})\psi^{-1}(u_{h_g})\, dg.
\end{align}
This formula is valid whenever the \rhs\ is absolutely convergent. The measure $dg$ is taken to be the restriction of the Haar measure on $G_n$ to $B_{n,*}H$ and on the \lhs\ we normalize the measure $dh$ accordingly (the measure $d_lb_*$ has already been normalized, see after \eqref{eq:P chi surjection}).

For any $\lambda=(\lambda_1,\ldots,\lambda_k)\in\Z^k$, put
$\varpi^{\lambda}=\diag(\varpi^{\lambda_1},\ldots,\varpi^{\lambda_k})$ and denote $t_{\lambda}=\diag(\varpi^{\lambda},I_k)$.
Let
\begin{align*}
&\Z^k_{\geq}=\setof{(\lambda_1,\ldots,\lambda_k)\in\Z^k}{\lambda_1\geq\ldots\geq\lambda_k},\\
&\Z^k_+=\setof{(\lambda_1,\ldots,\lambda_k)\in\Z^k}{\lambda_1\geq\ldots\geq\lambda_k\geq0}.
\end{align*}
We also use $2\Z^k_+$, where for all $i$, $\lambda_i\in 2\Z$.

The following lemma describes the evaluation of $\Lambda_{\chi}(\varphi_{w,\chi^{-1}})$ for certain elements $w$, where it can be done succinctly. It is the main technical result of this section. The proof is parallel to the proof of Propositions~5.2, 8.1 and 8.2 in \cite{Yia}, the complications involve the various sections $\mathfrak{s}$, $\mathfrak{h}$ and $\kappa$. The results will be used below to deduce the general formula.
\begin{lemma}\label{lemma:3 specific computations of Shalika model}
Assume $\Re(s_i)<\Re(s_{i+1})$ for all $1\leq i <k$ and $\Re(s_{k})<0$. Also assume
$\psi'=\psi_{\varrho}$, for some $\varrho\in\mathcal{O}^*$ (i.e., $\psi'(x)=\psi(\varrho x)$).
The following holds.
\begin{align}\label{eq:computation of functional for phi I with t}
&\Lambda_{\chi}(\mathfrak{s}(t_{-\lambda})\varphi_{e,\chi^{-1}})=\begin{cases}
\mathrm{vol}(\mathcal{I})\delta^{1/2}_{B_n}\chi(\mathfrak{s}(t_{\lambda}))&\lambda\in2\Z^k_+,\\
0&\lambda\in\Z^k_+\setdifference2\Z^k_+.
\end{cases}
\end{align}
\begin{align}\label{eq:computation of functional for phi (k,k+1)}
&\Lambda_{\chi}(\varphi_{s_{\alpha_k},\chi^{-1}})= \mathrm{vol}(\mathcal{I}\mathfrak{w}\mathcal{I})(1-q^{-1}+(-\varrho,\varpi)_2(\varpi,\varpi)_2^{k-1}q^{-1/2}\chi^{-1/2}(a_{\alpha_k})).
\end{align}
For all $1\leq i<k$,
\begin{align}\label{eq:computation of functional for phi (i,i+1)()}
&\Lambda_{\chi}(\varphi_{s_{\alpha_i}s_{\alpha_{n-i}},\chi^{-1}})=
\mathrm{vol}(\mathcal{I}\mathfrak{w}\mathcal{I})
(1-q^{-1}+q^{-2}+(1-q^{-1})^2\frac{\chi^{-1}(a_{\alpha_i})}{1-\chi^{-1}(a_{\alpha_i})}).
\end{align}
Here $\chi^{-1/2}(a_{\alpha_k})=q^{2s_k}$.
\end{lemma}
\begin{remark}
Note that $\chi^{-1/2}(a_{\alpha})$ is well defined
because $\chi(\mathfrak{s}(\diag(I_{i-1},\varpi^2,I_{n-i})))=q^{-2s_i}$.
\end{remark}
\begin{proof}
We use \eqref{int:integral on B_n*H} for the computation. We first proceed formally, the computation will also imply
that \eqref{int:integral on B_n*H} is absolutely convergent, providing the justification for replacing
\eqref{int:integral on quotient of H} with \eqref{int:integral on B_n*H}.

Consider \eqref{eq:computation of functional for phi I with t} and assume
$\lambda\in\Z^k_+$. Note that $t_{-\lambda}^{-1}=t_{\lambda}$.
The integrand vanishes unless $g\in\mathcal{I}t_{\lambda}$.
For $\lambda\in\Z^k_+$, $t_{\lambda}\in T_n^-$ whence
$t_{\lambda}^{-1}(\mathcal{I}^-)t_{\lambda}<\mathcal{I}^-$ and by \eqref{eq:reformulation of Iwahori} and
\eqref{eq:Iwahori contained in B*H},
\begin{align*}
\mathcal{I}t_{\lambda}\subset (B_n\cap K)t_{\lambda}(B_n\cap K)(H\cap K).
\end{align*}
In particular $\mathcal{I}t_{\lambda}\subset t_{\lambda}B_{n,*}H$, the latter does not intersect
$B_{n,*}H$ unless $\lambda\in2\Z^k_+$, by Claim~\ref{claim:coset decomposition B*H in BnH}. So the integrand vanishes unless
$\lambda\in2\Z^k_+$. In this case one can take $b_g=b_0t_{\lambda}b_0'$ for some $b_0,b_0'\in B_n\cap K$ and $h_g\in H\cap K$. Then
$\psi(u_{h_g})=\gamma_{\psi'}^{-1}(\det c_{h_g})=1$ because $h_g\in K$, so that both
\eqref{int:integral on B_n*H} and \eqref{int:integral on quotient of H} equal
\begin{align*}
\mathrm{vol}(\mathcal{I})\delta^{1/2}_{B_n}\chi(\mathfrak{s}(t_{\lambda})).
\end{align*}

Next we prove \eqref{eq:computation of functional for phi (k,k+1)}. Let $\mathfrak{w}$ be the representative of $s_{\alpha_k}$. Using \eqref{eq:reformulation of Iwahori} and \eqref{eq:Iwahori- without one root contained in BcapK times HcapK} we obtain
\begin{align*}
\mathcal{I}\mathfrak{w}\mathcal{I}=(B_n\cap K)\mathfrak{w}N_{\alpha_k}(\mathcal{O})\prod_{\beta>0:\beta\ne\alpha_k}N_{\beta}^-(\varpi\mathcal{O})
\subset (B_n\cap K)\mathfrak{w}N_{\alpha_k}(\mathcal{O})(H\cap K).
\end{align*}
The integrand vanishes unless $g\in\mathcal{I}\mathfrak{w}\mathcal{I}$.  Write
\begin{align*}
g=b_0\mathfrak{w}vh_0,\qquad b_0\in B_n\cap K,\quad v\in N_{\alpha_k}(\mathcal{O}),\quad h_0\in H\cap K.
\end{align*}
We decompose $\mathfrak{w}v=b_*h \in B_{n,*}H$, using $2\times 2$ matrix computations, where the matrices are regarded as elements of $G_2$ embedded in $G_n$ such that its derived group is
$\mathcal{R}_{\alpha_k}$ ($\mathcal{R}_{\alpha_k}$ was defined after \eqref{eq:one dim for regular hom}). Then $\mathfrak{w}$ is the matrix $\left(\begin{smallmatrix}&-1\\1&\end{smallmatrix}\right)$. If $v=\left(\begin{smallmatrix}1&z\\&1\end{smallmatrix}\right)$ and $z\ne0$,
\begin{align}\label{eq:explicit h p v epsilon for w k,k+1}
b_*=\left(\begin{array}{cc}1&-z\\&z^2\end{array}\right),\quad
h=\left(\begin{array}{cc}z^{-1}&\\z^{-2}&z^{-1}\end{array}\right). 
\end{align}
Then $b_g=b_0b_*$ and $h_g=hh_0$, hence
\begin{align*}
&\psi^{-1}(u_{h_g})=\psi^{-1}(u_h)=\psi^{-1}(z^{-1}),\\
&\gamma_{\psi'}^{-1}(\det c_{h_g})=(\det c_h,\det c_{h_0})_2\gamma_{\psi'}^{-1}(\det c_h)=(z,\det c_{h_0})_2\gamma_{\psi'}^{-1}(z^{-1}),\\
&\delta^{1/2}_{B_n}(b_g)=|z|^{-1},\\
&\chi(\mathfrak{s}(b_g))=|z|^{2s_{k+1}}=|z|^{-2s_{k}}.
\end{align*}
Next we show
\begin{align*}
\mathrm{ch}_{s_{\alpha_k},\chi^{-1}}(\mathfrak{s}(b_g)\mathfrak{h}(h_g))=(z,z)_2^{k-1}(z,\det c_{h_0})_2.
\end{align*}
Because $g\in\mathcal{I}\mathfrak{w}\mathcal{I}$, this will follow immediately from
\begin{align}\label{eq:kappa g and the b* h decomposition}
\mathfrak{s}(b_g)\mathfrak{h}(h_g)=(z,z)_2^{k-1}(z,\det c_{h_0})_2\kappa(g).
\end{align}
To show this, start with $\kappa(g)=\kappa(b_0)\kappa(\mathfrak{w})\kappa(v)\kappa(h_0)$. Since
$\mathfrak{s}(\mathfrak{w})\mathfrak{s}(v)=\mathfrak{s}(\mathfrak{w}v)=\mathfrak{s}(b_*h)$,
\begin{align}\label{eq:kappa g and the b* h decomposition computation 1}
\kappa(g)=\sigma(b_*,h)\kappa(b_0)\mathfrak{s}(b_*)\mathfrak{s}(h)\kappa(h_0).
\end{align}
Using \eqref{eq:Kubota formula} we see $\sigma(b_*,h)=1$.
Next we claim $\mathfrak{s}(h)=(z,z)_2^{k-1}\mathfrak{h}(h)$. Start with $\mathfrak{h}(h)=\mathfrak{s}(\omega)\mathfrak{s}(\rconj{\omega^{-1}}h)\mathfrak{s}(\omega)^{-1}$.
Write $\rconj{\omega^{-1}}h=cu$, where $c=\diag(I_{k-1},z^{-1},I_{k-1},z^{-1})$ and
\begin{align*}
u=\left(\begin{array}{cccc}I_{k-1}\\&1&&z^{-1}\\&&I_{k-1}\\&&&1\end{array}\right)\in U.
\end{align*}
Then
$\mathfrak{s}(\rconj{\omega^{-1}}h)=\mathfrak{s}(c)\mathfrak{s}(u)$. 
We find $\rconj{\mathfrak{s}(\omega)}\mathfrak{s}(c)$ using Remark~\ref{remark:computing conjugation t by w}.
Write $\omega=t_0\mathfrak{w}_{\omega}$ with
\begin{align*}
&\mathfrak{w}_{\omega}=
(\mathfrak{w}_{\alpha_{n-1}}\cdot\ldots\cdot\mathfrak{w}_{\alpha_{1}})\cdot\ldots\cdot
(\mathfrak{w}_{\alpha_{n-1}}\cdot\ldots\cdot\mathfrak{w}_{\alpha_{k}})=\prod_{i=1}^k(\mathfrak{w}_{\alpha_{n-1}}\cdot\ldots\cdot\mathfrak{w}_{\alpha_{i}})\in\mathfrak{W},\\
&t_0=\diag((-1)^{k}I_k,(-1)^{k-1},(-1)^{k-2},\ldots,1).
\end{align*}
If $w_{\omega}$ is the Weyl element corresponding to $\mathfrak{w}_{\omega}$ (or $\omega$),
\begin{align*}
\setof{\alpha>0}{w_{\omega}\alpha<0}=\setof{(i,j)}{1\leq i\leq k,\quad i<j\leq n}.
\end{align*}
Then 
\begin{align*}
\rconj{\mathfrak{s}(\omega)}\mathfrak{s}(c)=(-1,z^{-1})_2^{k-1}\
\rconj{\mathfrak{s}(t_0)}\mathfrak{s}(\rconj{\omega}c)=(z,z)_2^{k}\mathfrak{s}(\rconj{\omega}c).
\end{align*}
Consider $\rconj{\mathfrak{s}(\omega)}\mathfrak{s}(u)$. If $\omega'\in G_n$ satisfies
$\rconj{{\omega'}^{-1}\omega}u=u$, then $\rconj{\mathfrak{s}({\omega'}^{-1}\omega)}\mathfrak{s}(u)=\mathfrak{s}(u)$ hence
$\rconj{\mathfrak{s}(\omega)}\mathfrak{s}(u)=\rconj{\mathfrak{s}(\omega')}\mathfrak{s}(u)$. One can take
\begin{align*}
\omega'=\diag(I_{k-1},(-1)^k,I_{k})\mathfrak{w}_{\alpha_k}\cdot\ldots\cdot\mathfrak{w}_{\alpha_{n-1}}.
\end{align*}
The conjugation of $\mathfrak{s}(u)$ by
$\mathfrak{s}(\mathfrak{w}_{\alpha_{k+1}}\cdot\ldots\cdot\mathfrak{w}_{\alpha_{n-1}})$ commutes with $\mathfrak{s}$ (i.e.,
$\rconj{\mathfrak{s}(\ldots)}\mathfrak{s}(u)=\mathfrak{s}(\rconj{\ldots}u)$), because it keeps $u$ in $N_n$. The last two conjugations
are computed using \eqref{eq:block-compatibility} and \eqref{eq:Kubota formula},
\begin{align*}
\mathfrak{s}(\left(\begin{array}{cc}&(-1)^{k-1}\\1\end{array}\right))
\mathfrak{s}(\left(\begin{array}{cc}1&(-1)^{k-1}z^{-1}\\&1\end{array}\right))
\mathfrak{s}(\left(\begin{array}{cc}&1\\(-1)^{k-1}\end{array}\right))=
\mathfrak{s}(\left(\begin{array}{cc}1&\\z^{-1}&1\end{array}\right)).
\end{align*}
Therefore $\rconj{\mathfrak{s}(\omega)}\mathfrak{s}(u)=\mathfrak{s}(\rconj{\omega}u)$.
Also by \eqref{eq:Kubota formula}, $\mathfrak{s}(\rconj{\omega}c)\mathfrak{s}(\rconj{\omega}u)=(z,z)_2
\mathfrak{s}(\rconj{\omega}(cu))$.
Altogether
\begin{align*}
\mathfrak{h}(h)=\rconj{\mathfrak{s}(\omega)}(\mathfrak{s}(c)\mathfrak{s}(u))
=(z,z)_2^k\mathfrak{s}(\rconj{\omega}c)\mathfrak{s}(\rconj{\omega}u)=
(z,z)_2^{k-1}\mathfrak{s}(\rconj{\omega}(cu))=(z,z)_2^{k-1}\mathfrak{s}(h).
\end{align*}
Plugging this into \eqref{eq:kappa g and the b* h decomposition computation 1}, using \eqref{eq:the section h is almost a splitting} and Claim~\ref{claim:h and kappa agree on H cap K} yields
\begin{align*}
\kappa(g)=\kappa(b_0)\mathfrak{s}(b_*)(z,z)_2^{k-1}\mathfrak{h}(h)\mathfrak{h}(h_0)=(z,z)_2^{k-1}(z,\det c_{h_0})_2\mathfrak{s}(b_0b_*)\mathfrak{h}(hh_0),
\end{align*}
thereby giving \eqref{eq:kappa g and the b* h decomposition}.

We see that \eqref{int:integral on B_n*H} becomes
\begin{align*}
\mathrm{vol}(\mathcal{I}\mathfrak{w}\mathcal{I})\sum_{l=0}^{\infty}(q^{2s_k+1})^lq^{-l}
(\varpi^{-l},\varpi^{-l})_2^{k-1}\int_{\mathcal{O}^*}
\gamma_{\psi'}^{-1}(\varpi^{-l}x)\psi^{-1}(\varpi^{-l}x)\, dx.
\end{align*}
Note that $dx$ is the additive measure of $F$ (in particular $\int_{\mathcal{O}^*}dx=1-q^{-1}$).
\begin{remark}
This $dx$-integral is a particular instance of an integral studied by Szpruch \cite{Dani4}, who used it to represent
a local coefficient resembling Tate's $\gamma$-factor. 
\end{remark}
For $l>1$ the $dx$-integral vanishes, since
\begin{align*}
\int_{\mathcal{O}^*}
\gamma_{\psi'}^{-1}(\varpi^{-l}x)\psi^{-1}(\varpi^{-l}x)\, dx
=\gamma_{\psi'}^{-1}(\varpi^{-l})\sum_{\zeta\in\lmodulo{(1+\varpi\mathcal{O})}{\mathcal{O}^*}}(\varpi^{-l},\zeta)\int_{1+\varpi\mathcal{O}}\psi^{-1}(\varpi^{-l}\zeta x)\, dx=0,
\end{align*}
because the conductor of $\psi$ is $\mathcal{O}$. Consider $l=1$. Then \eqref{eq:Weil factor identities} and the assumption $|\varrho|=1$ show $\gamma_{\psi'}^{-1}(\varpi^{-1}x)=(-\varrho,\varpi)_2\gamma_{\psi}(\varpi^{-1}x)$, then $dx$-integral equals $(-\varrho,\varpi)_2$ multiplied by
\begin{align*}
\int_{\mathcal{O}^*}
\gamma_{\psi}(\varpi^{-1}x)\psi^{-1}(\varpi^{-1}x)\, dx.
\end{align*}
This integral was computed in \cite[Lemma~1.12]{Dani4} (this is $J_1(\mathbb{F},\psi^{-1},\chi^0)$ in his notation) and equals $q^{-1/2}$. Altogether we obtain
\begin{align*}
 \mathrm{vol}(\mathcal{I}\mathfrak{w}\mathcal{I})(1-q^{-1}+q^{2s_k-1/2}(-\varrho,\varpi)_2(\varpi,\varpi)_2^{k-1}).
\end{align*}


Finally we prove \eqref{eq:computation of functional for phi (i,i+1)()}. Let $\mathfrak{w}$ be the representative of $s_{\alpha_i}s_{\alpha_{n-i}}$. The argument is similar to the previous case, we point the relevant changes in the computation. Using \eqref{eq:Iwahori- without one root contained in BcapK times HcapK} and the fact that $\mathfrak{w}\in H$,
\begin{align*}
\mathcal{I}\mathfrak{w}\mathcal{I}\subset (B_n\cap K)\mathfrak{w}N_{\alpha_i}(\mathcal{O})N_{\alpha_{n-i}}(\mathcal{O})\mathfrak{w}^{-1}(H\cap K).
\end{align*}
Then $g=b_0\mathfrak{w}v\mathfrak{w}^{-1}h_0$ with $v\in N_{\alpha_i}(\mathcal{O})N_{\alpha_{n-i}}(\mathcal{O})$ ($b_0,h_0$ as above).
Decompose $\mathfrak{w}v\mathfrak{w}^{-1}=b_*h$, now we compute using $2$ blocks of $2\times2$ matrices, belonging
to $G_2\times G_2$ embedded in $G_n$ such that its derived group is $\mathcal{R}_{\alpha_i}\times\mathcal{R}_{\alpha_{n-i}}$. For $v=(\left(\begin{smallmatrix}1&z_1\\&1\end{smallmatrix}\right),\left(\begin{smallmatrix}1&z_2\\&1\end{smallmatrix}\right))$ with $z_1z_2\ne1$
and $(1-z_1z_2)\in F^{*2}\mathcal{O}^*$,
\begin{align*}
&b_*=\left(\left(\begin{array}{cc}(1-z_1z_2)^{-1}&z_2(1-z_1z_2)^{-1}\\&1\end{array}\right),
\left(\begin{array}{cc}(1-z_1z_2)^{-1}&z_1(1-z_1z_2)^{-1}\\&1\end{array}\right)\right),\\
&h=\left(\left(\begin{array}{cc}1&-z_2\\-z_1&1\end{array}\right),
\left(\begin{array}{cc}1&-z_1\\-z_2&1\end{array}\right)\right).
\end{align*}
When $(1-z_1z_2)\notin F^{*2}\mathcal{O}^*$ we still have $\mathfrak{w}v\mathfrak{w}^{-1}=b_*h$ but $b_*\notin B_{n,*}$, and since
\begin{align*}
g\in\diag(I_{i-1},(1-z_1z_2)^{-1},(1-z_1z_2),I_{n-i-1})B_{n,*}H,
\end{align*}
by Claim~\ref{claim:coset decomposition B*H in BnH} we obtain $g\notin B_{n,*}H$. Hence we can assume $1-z_1z_2\in F^{*2}\mathcal{O}^*$.
Then
\begin{align*}
&\psi^{-1}(u_{h_g})=1,\\
&\gamma_{\psi'}^{-1}(\det c_{h_g})=\gamma_{\psi'}^{-1}(1-z_1z_2)=1,\\
&\delta^{1/2}_{B_n}(b_g)=|1-z_1z_2|^{-1},\\
&\chi(\mathfrak{s}(b_g))=|1-z_1z_2|^{-s_i-s_{n-i}}=|1-z_1z_2|^{-s_i+s_{i+1}}.
\end{align*}
We may also take $z_1\ne0$.
Next we prove
\begin{align*}
\mathfrak{s}(b_g)\mathfrak{h}(h_g)=(-z_1,1-z_1z_2)_2\kappa(g),
\end{align*}
whence
\begin{align*}
\mathrm{ch}_{s_{\alpha_i}s_{\alpha_{n-i}},\chi^{-1}}(\mathfrak{s}(b_g)\mathfrak{h}(h_g))=(-z_1,1-z_1z_2)_2.
\end{align*}

As with the proof of \eqref{eq:kappa g and the b* h decomposition}, we begin with proving $\sigma(b_*,h)=1$. This is immediate if
$z_2=0$. When $z_2\ne0$, apply
\eqref{eq:Kubota formula} separately to each $G_2$-block, then
use \eqref{eq:block-compatibility}. We obtain
\begin{align*}
(-z_1,1-z_1z_2)_2\ (-z_2,1-z_1z_2)_2\ ((1-z_1z_2)^{-1},(1-z_1z_2))_2=1.
\end{align*}
Here we used the fact that $(x,1-x)=1$ for any $x\ne0$ and $1-z_1z_2\in F^{*2}\mathcal{O}^*$.
Next we show $\mathfrak{s}(h)=(z_1,1-z_1z_2)_2\mathfrak{h}(h)$. Put $x=\left(\begin{smallmatrix}1&-z_2\\-z_1&1\end{smallmatrix}\right)$. For any $g\in G_2$,
let
\begin{align*}
d_1(g)=\diag(I_{i-1},g,I_{n-i-1}),\qquad d_2(g)=\diag(I_{k+i-1},g,I_{k-i-1}).
\end{align*}
Then $\rconj{\omega^{-1}}h=d_1(x)d_2(x)$ and $\mathfrak{h}(h)=\rconj{\mathfrak{s}(\omega)}\mathfrak{s}(d_1(x)d_2(x))$. Now
\eqref{eq:block-compatibility} implies
\begin{align}\label{eq:before x_1x_2x_3}
\mathfrak{s}(d_1(x)d_2(x))=\mathfrak{s}(d_1(x))\mathfrak{s}(d_2(x)).
\end{align}
To compute $\rconj{\mathfrak{s}(\omega)}\mathfrak{s}(d_j(x))$ we write $x$ using the Iwasawa decomposition in $G_2$,
\begin{align*}
x=x_1x_2x_3=\left(\begin{array}{cc}1-z_1z_2&\\&1\end{array}\right)
\left(\begin{array}{cc}1&-z_2(1-z_1z_2)^{-1}\\&1\end{array}\right)\left(\begin{array}{cc}1&0\\-z_1&1\end{array}\right).
\end{align*}
Then by \eqref{eq:Kubota formula},
\begin{align}\label{eq:x_1x_2x_3}
\mathfrak{s}(d_j(x))=(-z_1,1-z_1z_2)_2\mathfrak{s}(d_j(x_1))\mathfrak{s}(d_j(x_2))\mathfrak{s}(d_j(x_3)),\qquad j=1,2.
\end{align}
We compute the conjugation separately for each $x_i$:

$\rconj{\mathfrak{s}(\omega)}\mathfrak{s}(d_j(x_1))=\mathfrak{s}(\rconj{\omega}d_j(x_1))$:
If $\omega=t_0\mathfrak{w}_{\omega}$ as above, $\rconj{\mathfrak{s}(\mathfrak{w}_{\omega})}\mathfrak{s}(d_j(x_1))=
\mathfrak{s}(\rconj{\mathfrak{w}_{\omega}}d_j(x_1))$ by Claim~\ref{claim:conjugation W and T} and
$\rconj{\mathfrak{s}(t_0)}\mathfrak{s}(\rconj{\mathfrak{w}_{\omega}}d_j(x_1))=
\mathfrak{s}(\rconj{\mathfrak{w}_{\omega}}d_j(x_1))$ because $(-1,1-z_1z_2)_2=1$ ($1-z_1z_2\in F^{*2}\mathcal{O}^*$).
Hence $\rconj{\mathfrak{s}(\omega)}\mathfrak{s}(d_j(x_1))=\mathfrak{s}(\rconj{\omega}d_j(x_1))$.

$\rconj{\mathfrak{s}(\omega)}\mathfrak{s}(d_j(x_2))=\mathfrak{s}(\rconj{\omega}d_j(x_2))$:
For $j=1$ we argue as in the case of $\rconj{\mathfrak{s}(\omega)}\mathfrak{s}(u)$ above, this time with
\begin{align*}
\omega'=\diag(I_{n-i-1},-1,I_i)\mathfrak{w}_{\alpha_{n-i}}\cdot\ldots\cdot\mathfrak{w}_{\alpha_{i}}
\mathfrak{w}_{\alpha_{n-i}}\cdot\ldots\cdot\mathfrak{w}_{\alpha_{i+1}},
\end{align*}
and obtain $\rconj{\mathfrak{s}(\omega)}\mathfrak{s}(d_1(x_2))=\mathfrak{s}(\rconj{\omega}d_1(x_2))$. For $j=2$ this holds simply because 
$\rconj{\omega}d_2(x_2)\in N_n$.

$\rconj{\mathfrak{s}(\omega)}\mathfrak{s}(d_j(x_3))=\mathfrak{s}(\rconj{\omega}d_j(x_3))$:
A direct computation using \eqref{eq:Kubota formula} shows
\begin{align*}
\mathfrak{s}(x_3)
=\mathfrak{s}\left(\begin{array}{cc}&-1\\1&\end{array}\right)
\mathfrak{s}\left(\begin{array}{cc}1&z_1\\&1\end{array}\right)
\mathfrak{s}\left(\begin{array}{cc}&1\\-1&\end{array}\right)=\kappa(x_3)
\end{align*}
($z_1\in\mathcal{O}$ and $\mathfrak{s}$ and $\kappa$ agree on $N_n\cap K$). Hence $\rconj{\mathfrak{s}(\omega)}\mathfrak{s}(d_j(x_3))=\mathfrak{s}(\rconj{\omega}d_j(x_3))$.

Summarizing,
\begin{align}\label{eq:after1 x_1x_2x_3}
\rconj{\mathfrak{s}(\omega)}\mathfrak{s}(d_j(x))=(-z_1,1-z_1z_2)_2\mathfrak{s}(\rconj{\omega}d_j(x_1))\mathfrak{s}(\rconj{\omega}d_j(x_2))\mathfrak{s}(\rconj{\omega}d_j(x_3)).
\end{align}
Similarly we see that
\begin{align}\label{eq:after x_1x_2x_3}
&\mathfrak{s}(\rconj{\omega}d_1(x))=\mathfrak{s}(\rconj{\omega}d_1(x_1))\mathfrak{s}(\rconj{\omega}d_1(x_2))
\mathfrak{s}(\rconj{\omega}d_1(x_3)),\\\label{eq:after 2 x_1x_2x_3}
&\mathfrak{s}(\rconj{\omega}d_2(x))=(-z_1,1-z_1z_2)_2\mathfrak{s}(\rconj{\omega}d(x_1))\mathfrak{s}(\rconj{\omega}d(x_2))\mathfrak{s}(\rconj{\omega}d(x_3)),\\\label{eq:after 3 x_1x_2x_3}
&\mathfrak{s}(\rconj{\omega}(d_1(x)d_2(x)))=\mathfrak{s}(\rconj{\omega}d_1(x))\mathfrak{s}(\rconj{\omega}d_2(x)).
\end{align}
Thus \eqref{eq:before x_1x_2x_3}-\eqref{eq:after 3 x_1x_2x_3} yield
\begin{align*}
\mathfrak{h}(h)&=\rconj{\mathfrak{s}(\omega)}\mathfrak{s}(d_1(x)d_2(x))\\
&=\rconj{\mathfrak{s}(\omega)}\mathfrak{s}(d_1(x))\ \rconj{\mathfrak{s}(\omega)}\mathfrak{s}(d_2(x))\\
&=\mathfrak{s}(\rconj{\omega}d_1(x_1))\mathfrak{s}(\rconj{\omega}d_1(x_2))
\mathfrak{s}(\rconj{\omega}d_1(x_3))\mathfrak{s}(\rconj{\omega}d_2(x_1))\mathfrak{s}(\rconj{\omega}d_2(x_2))
\mathfrak{s}(\rconj{\omega}d_2(x_3))\\
&=(-z_1,1-z_1z_2)_2\mathfrak{s}(\rconj{\omega}d_1(x))\mathfrak{s}(\rconj{\omega}d_2(x))
\\&=(-z_1,1-z_1z_2)_2\mathfrak{s}(\rconj{\omega}(d_1(x)d_2(x)))\\
&=(-z_1,1-z_1z_2)_2\mathfrak{s}(h).
\end{align*}
Altogether \eqref{int:integral on B_n*H} becomes
\begin{align*}
\mathrm{vol}(\mathcal{I}\mathfrak{w}\mathcal{I})\int_{\mathcal{O}}\int_{\mathcal{O}}(-z_1,1-z_1z_2)_2|1-z_1z_2|^{-s_i+s_{i+1}-1}\mathrm{ch}_{F^{*2}\mathcal{O}^*}(1-z_1z_2)\, dz_1\, dz_2.
\end{align*}
Here $\mathrm{ch}_{F^{*2}\mathcal{O}^*}$ is the characteristic function of $F^{*2}\mathcal{O}^*$.
As in \cite{Yia} (proof of Proposition~8.2) we split the integral into the following sum,
\begin{align*}
\int_{\mathcal{O}}\int_{\varpi\mathcal{O}}\, dz_1\, d z_2 +
\int_{\varpi\mathcal{O}}\int_{\mathcal{O}^*}\, dz_1\, d z_2+
\int_{\mathcal{O}^*}\int_{\mathcal{O}^*}|1-z_1z_2|^{-s_i+s_{i+1}-1}\mathrm{ch}_{F^{*2}\mathcal{O}^*}(1-z_1z_2)\, dz_1\, d z_2.
\end{align*}
We used the fact that when either $z_1$ or $z_2$ belongs to $\varpi\mathcal{O}$ (and both $z_1,z_2\in\mathcal{O}$), $|1-z_1z_2|=1$ and $1-z_1z_2\in F^{*2}$, in particular $(-z_1,1-z_1z_2)_2=1$. For the rightmost integral note that when $z_1\in\mathcal{O}^*$,
$\mathrm{ch}_{F^{*2}\mathcal{O}^*}(1-z_1z_2)=1$ implies $(-z_1,1-z_1z_2)_2=1$. The first summand contributes $q^{-1}$, the second $q^{-1}(1-q^{-1})$. We proceed with the integral on the right. After a change of variables it becomes
\begin{align*}
(1-q^{-1})\int_{\mathcal{O}^*}|1-z|^{-s_i+s_{i+1}-1}\mathrm{ch}_{F^{*2}\mathcal{O}^*}(1-z)\, dz.
\end{align*}
Put $|1-z|=q^{-l}$ for some $l\geq0$. Then $\mathrm{ch}_{F^{*2}\mathcal{O}^*}(1-z)\ne0$ if and only if $l$ is even. Hence we get
\begin{align*}
(1-q^{-1})(1-2q^{-1}+(1-q^{-1})\frac{q^{-2(s_{i+1}-s_i)}}{1-q^{-2(s_{i+1}-s_i)}}).
\end{align*}
Summing up we obtain
\begin{align*}
\mathrm{vol}(\mathcal{I}\mathfrak{w}\mathcal{I})
(1-q^{-1}+q^{-2}+(1-q^{-1})^2\frac{q^{-2(s_{i+1}-s_i)}}{1-q^{-2(s_{i+1}-s_i)}}).
\end{align*}
The proof is complete.
\end{proof}
Recall that we assume $\chi$ satisfies \eqref{eq:assumptions on character for H functional}.
We say that $f_{\chi^{-1}}$ is a polynomial (resp. rational) section if its restriction to $\coprod_{t\in\lmodulo{T_{n,*}}{T_n}}\mathfrak{s}(t)\kappa(K)$ depends polynomially (resp. rationally) on $\chi^{-1}$, or more precisely
on $q^{\pm 2s_1},\ldots,q^{\pm 2s_k}$.
The definition is independent of the actual choice of representatives of $\lmodulo{T_{n,*}}{T_n}$.
\begin{proposition}\label{proposition:meromorphic continuation and a basis of Shalika functionals}
The functional $\Lambda_{\chi}$ has a meromorphic continuation to all $\chi$
satisfying \eqref{eq:assumptions on character for H functional}, in the sense that
if $f_{\chi^{-1}}$ is a rational section, $\Lambda_{\chi}(f_{\chi^{-1}})$ is a rational function, i.e.,
belongs to $\C(q^{-2s_1},\ldots,q^{-2s_k})$.
\end{proposition}
\begin{proof}
The argument appeared in \cite[\S~7]{Yia}. According to Theorem~\ref{theorem:unramified Shalika functional 1 dim}
and Claim~\ref{claim:H and Shalika functionals}, there exists at most one solution (a functional) to the system of equations, defined by \eqref{eq:equivariance props of a H psi' psi functional} with the additional condition that $\varphi_{e,\chi^{-1}}$ is mapped to $\mathrm{vol}(\mathcal{I})$. Furthermore, according to Claim~\ref{claim:functional on quotient of H is abs conv} and \eqref{eq:computation of functional for phi I with t}, in an open subset of $\C^k$, there exists a nontrivial solution. Hence by Bernstein's continuation principle (in \cite{Banks}) $\Lambda_{\chi}(f_{\chi^{-1}})$ extends to a rational function. 
\end{proof}
\begin{corollary}\label{corollary:3 specific computations of Shalika model meromorphic continuation}
The results of Lemma~\ref{lemma:3 specific computations of Shalika model} hold for all $\chi$ satisfying \eqref{eq:assumptions on character for H functional}, by means of meromorphic continuation.
\end{corollary}


Below we state the explicit formula for the unramified function $g\mapsto\Lambda_{\chi}(g\varphi_{K,\chi^{-1}})$. The formula is for $g=t_{\lambda}$ with $\lambda\in\Z^k_{\geq}$. The following simple claim shows that these elements already determine the values of the function.
\begin{claim}\label{claim:Shalika or H functional are determined on the torus}
Let $l$ be a metaplectic $(\psi',\psi)$-Shalika functional (resp. $(H,\psi',\psi)$-functional) on a genuine unramified $\pi\in\Alg{\widetilde{G}_n}$. Let $\xi$ be an unramified vector in the space of $\pi$ and consider the function $l_{\xi}(g)=l(\pi(g)\xi)$.
\begin{enumerate}[leftmargin=*]
\item\label{claim Shalika or H determined:K invariance left and right}$l_{\xi}(\kappa(k_0)g\kappa(k))=l_{\xi}(g)$ for any $k_0\in(G_k^{\triangle}U)\cap K$ (resp. $k_0\in H\cap K$), $g\in\widetilde{G}_n$ and $k\in K$.
\item\label{claim Shalika or H determined:part 1}$l_{\xi}$ is uniquely determined by its values on $t_{\lambda}$ (resp. $t_{-\lambda}$), where
$\lambda\in\Z^k_{\geq}$.
\item\label{claim Shalika or H determined:part 3}
$l_{\xi}(t_{\lambda})=0$ (resp. $t_{-\lambda}$) for $\lambda\in\Z^k_{\geq}\setdifference\Z^k_+$.
\end{enumerate}
\end{claim}
\begin{proof}
The first assertion follows from Claim~\ref{claim:h and kappa agree on H cap K} and \eqref{eq:equivariance props of a H psi' psi functional}. The second and third 
follow by a straightforward adaptation of the proofs of these results in the non-metaplectic setting (\cite[\S~6.2]{JR2}). 
For example, we explain the proof of \eqref{claim Shalika or H determined:part 3} for
the $(H,\psi',\psi)$-functional. Consider $\lambda\in\Z^k_{\geq}$.
Let $u\in \mathcal{O}_{k\times k}$ and $h=\rconj{\omega}\left(\begin{smallmatrix}I_k&u\\&I_k\end{smallmatrix}\right)\in H\cap K$. Then
\begin{align*}
\mathfrak{s}(t_{-\lambda})\mathfrak{h}(h)=\mathfrak{h}(\rconj{\omega}\left(\begin{smallmatrix}I_k&u\varpi^{\lambda}\\&I_k\end{smallmatrix}\right))\mathfrak{s}(t_{-\lambda}),
\end{align*}
whence by Claim~\ref{claim:h and kappa agree on H cap K},
\begin{align*}
l_{\xi}(\mathfrak{s}(t_{-\lambda}))
=l_{\xi}(\mathfrak{s}(t_{-\lambda})\mathfrak{h}(h))
=\psi(\left(\begin{smallmatrix}I_k&u\varpi^{\lambda}\\&I_k\end{smallmatrix}\right))l_{\xi}(\mathfrak{s}(t_{-\lambda})).
\end{align*}
Thus $l_{\xi}(\mathfrak{s}(t_{-\lambda}))=0$, unless $\lambda\in\Z^k_+$.
\end{proof}

Let $\Sp_k$ be the symplectic group on $n$ variables, embedded in $G_n$ as the subgroup
\begin{align*}
\setof{g\in G_n}{\transpose{g}\left(\begin{smallmatrix}&J_k\\-J_k\end{smallmatrix}\right)g=\left(\begin{smallmatrix}&J_k\\-J_k\end{smallmatrix}\right)},
\end{align*}
where $\transpose{g}$ denotes the transpose of $g$. We fix the Borel subgroup $\Sp_k\cap B_n$ and let $\Sigma_{\mathrm{Sp_k}}$ (resp. $\Sigma_{\Sp_k}^+$) be the corresponding root system (resp. positive roots) of $\Sp_k$.
The map $\Sigma_{G_n}\rightarrow\Sigma_{\Sp_k}$ is two-to-one onto the short roots, and injective onto the long roots.
Recall that we assumed $q^{2s_i}=q^{-2s_{n-i+1}}$ for all $1\leq i\leq k$ (in \eqref{eq:assumptions on character for H functional}). Equivalently, $\chi$ is trivial on $\mathfrak{s}(H\cap B_{n,*})$.
Hence if $\beta,\beta'\in\Sigma_{G_n}$ are the two roots corresponding to a short root $\alpha\in\Sigma_{\Sp_k}$,
$c_{\beta}(\chi)=c_{\beta'}(\chi)$. Thus we may extend the notation
$c_{\alpha}(\chi)$ to short roots $\alpha$, meaning $c_{\beta}(\chi)$. The same applies to $\chi^{-1}$. Similarly, we define $c_{\alpha}(\chi)$ for long roots
(there is no ambiguity here, for any $\chi$).

Denote the Weyl group of $\Sp_k$ by $W_{\Sp_k}$, it is
the group generated by $s_k$ and $s_is_{n-i}$, $1\leq i<k$. For $w\in W_{\Sp_k}$, let $\ell_{\Sp_k}(w)$
be the minimal number of simple reflections in $W_{\Sp_k}$ whose product is $w$, i.e., the usual length with respect to
$W_{\Sp_k}$. 

Recall that $\psi'=\psi_{\varrho}$ (see Lemma~\ref{lemma:3 specific computations of Shalika model}),
put $\epsilon_{\varrho,k}=(-\varrho,\varpi)_2(\varpi,\varpi)_2^{k-1}$. For any $\alpha\in\Sigma_{\Sp_k}^+$ let
\begin{align*}
y_{\alpha}(\chi^{-1})=\begin{dcases}
c_{\alpha}(\chi^{-1})c_{\alpha}(\chi)&\text{$\alpha$ is short,}\\
\frac{(1+q^{-1/2}\epsilon_{\varrho,k}\chi^{-1/2}(a_{\alpha}))(1-q^{-1/2}\epsilon_{\varrho,k}\chi^{1/2}(a_{\alpha}))}{1-\chi(a_{\alpha})}
&\text{$\alpha$ is long.}
\end{dcases}
\end{align*}
As in Lemma~\ref{lemma:3 specific computations of Shalika model}, if $\alpha$ is the long root $(i,n-i+1)$,
$\chi^{-1/2}(a_{\alpha})=q^{-2s_i}$.
\begin{theorem}\label{theorem:Casselman--Shalika formula for metaplectic H}
Let $\lambda\in\Z^k_{\geq}$. The function $\Lambda_{\chi}(\mathfrak{s}(t_{-\lambda})\varphi_{K,\chi^{-1}})$ vanishes unless $\lambda\in 2\Z^k_+$. In this case,
\begin{align*}
&\Lambda_{\chi}(\mathfrak{s}(t_{-\lambda})\varphi_{K,\chi^{-1}})
=Q^{-1}\sum_{w\in W_{\Sp_k}}
\prod_{\setof{\alpha>0}{w\alpha>0}}c_{\alpha}(\chi^{-1})
\prod_{\setof{\alpha\in\Sigma_{\Sp_k}^+}{w\alpha<0}}y_{\alpha}(\chi^{-1})
\delta^{1/2}_{B_n}\ \rconj{w}\chi(\mathfrak{s}(t_{\lambda})).
\end{align*}
More compactly, put
\begin{align*}
\beta(\chi^{-1})=&
\prod_{\alpha\in\Sigma_{\Sp_k}^{+}}\chi^{1/2}(a_{\alpha})
\prod_{\text{short }\alpha\in\Sigma_{\Sp_k}^{+}}(1-q^{-1}\chi^{-1}(a_{\alpha}))
\prod_{\text{long }\alpha\in\Sigma_{\Sp_k}^{+}}(1-q^{-1/2}\epsilon_{\varrho,k}\chi^{-1/2}(a_{\alpha})).
\end{align*}
Then
\begin{align*}
&\Lambda_{\chi}(\mathfrak{s}(t_{-\lambda})\varphi_{K,\chi^{-1}})=Q^{-1}c_{w_0}(\chi^{-1})\beta(\chi^{-1})^{-1}
\sum_{w\in W_{\Sp_k}}(-1)^{\ell_{\Sp_k}(w)}\beta(\rconj{w}\chi^{-1})\delta^{1/2}_{B_n}\ \rconj{w}\chi(\mathfrak{s}(t_{\lambda})).
\end{align*}
\end{theorem}
\begin{remark}
The product over $\chi^{1/2}(a_{\alpha})$ in this formula is well defined because
\begin{align*}
\frac{\prod_{\alpha\in\Sigma_{\Sp_k}^{+}}\rconj{w}\chi^{1/2}(a_{\alpha})}
{\prod_{\alpha\in\Sigma_{\Sp_k}^{+}}\chi^{1/2}(a_{\alpha})}=
\prod_{\setof{\alpha\in\Sigma_{\Sp_k}^{+}}{w\alpha<0}}\chi^{-1}(a_{\alpha}).
\end{align*}
\end{remark}
\begin{remark}
To obtain the formula for the (non-metaplectic) Shalika model, simply remove the product over long roots in the definition of $\beta(\chi^{-1})$.
Up to a factor independent of $\lambda$ this was the formula obtained by Sakellaridis \cite{Yia} (see also \cite[\S~5.5.2]{Yia2}).
\end{remark}
\begin{proof}
We use the notation of \S~\ref{subsection:Hironaka theorem}.
The space of $(H,\psi',\psi)$-functionals on $\mathrm{I}(\chi^{-1})$ is  at most one-dimensional. Therefore
\eqref{eq:matrix relation between family of functionals} becomes
\begin{align}\label{eq:def of A(w, chi) for Shalika model}
T_{w^{-1},\rconj{w}\chi^{-1}}^*\Lambda_{\chi}=A(w,\chi)\Lambda_{\rconj{w}\chi},
\end{align}
where $A(w,\chi)$ is a scalar (a rational function of $\chi$).
By virtue of Corollary~\ref{corollary:analog of Hironaka theory with phi I functions},
\begin{align*}
\Lambda_{\chi}(\mathfrak{s}(t_{-\lambda})\varphi_{K,\chi^{-1}})=\frac1{Q\mathrm{vol}(\mathcal{I})}\sum_{w\in W}\frac{c_{w_0}(\rconj{w_0w}\chi)}{c_{w^{-1}}(\rconj{w}\chi^{-1})}
A(w,\chi)(\Lambda_{\rconj{w}\chi}(\mathfrak{s}(t_{-\lambda})\varphi_{e,\chi^{-1}}).
\end{align*}
According to Lemma~\ref{lemma:3 specific computations of Shalika model},
\begin{align}\label{eq:evaluation of Lambda at varphi e in the proof of the theorem}
\Lambda_{\rconj{w}\chi}(\mathfrak{s}(t_{-\lambda})\varphi_{e,\chi^{-1}})
=\begin{dcases}
 \mathrm{vol}(\mathcal{I})\delta^{1/2}_{B_n}\ \rconj{w}\chi(\mathfrak{s}(t_{\lambda}))&\lambda\in 2\Z^k_+,\\
 0&\lambda\in\Z^k_+\setdifference2\Z^k_+.
 \end{dcases}
\end{align}
More precisely $\Lambda_{\rconj{w}\chi}$ is a functional on $\mathrm{I}(\rconj{w}\chi^{-1})$, and the domain of absolute convergence changes after the conjugation,
so actually we apply Corollary~\ref{corollary:3 specific computations of Shalika model meromorphic continuation} and deduce this equality in the sense of meromorphic continuation.

Therefore $\Lambda_{\chi}(\mathfrak{s}(t_{-\lambda})\varphi_{K,\chi^{-1}})$ vanishes, unless $\lambda\in 2\Z^k_+$. Henceforth assume this is the case.
Equality~\eqref{eq:evaluation of Lambda at varphi e in the proof of the theorem} also implies $\Lambda_{\rconj{w}\chi}(\varphi_{e,\chi^{-1}})\ne0$.
Therefore $A(w,\chi)$ is given by the quotient
\begin{align}\label{eq:A(w,chi) as a quotient}
A(w,\chi)=\frac{T_{w^{-1},\rconj{w}\chi^{-1}}^*\Lambda_{\chi}(\varphi_{e,\rconj{w}\chi^{-1}})}{\Lambda_{\rconj{w}\chi}(\varphi_{e,\rconj{w}\chi^{-1}})}
=\frac{\Lambda_{\chi}(T_{w^{-1},\rconj{w}\chi^{-1}}\varphi_{e,\rconj{w}\chi^{-1}})}{\mathrm{vol}(\mathcal{I})}.
\end{align}
We claim that only elements of $W_{\Sp_k}$ contribute to the sum.
\begin{claim}\label{claim:w not in Spk does not contribute}
If $w\in W\setdifference W_{\Sp_k}$, $A(w,\chi)=0$.
\end{claim}
Using the fact that $c_{w_0}(\rconj{w_0w}\chi)=c_{w_0}(\rconj{w}\chi^{-1})$, we obtain (this is \cite[(48)]{Yia})
\begin{align*}
\Lambda_{\chi}(\mathfrak{s}(t_{-\lambda})\varphi_{K,\chi^{-1}})=Q^{-1}\sum_{w\in W_{\Sp_k}}\prod_{\setof{\alpha>0}{w\alpha>0}}c_{\alpha}(\chi^{-1})
A(w,\chi)\delta^{1/2}_{B_n}\ \rconj{w}\chi(\mathfrak{s}(t_{\lambda})).
\end{align*}

It remains to compute $A(w,\chi)$ for $w\in W_{\Sp_k}$. Let $\ell_{\Sp_k}(w)=l$ and write
$w=s'_1\cdot\ldots \cdot s'_l$, where $s'_i$ is either $s_{\alpha_k}$ or $s_{\alpha_i}s_{\alpha_{n-i}}$ with $1\leq i<k$.
Then
\begin{align*}
A(w,\chi)=A(s'_{1},\rconj{s'_2\cdot\ldots\cdot s'_l}\chi)\ldots A(s'_{l-1},\rconj{s'_l}\chi)A(s'_l,\chi).
\end{align*}
We will compute $A(s'_i,\chi)$ succinctly using Claim~\ref{claim:analog of theorem 3.4 of Casselman} and Corollary~\ref{corollary:3 specific computations of Shalika model meromorphic continuation}.
\begin{claim}\label{claim:coefficients A for w in Spk}
We have $A(s_{\alpha_k},\chi)=y_{\alpha_k}(\chi^{-1})$ and
$A(s_{\alpha_i}s_{\alpha_{n-i}},\chi)=y_{\alpha_i}(\chi^{-1})$.
\end{claim}
We conclude that for any $w\in W_{\Sp_k}$,
\begin{align*}
A(w,\chi)=\prod_{\setof{\alpha\in\Sigma_{\Sp_k}^+}{w\alpha<0}}y_{\alpha}(\chi^{-1}).
\end{align*}
Plugging this into the last expression for $\Lambda_{\chi}(\mathfrak{s}(t_{-\lambda})\varphi_{K,\chi^{-1}})$ yields the theorem.
To obtain the compact form with $\beta(\chi^{-1})$ one can apply manipulations similar to \cite[(75)--(78)]{Yia}.
\end{proof}
\begin{proof}[Proof of Claim~\ref{claim:w not in Spk does not contribute}]
According to \eqref{eq:A(w,chi) as a quotient}
it is enough to show
that $T_{w^{-1},\rconj{w}\chi^{-1}}^*\Lambda_{\chi}$ vanishes on
$\varphi_{e,\rconj{w}\chi^{-1}}$. By \eqref{eq:Iwahori contained in B*H}
the support of $\varphi_{e,\rconj{w}\chi^{-1}}$ is contained in $\widetilde{B}_{n,*}H$.
Let $\Lambda$ denote the restriction of $T_{w^{-1},\rconj{w}\chi^{-1}}^*\Lambda_{\chi}$ to the subspace of functions in
$\mathrm{I}(\rconj{w}\chi^{-1})$, whose support is contained in $\widetilde{B}_{n,*}H$.
We prove $\Lambda=0$. The argument was adapted from \cite[Proposition~5.2]{Yia}.

Consider the space $\mathcal{S}^{\mathrm{gen}}(p^{-1}(B_{n,*}H),\delta^{1/2}_{B_n}\ \rconj{w}\chi^{-1})$ of complex-valued locally constant genuine functions $f$ on $p^{-1}(B_{n,*}H)$ such that
\begin{align}\label{eq:condition for BH w chi}
f(\mathfrak{s}(b_*)h)=\delta^{1/2}_{B_n}\ \rconj{w}\chi^{-1}(\mathfrak{s}(b_*))f(h),\qquad\forall b_*\in B_{n,*},h\in H.
\end{align}
Since $B_{n,*}H$ is open in $G_n$ (Claim~\ref{claim:B_{n,*}H is open}), $\Lambda$ is the restriction of $T_{w^{-1},\rconj{w}\chi^{-1}}^*\Lambda_{\chi}$ to this subspace.

Also consider $\mathcal{S}^{\mathrm{gen}}(\lmodulo{\mathfrak{s}(H\cap B_{n,*})}{\widetilde{H}},\rconj{w}\chi^{-1})$, where the functions satisfy \eqref{eq:condition for BH w chi} on $b_*\in B_{n,*}\cap H$, and note that
$\delta_{B_n}$ is trivial there.

These spaces are isomorphic: in one direction use restriction to $\widetilde{H}$, conversely one can extend a function $f$ in $\mathcal{S}^{\mathrm{gen}}(\lmodulo{\mathfrak{s}(H\cap B_{n,*})}{\widetilde{H}},\rconj{w}\chi^{-1})$ to the former space using
\begin{align*}
\mathfrak{s}(b_*)\mathfrak{h}(h)\mapsto \delta^{1/2}_{B_n}\ \rconj{w}\chi^{-1}(\mathfrak{s}(b_*))f(\mathfrak{h}(h))
\end{align*}
(to see this is well defined use \eqref{eq:the section h is almost a splitting}).

Thus we may regard $\Lambda$ as a functional on $\mathcal{S}^{\mathrm{gen}}(\lmodulo{\mathfrak{s}(H\cap B_{n,*})}{\widetilde{H}},\rconj{w}\chi^{-1})$. Then it lifts to a functional on $\mathcal{S}^{\mathrm{gen}}(\widetilde{H})$, by virtue of the projection
\begin{align*}
f^{B_{n,*}\cap H,\rconj{w}\chi^{-1}}(h)=\int_{B_{n,*}\cap H}\rconj{w}\chi(\mathfrak{s}(b_*))f(\mathfrak{s}(b_*)h)\, db_*,\qquad h\in \widetilde{H},f\in \mathcal{S}^{\mathrm{gen}}(\widetilde{H}).
\end{align*}
Furthermore, since $\Lambda$ is an $(H,\psi',\psi)$-functional, it is a scalar multiple of the functional
$l_H$ (defined after Claim~\ref{claim:h and s agree on H cap B_n*}).
Indeed, this follows from the Frobenius reciprocity: if we let $\iota$ be the identity character of
$\mu_2$ and regard $\gamma_{\psi'}\otimes\psi\otimes\iota$ as the genuine character of $\widetilde{H}$
taking $\mathfrak{h}(h)$ to $\gamma_{\psi'}(c_h)\psi(u_h)$, both $\Lambda$ and
$l_H$ belong to
\begin{align*}
\Hom_{\widetilde{H}}(\mathcal{S}^{\mathrm{gen}}(\widetilde{H}),\gamma_{\psi'}\otimes\psi\otimes\iota)
=\Hom_{\widetilde{H}}(\ind_{\mu_2}^{\widetilde{H}}(\iota),\gamma_{\psi'}\otimes\psi\otimes\iota)
=\Hom_{\mu_2}(\iota,\iota).
\end{align*}

Finally for $b_*\in B_{n,*}\cap H$ and $f\in \mathcal{S}^{\mathrm{gen}}(\widetilde{H})$, let $L(b_*)f(h)=f(\mathfrak{s}(b_*)h)$. Since $H$ is unimodular we see that
$l_H(L(b_*)f)=l_H(f)$, so assuming $\Lambda\ne0$, the same holds for $\Lambda$. But this contradicts the fact that $\Lambda$ factors through $\mathcal{S}^{\mathrm{gen}}(\lmodulo{\mathfrak{s}(H\cap B_{n,*})}{\widetilde{H}},\rconj{w}\chi^{-1})$, because for
$w\in W\setdifference W_{\Sp_k}$, $\rconj{w}\chi^{-1}$ is not trivial on $\mathfrak{s}(H\cap B_{n,*})$.
\end{proof}
\begin{proof}[Proof of Claim~\ref{claim:coefficients A for w in Spk}]
First note that by Claim~\ref{claim:analog of theorem 3.4 of Casselman}, for any $\alpha\in\Delta_{G_n}$,
\begin{align}\label{eq:applying 3.4 to alpha}
&T_{s_{\alpha}}\varphi_{e,\rconj{s_{\alpha}}\chi^{-1}}= (c_{\alpha}(\rconj{s_{\alpha}}\chi^{-1})-1)\varphi_{e,\chi^{-1}}+q^{-1}\varphi_{s_{\alpha},\chi^{-1}}.
\end{align}

Consider $A(s_{\alpha_k},\chi)$. Assume that $\mathfrak{w}$ represents $s_{\alpha_k}$. We apply $\Lambda_{\chi}$ to
\eqref{eq:applying 3.4 to alpha} with $\alpha=\alpha_k$.
By Corollary~\ref{corollary:3 specific computations of Shalika model meromorphic continuation},
\begin{align*}
\Lambda_{\chi}(T_{s_{\alpha_k}}\varphi_{e,\rconj{s_{\alpha_k}}\chi^{-1}})
=&\mathrm{vol}(\mathcal{I})(c_{\alpha_k}(\rconj{s_{\alpha_k}}\chi^{-1})-1)
\\&+q^{-1}\mathrm{vol}(\mathcal{I}\mathfrak{w}\mathcal{I})
(1-q^{-1}+(-\varrho,\varpi)_2(\varpi,\varpi)_2^{k-1}q^{-1/2}\chi^{-1/2}(a_{\alpha_k})).
\end{align*}
Since $c_{\alpha_k}(\rconj{s_{\alpha_k}}\chi^{-1})=c_{\alpha_k}(\chi)$ and
$\mathrm{vol}(\mathcal{I}\mathfrak{w}\mathcal{I})=q\cdot\mathrm{vol}(\mathcal{I})$, we find that
\begin{align*}
A(s_{\alpha_k},\chi)&=c_{\alpha_k}(\chi)-q^{-1}+(-\varrho,\varpi)_2(\varpi,\varpi)_2^{k-1}q^{-1/2}\chi^{-1/2}(a_{\alpha_k})\\
&=-\chi^{-1}(a_{\alpha_k})\frac{(1+q^{-1/2}\epsilon_{\varrho,k}\chi^{-1/2}(a_{\alpha_k}))(1-q^{-1/2}\epsilon_{\varrho,k}\chi^{1/2}(a_{\alpha_k}))}{1-\chi^{-1}(a_{\alpha_k})}=y_{\alpha_k}(\chi^{-1}).
\end{align*}

We turn to $A(s_{\alpha_i}s_{\alpha_{n-i}},\chi)$.
Applying Claim~\ref{claim:analog of theorem 3.4 of Casselman} twice,
\begin{align}\label{eq:T_w for short roots}
T_{s_{\alpha_i}s_{\alpha_{n-i}}}\varphi_{e,\rconj{s_{\alpha_i}s_{\alpha_{n-i}}}\chi^{-1}}=& (c_{\alpha_{n-i}}(\rconj{s_{\alpha_{n-i}}}\chi^{-1})-1)(c_{\alpha_i}(\rconj{s_{\alpha_i}s_{\alpha_{n-i}}}\chi^{-1})-1)\varphi_{e,\chi^{-1}}\\\nonumber
&+q^{-1}(c_{\alpha_i}(\rconj{s_{\alpha_i}s_{\alpha_{n-i}}}\chi^{-1})-1)\varphi_{s_{\alpha_{n-i}},\chi^{-1}}\\\nonumber
&+q^{-1}(c_{\alpha_{n-i}}(\rconj{s_{\alpha_{n-i}}}\chi^{-1})-1)\varphi_{s_{\alpha_i},\chi^{-1}}\\\nonumber
&+q^{-2}\varphi_{s_{\alpha_i}s_{\alpha_{n-i}},\chi^{-1}}.
\end{align}
As in \cite{Yia} the value of $\Lambda_{\chi}$ on the second and third summands is already determined:
Claim~\ref{claim:w not in Spk does not contribute} implies that $\Lambda_{\chi}$ vanishes on $T_{s_{\alpha_i}}\varphi_{e,\rconj{s_{\alpha_i}}\chi^{-1}}$ hence when we apply $\Lambda_{\chi}$ to
\eqref{eq:applying 3.4 to alpha} with $\alpha=\alpha_i$,
\begin{align*}
\Lambda_{\chi}(\varphi_{s_{\alpha_i},\chi^{-1}})=
q(1-c_{\alpha_i}(\rconj{s_{\alpha_i}}\chi^{-1}))\Lambda_{\chi}(\varphi_{e,\chi^{-1}})
=\mathrm{vol}(\mathcal{I})q(1-c_{\alpha_i}(\rconj{s_{\alpha_i}}\chi^{-1})).
\end{align*}
Similarly
\begin{align*}
\Lambda_{\chi}(\varphi_{s_{\alpha_{n-i}},\chi^{-1}})
=\mathrm{vol}(\mathcal{I})q(1-c_{\alpha_{n-i}}(\rconj{s_{\alpha_{n-i}}}\chi^{-1})).
\end{align*}
Also by Corollary~\ref{corollary:3 specific computations of Shalika model meromorphic continuation},
\begin{align*}
&\Lambda_{\chi}(\varphi_{s_{\alpha_i}s_{\alpha_{n-i}},\chi^{-1}})=
\mathrm{vol}(\mathcal{I}\mathfrak{w}\mathcal{I})
(1-q^{-1}+q^{-2}+(1-q^{-1})^2\frac{\chi^{-1}(a_{\alpha_i})}{1-\chi^{-1}(a_{\alpha_i})}).
\end{align*}
Note that now $\mathrm{vol}(\mathcal{I}\mathfrak{w}\mathcal{I})=q^2\mathrm{vol}(\mathcal{I})$.
Plugging these identities into \eqref{eq:T_w for short roots} and using the fact that $c_{\alpha_{n-i}}(\chi)=c_{\alpha_i}(\chi)$ (because $s_{n-i}=-s_{i+1}$ and $s_{n-i+1}=-s_i$) yields
\begin{align*}
&\Lambda_{\chi}(T_{s_{\alpha_i}s_{\alpha_{n-i}}}\varphi_{e,\rconj{s_{\alpha_i}s_{\alpha_{n-i}}}\chi^{-1}})\\&=\mathrm{vol}(\mathcal{I})(
-(c_{\alpha_i}(\chi)-1)^2+1-q^{-1}+q^{-2}+(1-q^{-1})^2\frac{\chi^{-1}(a_{\alpha_i})}{1-\chi^{-1}(a_{\alpha_i})}).
\end{align*}
We see that
\begin{align*}
A(s_{\alpha_i}s_{\alpha_{n-i}},\chi)=&-\chi^{-1}(a_{\alpha_i})\frac{1-q^{-1}\chi(a_{\alpha_i})}
{1-q^{-1}\chi^{-1}(a_{\alpha_i})}c_{\alpha_i}(\chi^{-1})^2=c_{\alpha_i}(\chi)c_{\alpha_i}(\chi^{-1})=y_{\alpha_i}(\chi^{-1}),
\end{align*}
completing the proof of the claim.
\end{proof}

The following corollary strengthens Proposition~\ref{proposition:meromorphic continuation and a basis of Shalika functionals}.
\begin{corollary}\label{corollary:existence of metaplectic Shalika functional}
Let $\chi$ be a character satisfying \eqref{eq:assumptions on character for H functional}. Then
$\Lambda_{\chi}$ is a nonzero $(H,\psi',\psi)$-functional on
$\mathrm{I}(\chi^{-1})$. Moreover it depends polynomially on $\chi$, in the sense that for a polynomial section $f_{\chi^{-1}}$, $\Lambda_{\chi}(f_{\chi^{-1}})$ belongs to  $\C[q^{-2s_1},\ldots,q^{-2s_k}]$.
\end{corollary}
\begin{proof}
According to Proposition~\ref{proposition:meromorphic continuation and a basis of Shalika functionals},
the functional $\Lambda_{\chi}$ is defined on $\mathrm{I}(\chi^{-1})$ by means of meromorphic continuation. It is
nonzero because it does not vanish on $\varphi_{e,\chi^{-1}}$ (Corollary~\ref{corollary:3 specific computations of Shalika model meromorphic continuation}).
It remains to show it does not have a pole. Claim~\ref{claim:Shalika or H functional are determined on the torus} and
Theorem~\ref{theorem:Casselman--Shalika formula for metaplectic H} show that there are
no poles on the $\widetilde{G}_n$-space spanned by $\varphi_{K,\chi^{-1}}$. Claim~\ref{claim:T_w isomorphism} implies that
for almost all $\chi$ (except for a discrete subset of $\C^k$), this space is all of $\mathrm{I}(\chi^{-1})$.
Since $\Lambda_{\chi}$ is determined by meromorphic continuation, the assertion holds.
\end{proof}
\begin{remark}\label{remark:Shalika functional vanishes on G-span of unramified vector}
The representation $\mathrm{I}(\chi^{-1})$ might be reducible. In this case it may happen that
\begin{align*}
\prod_{\setof{\alpha>0}{w\alpha>0}}c_{\alpha}(\chi^{-1})
\prod_{\setof{\alpha\in\Sigma_{\Sp_k}^+}{w\alpha<0}}y_{\alpha}(\chi^{-1})=0,\qquad\forall w\in W_{\Sp_k},
\end{align*}
and consequently $\Lambda_{\chi}$ will vanish on the $\widetilde{G}_n$-space spanned by $\varphi_{K,\chi^{-1}}$.
\end{remark}

We finally relate between the $(H,\psi',\psi)$-functional and the
metaplectic $(\psi',\psi)$-Shalika functional, on the unramified normalized element of $\mathrm{I}(\chi^{-1})$.
\begin{corollary}\label{corollary:Casselman--Shalika formula for metaplectic Shalika}
Assume $\chi$ satisfies \eqref{eq:assumptions on character for H functional} and let $l$ be
the functional on $\mathrm{I}(\chi^{-1})$ given by
\begin{align*}
l(f_{\chi^{-1}})=\Lambda_{\chi}(\mathfrak{s}(w)f_{\chi^{-1}}).
\end{align*}
\begin{enumerate}[leftmargin=*]
\item $l$ is a nonzero metaplectic $(\psi',\psi)$-Shalika functional.
\item $l$ depends polynomially on $\chi$.
\item $l(\mathfrak{s}(t_{\lambda})\varphi_{K,\chi^{-1}})=\Lambda_{\chi}(\mathfrak{s}(t_{-\lambda}))$ for all $\lambda\in\Z^k_{\geq}$. In particular $l(\mathfrak{s}(t_{\lambda})\varphi_{K,\chi^{-1}})$ vanishes on $\lambda\in\Z^k_{\geq}\setdifference2\Z^k_+$. Consequently,
Theorem~\ref{theorem:Casselman--Shalika formula for metaplectic H} gives the Casselman--Shalika formula for the metaplectic
Shalika model.
\end{enumerate}
\end{corollary}
\begin{proof}
We adapt the argument from \cite[(25)]{Yia}. 
We have $h\ \rconj{w}t_{\lambda}=t_{-\lambda}$ with
\begin{align*}
h=\diag(\varpi^{-\lambda},J_k\varpi^{-\lambda}J_k)\in H.
\end{align*}
Put $\mathfrak{h}(h)\rconj{\mathfrak{s}(w)}\mathfrak{s}(t_{\lambda})=\epsilon\mathfrak{s}(t_{-\lambda})$ for some
$\epsilon\in\mu_2$. Because $\varphi_{K,\chi^{-1}}$ is
unramified,
\begin{align*}
l(\mathfrak{s}(t_{\lambda})\varphi_{K,\chi^{-1}})=
\Lambda_{\chi}(\rconj{\mathfrak{s}(w)}\mathfrak{s}(t_{\lambda})\varphi_{K,\chi^{-1}})
=\epsilon\gamma_{\psi'}(\det c_h)^{-1}
\Lambda_{\chi}(\mathfrak{s}(t_{-\lambda})\varphi_{K,\chi^{-1}}).
\end{align*}
This vanishes unless $\lambda\in2\Z^k_+$, but then $\epsilon=1$ and $\gamma_{\psi'}(\det \varpi^{\lambda})=1$. The
result follows.
\end{proof}

\subsection{Archimedean uniqueness results}\label{subsection:Archimedean results}
In this section we establish uniqueness results for the metaplectic Shalika model of a principal series representation, analogous to those of Theorem~\ref{theorem:unramified Shalika functional 1 dim}. We assume $F=\R$, but the proof applies to $\C$ as well. In fact, we only need the real case, because we will apply this result to the exceptional representation, which is irreducible, and then over $\C$ uniqueness was proved in \cite{AGJ}.

Over $\R$ it is convenient to work with the Langlands decomposition $B_n=T_nN_n=\mathcal{M}_nA_nN_n$, where
\begin{align*}
&\mathcal{M}_n=\setof{\diag(t_1,\ldots,t_n)}{t_i\in\mu_2}\isomorphic\mu_2^n, \\ &A_n=\setof{\diag(t_1,\ldots,t_n)}{t_i\in\R_{>0}}\isomorphic\R_{>0}^n.
\end{align*}
Note that $A_n$ splits under the cover.
The group $\widetilde{\mathcal{M}}_n$ is a non-trivial extension of $\mathcal{M}_n$ which we now describe in more detail. Let $C(n)$ denote the Clifford algebra in $n$ generators $e_1,\ldots,e_n$ with relations $e_i^2=1$ and $e_ie_j+e_je_i=0$. Then
\begin{align*}
\widetilde{\mathcal{M}}_n \isomorphic\setof{\pm e_{i_1}\cdots e_{i_r}}{1\leq i_1<\ldots<i_r\leq n}.
\end{align*}
The covering map $\widetilde{\mathcal{M}}_n\to \mathcal{M}_n$ is given by $\pm e_{i_1}\cdots e_{i_r}\mapsto\varepsilon_I$ with $I=\{i_1,\ldots,i_r\}$, where $\varepsilon_I$ is the diagonal matrix having $-1$ on the diagonal at positions $i_1,\ldots,i_r$, and $1$ at the remaining ones. Further, the section $\mathfrak{s}$ is on $\mathcal{M}_n$ given by
\begin{align*}
\mathfrak{s}(\varepsilon_I) = e_{i_r}\cdots e_{i_1} = (-1)^{\frac{r(r-1)}{2}}e_{i_1}\cdots e_{i_r}.
\end{align*}
\begin{claim}
The group $\widetilde{\mathcal{M}}_n$ has $2^n+1$ irreducible representations, $2^n$ non-genuine characters and one genuine representation of dimension $2^k$. The genuine representation $\sigma$ is the restriction of the pin-representation of $\Pin(n)\subset C(n)$ to $\widetilde{\mathcal{M}}_n$.
\end{claim}
\begin{proof}
It is easy to see that $\widetilde{\mathcal{M}}_n$ has $2^n+1$ conjugacy classes and hence $2^n+1$ irreducible representations. Since $\mathcal{M}_n\cong\mu_2^n$ there are $2^n$ non-isomorphic characters of $\mathcal{M}_n$ which lift to non-isomorphic characters of $\widetilde{\mathcal{M}}_n$. Since the squares of the dimensions of all irreducible representations of $\widetilde{\mathcal{M}}_n$ sum up to the order of $\widetilde{\mathcal{M}}_n$ the last remaining irreducible representation must have dimension $2^k$. The restriction of the pin-representation to $\widetilde{\mathcal{M}}_n$ is easily seen to be genuine and has dimension $2^k$, hence has to be irreducible by a complete reducibility argument.
\end{proof}

The genuine principal series of $\widetilde{G}_n$ is given by
\begin{align*}
J(\nu)=\Ind_{\widetilde{\mathcal{M}}_nA_nN_n}^{\widetilde{G}_n}(\sigma\otimes e^{\nu+\rho}\otimes 1),
\end{align*}
where $\nu=(\nu_1,\ldots,\nu_n)\in\C^n$ corresponds to the character of the Lie algebra $\mathfrak{a}_n$ of $A_n$ given by $\nu(\diag(t_1,\ldots,t_n))=\sum_{j=1}^n\nu_j t_j$, and $\rho=(\rho_1,\ldots,\rho_n)=(\frac{n-1}{2},\frac{n-3}{2},\ldots,\frac{1-n}{2})$ denotes half the sum of all positive roots.
\begin{theorem}\label{theorem:uniqueness of metaplectic Shalika over R}
Assume that $\nu\in\C^n$ satisfies the condition
$$ \nu_i+\nu_j\neq0, \qquad \forall1\leq i\leq k,\,1\leq j\leq n\text{ with }j\neq i,n-i+1. $$
Then the space of metaplectic Shalika functionals on $J(\nu)$ is at most one-dimensional. Furthermore, it is trivial unless $\nu_i+\nu_{n-i+1}=0$ for all $1\leq i\leq k$.
\end{theorem}
\begin{proof}
By Archimedean Bruhat theory we have for any two closed subgroups $H_1,H_2\subset\widetilde{G}_n$ and any two (finite-dimensional) representations $\xi_1$ and $\xi_2$ of $H_1$ and $H_2$ that
\begin{multline}
 \Dim\Hom_{\widetilde{G}_n}(\Ind_{H_1}^{\widetilde{G}_n}(\xi_1\delta_{H_1}^{1/2}),\Ind_{H_2}^{\widetilde{G}_n}(\xi_2\delta_{H_2}^{1/2}))\\
 \leq \sum_{H_1gH_2}\sum_{m=0}^\infty \Dim\Hom_{H(g)}(\xi_1\otimes(\xi_2^g)^\vee,\frac{\delta_{H(g)}}{(\delta_{H_1}\delta_{H_2}^g)^{1/2}}\otimes S^m(V(g))),\label{eq:GeneralBruhatDimensionEstimate}
\end{multline}
where the sum is over all double cosets $H_1gH_2\in H_1\backslash\widetilde{G}_n/H_2$, and $H(g)=H_1\cap(gH_2g^{-1})$, $V(g)=\mathfrak{g}_n/(\mathfrak{h}_1+\Ad(g)\mathfrak{h}_2)$ and $\xi^g(h)=\xi(g^{-1}hg)$ for $h\in H(g)$ and $\xi$ a representation of $H_2$. Here $\mathfrak{g}_n$, $\mathfrak{h}_1$ and $\mathfrak{h}_2$ denote the Lie algebras of $\widetilde{G}_n$, $H_1$ and $H_2$. Further, $\delta_{H(g)}$ is the modulus character of $H(g)$. In our case $H_1=\widetilde{G}_k^\triangle U_k$, $H_2=\widetilde{B}_n$ and $\xi_1=\gamma_{\psi'}\otimes\psi$, $\xi_2=\sigma\otimes e^\nu$, $\delta_{H_1}=1$, $\delta_{H_2}=e^{2\rho}$. Further, the double coset space $\widetilde{G}_k^\triangle U_k\backslash\widetilde{G}_n/\widetilde{B}_n$ is parametrized by $\diag(S_k)\backslash S_n$, where $S_n$ is the symmetric group in $n$ letters and $\diag(S_k)\subset S_n$ the diagonal embedding of the symmetric group $S_k$ acting on the first and last $k$ letters in the same fashion. For such a permutation matrix $g$ we write $g=(\delta_{i,j_i})_{i=1,\ldots,n}$, then $ge_{j_i}=e_i$.

For a permutation matrix $g\in G_n$ and $m\geq0$ we put
$$ \mathcal{H}(g,m) = \Hom_{H(g)}((\gamma_{\psi'}\otimes\psi)\otimes((\sigma^g)^\vee\otimes e^{-g(\nu-\rho)}),\delta_{H(g)}\otimes S^m(V(g))), $$
the $\Hom$-space in \eqref{eq:GeneralBruhatDimensionEstimate}. We claim that $\mathcal{H}(g,m)=0$ unless $m=0$ and
$$ g=g_0=\left(\begin{array}{cc}&g_1\\g_2&\end{array}\right), \qquad \text{where} \qquad (g_2^{-1}g_1)e_i=e_{k-i+1},\ i=1,\ldots,k. $$
Note that $H_1g_0H_2$ is the same equivalence class for all choices of $g_1$ and $g_2$. Further, for this particular $g=g_0$ we have $\mathcal{H}(g_0,0)=0$ unless $\nu_i+\nu_{n-i+1}=0$ for all $i=1,\ldots,k$, and in this case $\Dim\mathcal{H}(g_0,0)=1$. We prove these claims in 3 steps.

\begin{enumerate}[leftmargin=*]
\item First, the same analysis as in the proof of Theorem~\ref{theorem:unramified Shalika functional 1 dim} shows that for $m=0$ we have $\mathcal{H}(g,0)=0$ for $g\neq g_0$.
\item Next we claim that whenever $V(g)\neq0$ then $\mathcal{H}(g,m)=0$ for all $m>0$. The $H(g)$-representation $V(g)=\mathfrak{g}_n/(\mathfrak{h}_1+\Ad(g)\mathfrak{h}_2)$ is the sum of the equivalence classes of $E_{\ell,m}$ with either $\ell,m\in\{1,\ldots,k\}$, $\ell,m\in\{k+1,\ldots,n\}$, or $\ell\in\{k+1,\ldots,n\}$ and $m\in\{1,\ldots,k\}$. We claim that if the equivalence class of some $E_{\ell,m}$ in $V(g)$ is nonzero, then there exists a root subgroup of $H(g)$ that acts nilpotently on $E_{\ell,m}$. This root subgroup acts on the left hand side of $\mathcal{H}(g,m)$ either trivially or by the character $\psi$, and on the right hand side it only acts on $S^m(V(g))$. Hence, no homomorphism in $\mathcal{H}(g,m)$ can have a component of $E_{\ell,m}$ in its image. This is true for all generators of $V(g)$ and hence $\mathcal{H}(g,m)=0$ for $m>0$.\\
It remains to show that every nonzero equivalence class $E_{\ell,m}$ in $V(g)$ has a root subgroup in $H(g)$ acting nilpotently on it. First, if $\ell\in\{k+1,\ldots,n\}$ and $m\in\{1,\ldots,k\}$ with $E_{\ell,m}$ nonzero in $V(g)$, then $E_{\ell,m}\notin\Ad(g)\mathfrak{h}_2$. Hence $E_{j_\ell,j_m}\notin\mathfrak{h}_2$ which means $j_\ell>j_m$. But then $E_{j_m,j_\ell}\in\mathfrak{h}_2$ and therefore $E_{m,\ell}\in\Ad(g)\mathfrak{h}_2.$ Since also $E_{m,\ell}\in\mathfrak{h}_1$ we have $E_{m,\ell}\in\mathfrak{h}(g)$, so $H(g)$ contains the root subgroup $N_{m,\ell}$ which acts nilpotently on $E_{\ell,m}$.\\
Next, if $\ell,m\in\{1,\ldots,k\}$ with $E_{\ell,m}$ non-trivial in $V(g)$, then $E_{\ell,m}\notin\Ad(g)\mathfrak{h}_2$ and also $E_{\ell+k,m+k}\notin\Ad(g)\mathfrak{h}_2$, because $E_{\ell,k}+E_{\ell+k,m+k}\in\mathfrak{h}_1$. Hence $E_{j_\ell,j_m},E_{j_{\ell+k},j_{m+k}}\notin\mathfrak{h}_2$. As before this implies $E_{j_m,j_\ell},E_{j_{m+k},j_{\ell+k}}\in\mathfrak{h}_2$ and hence $E_{m,\ell},E_{m+k,\ell+k}\in\Ad(g)\mathfrak{h}_2$. Therefore $E_{m+\ell}+E_{m+k,\ell+k}\in\mathfrak{h}(g)$ and the corresponding root subgroup acts nilpotently on $E_{\ell,m}$. We note that in this case $\ell\neq m$. This shows the claim.
\item Finally we show that $\mathcal{H}(g_0,0)$ is non-trivial if and only if $\nu_i+\nu_{n-i+1}=0$ for all $i=1,\ldots,k$, and in this case $\Dim\mathcal{H}(g_0,0)=1$. For $g=g_0$ we have $H(g)=\widetilde{(\mathcal{M}_kA_k)}^\triangle$ which is the direct product of $A_k^\triangle$ and
$$ \widetilde{\mathcal{M}}_k^\triangle = \{\pm e_{i_1}e_{i_1+k}\cdots e_{i_r}e_{i_r+k}:1\leq i_1<\ldots<i_r\leq k,\,r\in\{0,\ldots,k\}\}. $$
Hence $\delta_{H(g)}=1$, $\psi|_{H(g)}=1$ and we have
\begin{align*}
 \mathcal{H}(g,0) &= \Hom_{H(g)}(\gamma_{\psi'}\otimes(\sigma^g)^\vee\otimes e^{-g(\nu-\rho)},\C)\\
  &= \Hom_{\widetilde{\mathcal{M}}_k^\triangle}(\gamma_{\psi'},\sigma^g)\otimes\Hom_{A_k^\triangle}(e^{-g(\nu-\rho)},\C).
\end{align*}
The second factor is nonzero if and only if $\nu_i+\nu_{n-i+1}=\rho_i+\rho_{n-i+1}=0$ for all $i=1,\ldots,k$, in which case it is one-dimensional. It remains to show that the first factor is one-dimensional. For this we use some Clifford analysis.\\
The $k$-dimensional subspace $X\subset\C^n$ spanned by $f_i=e_i+\sqrt{-1}e_{i+k}$, $i=1,\ldots,k$, is isotropic with respect to the standard bilinear form on $\C^n$. Let $X'\subset\C^n$ be the subspace spanned by $f_i'=e_i-\sqrt{-1}e_{i+k}$, $i=1,\ldots,k$, then $\C^n=X\oplus X'$. There is an algebra isomorphism $\Phi:C(n)\to\End(\Lambda^*\,X),\,x\mapsto\Phi_x$ which is on $X,X'\subset C(n)$ given by
\begin{align*}
 \Phi_x(x_1\wedge\cdots\wedge x_r) &= x\wedge x_1\wedge\cdots\wedge x_r, && x\in X,\\
 \Phi_{x'}(x_1\wedge\cdots\wedge x_r) &= 2\sum_{i=1}^r(-1)^{i-1}(x'|x_i)x_1\wedge\cdots\wedge\widehat{x}_i\wedge\cdots\wedge x_r, && x'\in X',
\end{align*}
i.e. $X$ acts by exterior multiplication and $X'$ by contractions. The pin-representation $\sigma$ of $\widetilde{\mathcal{M}}_n$ on $\Lambda^*\,X$ is given by $\sigma(x)=\Phi_x$. We compute the action of $\widetilde{\mathcal{M}}_k^\triangle$ on $\Lambda^*\,X$. For this we use the basis of $\Lambda^*\,X$ given by $f_I=f_{i_1}\wedge\ldots\wedge f_{i_r}$ with $I=\{i_1,\ldots,i_r\}\subset\{1,\ldots,k\}$, $1\leq i_1<\ldots<i_r\leq k$. First, each monomial $e_ie_{i+k}$ acts on the basis elements as by
\begin{align*}
 \sigma(e_ie_{i+k})(f_{i_1}\wedge\ldots f_{i_r}) &= \frac{1}{4\sqrt{-1}}\sigma(f_i'f_i-f_if_i')(f_I) = \frac{1}{4\sqrt{-1}}(\Phi_{f_i'}\Phi_{f_i}-\Phi_{f_i}\Phi_{f_i'})(f_I)\\
 &= \frac{1}{4\sqrt{-1}}(2\Phi_{f_i'}\Phi_{f_i}-\Phi_{f_i'f_i+f_if_i'})(f_I) = \frac{1}{4\sqrt{-1}}(2\Phi_{f_i'}\Phi_{f_i}-4)(f_I)\\
 &= \frac{1}{\sqrt{-1}}f_I\times\begin{cases}+1&\mbox{if $i\notin I$,}\\-1&\mbox{if $i\in I$.}\end{cases}
\end{align*}
Hence,
$$ \sigma(\pm e_{i_1}e_{i_1+k}\cdots e_{i_r}e_{i_r+k})f_I = \pm (-\sqrt{-1})^r (-1)^{\#(I\cap\{i_1,\ldots,i_r\})}. $$
Note that $\sigma^g|_{\widetilde{\mathcal{M}}_k^\triangle}=\sigma|_{\widetilde{\mathcal{M}}_k^\triangle}$. To determine the action of $\widetilde{\mathcal{M}}_k^\triangle$ under $\gamma_{\psi'}$ we assume that $\gamma_{\psi'}(-1)=\sqrt{-1}$. The case of $\gamma_{\psi'}(-1)=-\sqrt{-1}$ is handled similarly. Then
\begin{align*}
 \gamma_{\psi'}(\pm e_{i_1}e_{i_1+k}\cdots e_{i_r}e_{i_r+k}) &= \pm(-1)^{r+\frac{r(r-1)}{2}}\gamma_{\psi'}(e_{i_r+k}\cdots e_{i_1+k}e_{i_r}\cdots e_{i_1})\\
 &= \pm(-1)^{r+\frac{r(r-1)}{2}}\gamma_{\psi'}(\mathfrak{s}(\varepsilon_{I\cup(I+k)})\\
 &= \pm(-1)^{r+\frac{r(r-1)}{2}}\gamma_{\psi'}((-1)^r)\\
 &= \pm(-\sqrt{-1})^r.
\end{align*}
Hence, an intertwining operator in $\Hom_{\widetilde{\mathcal{M}}_k^\triangle}(\gamma_{\psi'},\sigma^g)$ can only map onto $f_I$ for the empty set $I$, so this space is one-dimensional.\qedhere
\end{enumerate}
\end{proof}

\section{Metaplectic Shalika model for exceptional representations}\label{section:Metap Shalika model for theta}
\subsection{Exceptional representations}\label{subsection:The exceptional representations}
The exceptional representations were introduced and studied by Kazhdan and Patterson \cite{KP}, locally and globally.
Over a local field, the exceptional representation $\theta_n$ of $\widetilde{G}_n$ is the unique irreducible quotient of a principal series representation corresponding to a genuine character, which is a lift
of $\delta^{1/4}_{B_n}|_{T_n^2}$ to $C_{\widetilde{T}_n}$ (in the notation of
\S~\ref{section:Casselman's basis and a result of Hironaka} $\theta$ is a quotient of $\mathrm{I}(\chi)$ where
$\chi|_{T_n^2}=\delta^{1/4}_{B_n}$). When $n$ is odd, there are $[F^*:F^{*2}]$ possible lifts of $\delta^{1/4}_{B_n}|_{T_n^2}$ to $C_{\widetilde{T}_n}$, resulting in non-isomorphic representations, but
$\theta_{2k}$ is unique.

Any exceptional representation of $\widetilde{G}_n$ can be obtained by twisting
$\theta_n$ with a character of $F^*$ (pulled back to $G_n$ via $\det$). Here for simplicity we take this character to be trivial.

Over a global field we denote this representation by $\Theta_{n}$, it has an automorphic realization
as the residual representation of an Eisenstein series on $\lmodulo{B_n(F)}{G_n(F)}$. Abstractly, it is isomorphic to the restricted tensor product of local exceptional representations. For a detailed description of the properties of exceptional representations see \cite{KP,BG,Kable}.

\begin{remark}
Over $\C$ the group $G_n$ splits under the cover, then the results of Vogan
\cite[Lemma 2.4, Corollary 11.6 and Lemma 11.11]{Vog86} show that $\theta_{2k}$ is isomorphic to the Stein complementary series representation
$\Ind_{Q_k}^{G_{2k}}(\delta_{Q_k}^{1/2+1/(2k)})\cong\Ind_{Q_k}^{G_{2k}}(\delta_{Q_k}^{1/2-1/(2k)})$.
In particular it is unitary, its $K$-types have multiplicity one, and it contains the trivial $K$-type. In the real case this is not true since $\theta_{2k}$ is genuine and $\widetilde{Q}_k$ does not have any genuine characters to induce from.
\end{remark}

\subsection{Local metaplectic Shalika model}\label{subsection:The metaplectic Shalika model of theta}
Let $F$ be a $p$-adic field, $n=2k$ and $\theta=\theta_{n}$.
In \cite[Theorem~3.1]{me11} we proved that $\theta_{U_k,\psi}$ is the one-dimensional representation of $\widetilde{G}_k^{\triangle}$ given by
\begin{align*}
\epsilon\mathfrak{s}(c^{\triangle})\rightarrow\epsilon\gamma_{\psi,(-1)^{k}}(\det c),\qquad \epsilon\in\mu_2.
\end{align*}
In particular $\theta$ admits a unique (up to scalar multiplication) metaplectic $(\psi_{(-1)^{k}},\psi)$-Shalika functional and a unique metaplectic Shalika model
$\mathscr{S}(\theta,\psi_{(-1)^{k}},\psi)$.

Assume all data are unramified as in \S~\ref{subsection:unramified representations}. Then $\theta$ is also unramified. The unramified normalized function
$\mathscr{S}\in\mathscr{S}(\theta,\psi_{(-1)^{k}},\psi)$ is the unique function which is identically $1$ on
$\kappa(K)$. 
\begin{lemma}\label{lemma:unramified normalized Shalika function of theta}
For any $\lambda\in\Z^k_{\geq}$,
\begin{align*}
&\mathscr{S}(\mathfrak{s}(t_{\lambda}))=\begin{cases}
\delta^{1/4}_{B_n}(t_{\lambda})&\lambda\in2\Z^k_+,\\
0&\text{otherwise.}
\end{cases}
\end{align*}
\end{lemma}
\begin{proof}
In the notation of \S~\ref{section:Casselman's basis and a result of Hironaka}, $\theta$ is a
quotient of $\mathrm{I}(\chi)$ where $\chi=\chi_{\underline{s}}=\delta_{B_n}^{1/4}$,
\begin{align*}
\underline{s}=\frac14(n-1,n-3,\ldots,3-n,1-n).
\end{align*}
This character satisfies \eqref{eq:assumptions on character for H functional}. Put $\psi'=\psi_{\varrho}$ with
$\varrho=(-1)^{k}$.
On the one hand let $l_{\theta}$ be the
metaplectic $(\psi',\psi)$-Shalika functional on $\theta$ such that $\mathscr{S}(g)=l_{\theta}(\theta(g)\varphi_{K,\chi})$. Clearly it extends to a similar nonzero
functional on $\mathrm{I}(\chi)$. 

On the other hand let $l$ be the nonzero metaplectic $(\psi',\psi)$-Shalika
functional on $\mathrm{I}(\chi)$, whose existence was guaranteed by
Corollary~\ref{corollary:Casselman--Shalika formula for metaplectic Shalika}. By Theorem~\ref{theorem:unramified Shalika functional 1 dim}
the functionals $l_{\theta}$ and $l$ are proportional, in particular $l$ is nonzero on $\varphi_{K,\chi}$, so that
$l_{\theta}=l(\varphi_{K,\chi})^{-1}l$ and
\begin{align*}
\mathscr{S}(g)=l(\varphi_{K,\chi})^{-1}l(g\varphi_{K,\chi}),\qquad\forall g\in\widetilde{G}_n.
\end{align*}

Looking at the formula of Theorem~\ref{theorem:Casselman--Shalika formula for metaplectic H} for
$\mathrm{I}(\chi)$ and $\varphi_{K,\chi}$, the coefficient corresponding
to $w\in W_{\Sp_k}$ takes the form
\begin{align*}
Q^{-1}\prod_{\setof{\alpha>0}{w\alpha>0}}c_{\alpha}(\chi)
\prod_{\setof{\alpha\in\Sigma_{\Sp_k}^+}{w\alpha<0}}y_{\alpha}(\chi).
\end{align*}
Observe that none of the terms here has a pole, and by definition for any $\alpha\in\Delta_{G_n}$ (a simple root),
$\chi(a_{\alpha})=q^{-1}$ so that $c_{\alpha}(\chi^{-1})=0$. Now consider the product of
factors $y_{\alpha}(\chi)$. If $w\alpha<0$ for some short root $\alpha\in\Sigma_{\Sp_k}^+$,
$y_{\alpha}(\chi)=c_{\alpha}(\chi)c_{\alpha}(\chi^{-1})=0$. Otherwise if $w\ne e$, we must have $w\alpha_k<0$. But then
since
\begin{align*}
\epsilon_{\varrho,k}=(-(-1)^k,\varpi)_2(\varpi,\varpi)_2^{k-1}=1,
\end{align*}
\begin{align*}
y_{\alpha_k}(\chi)=
\frac{(1+q^{-1}\epsilon_{\varrho,k})
(1-\epsilon_{\varrho,k})}{1-q}=0.
\end{align*}
Therefore all coefficients except the one for $w=e$ vanish and for $\lambda\in2\Z^k_+$,
\begin{align*}
l(\mathfrak{s}(t_{\lambda})\varphi_{K,\chi})=Q^{-1}c_{w_0}(\chi)\delta^{1/2}_{B_n}\ \chi^{-1}(\mathfrak{s}(t_{\lambda}))
=Q^{-1}c_{w_0}(\chi)\delta^{1/4}_{B_n}(t_{\lambda}).
\end{align*}
We conclude
$l_{\theta}=Qc_{w_0}(\chi)^{-1}l$ and the result immediately follows. 
\end{proof}
\begin{remark}
The proof implies that $l$ factors through the quotient $\theta$, which is also the image of the intertwining operator $T_{w_0}$. Since
by \eqref{eq:Gindikin-Karpelevich formula}, $T_{w_0}\varphi_{K,\chi}=c_{w_0}(\chi)\varphi_{K,\chi^{-1}}$, it is  natural to
find $c_{w_0}(\chi)$ in the normalization.
\end{remark}
\begin{remark}
One could attempt to consider $\theta$ as a subrepresentation of
$\mathrm{I}(\chi^{-1})$. As such, it contains $\varphi_{K,\chi^{-1}}$ and the formula of
Theorem~\ref{theorem:Casselman--Shalika formula for metaplectic H} for
$(\mathrm{I}(\chi^{-1}),\varphi_{K,\chi^{-1}})$ may be used.
However unless $k=1$, all the coefficients, which now take the form
\begin{align*}
Q^{-1}\prod_{\setof{\alpha>0}{w\alpha>0}}c_{\alpha}(\chi^{-1})
\prod_{\setof{\alpha\in\Sigma_{\Sp_k}^+}{w\alpha<0}}y_{\alpha}(\chi^{-1}),
\end{align*}
vanish (if $w=w_0$ and $k>1$, $y_{\alpha_1}(\chi^{-1})=0$). In other words, the functional $\Lambda_{\chi}$ vanishes
on $\theta$. This is the phenomena described in
Remark~\ref{remark:Shalika functional vanishes on G-span of unramified vector}. This also means that
the metaplectic Shalika functional on $\theta$ does not extend to $\mathrm{I}(\chi^{-1})$.
\end{remark}

Now assume $F=\R$. The representation $\theta$ is the unique irreducible quotient of a principal series representation, which satisfies the assumptions of Theorem~\ref{theorem:uniqueness of metaplectic Shalika over R}. Hence the
space of metaplectic Shalika functionals on $\theta$ is at most
one-dimensional. This remains true over $\C$, because
then the cover is split over $G_n$ and since $\theta$ is an irreducible admissible smooth Fr\'{e}chet representation, the bound follows from \cite{AGJ}. We will prove that the space of metaplectic Shalika functionals on $\theta$ is precisely one-dimensional over both $\R$ and $\C$: existence of the functional will follow from a local--global
argument in \S~\ref{section:Fourier coefficients} (Corollary~\ref{corollary:metaplectic Shalika functional in Archimedean places}).
\subsection{Fourier coefficients of exceptional representations}\label{section:Fourier coefficients}
Let $F$ be a global field and $\Theta=\Theta_{n}$. For an automorphic form $\varphi$ in the space of $\Theta$, a unipotent subgroup $U$, and a character $\psi$ of $U(\Adele)$ which is trivial on $U(F)$, let $\varphi^{U,\psi}$ be the corresponding Fourier coefficient, given by
\begin{align}\label{eq:Fourier coefficient of varphi}
\varphi^{U,\psi}(g)=\int_{\lmodulo{U(F)}{U(\Adele)}}\varphi(ug)\psi^{-1}(u)\, du,\qquad g\in \widetilde{G}_n(\Adele).
\end{align}
These coefficients are related to local twisted Jacquet modules at the finite places. More explicitly, if a Fourier coefficient does not vanish identically on $\Theta$, then for each place $v<\infty$ of $F$,
$(\Theta_v)_{U,\psi}\ne0$. Conversely if $(\Theta_v)_{U,\psi}=0$ for some $v<\infty$, the coefficient \eqref{eq:Fourier coefficient of varphi} is zero for any $\varphi$ and $g$ (see e.g., \cite[Proposition~1]{JR}).

Let $\theta=\Theta_v$ for some $v<\infty$.
In \cite{me11} we described the twisted Jacquet modules of $\theta$ along the unipotent subgroups $U_k$. We use these results to deduce several properties of the Fourier coefficients.

For $0<k<n$, $M_k<G_n$ acts on the set of characters of $U_k$, with $\min(k,n-k)+1$ orbits. We choose representatives
\begin{align*}
\psi_j(u)=\psi(\sum_{i=1}^ju_{k-i+1,k+j-i+1}),\quad 0\leq j\leq \min(k,n-k).
\end{align*}
(Here $u$ is regarded as an $n\times n$ matrix.)
In particular $\psi_0$ is trivial and when $n=2k$, $\psi_k$ is the Shalika character.
For $n=5$ and $k=3$, say, $\psi_1(u)=\psi(u_{3,4})$ and $\psi_2(u)=\psi(u_{3,5}+u_{2,4})$.
The stabilizer of $\psi_j$ in $M_k$ is
\begin{align*}
\mathrm{St}_{n,k}(\psi_j)=\left\{\left(\begin{array}{cccc}b&v\\&c\\&&c&y\\&&&d\end{array}\right):b\in G_{k-j},c\in G_j,d\in G_{n-k-j}\right\}.
\end{align*}
Let $V_j$ and $Y_j$ be the unipotent subgroups of $\mathrm{St}_{n,k}(\psi_j)$ defined by the coordinates of $v$ and $y$.
Note that $V_0$ and $Y_0$ are trivial.

The Jacquet module $\theta_{U_k,\psi_j}$ is a representation of
$p^{-1}(\mathrm{St}_{n,k}(\psi_j))$. The following theorem was proved in \cite{me11} (with minor notational differences):
\begin{theorem}\label{theorem:intro Jacquet modules}
For any $0\leq j\leq \min(k,n-k)$, the representation $\theta_{U_k,\psi_j}$ is isomorphic to
$\delta^{1/4}(\theta_{k-j}\widetilde{\otimes}\theta_{n-k-j})\otimes\gamma_{\psi,(-1)^{j}}\psi_j$.
Here $\delta$ is the modulus character of the standard parabolic subgroup of $G_n$ corresponding to the partition $(k-j,2j,n-k-j)$; $\widetilde{\otimes}$ is the metaplectic tensor of Kable \cite{Kable}; $\gamma_{\psi,(-1)^{j}}$ is a genuine character of $\widetilde{G}_j^{\triangle}$ and $\gamma_{\psi,(-1)^{j}}\psi_j$ is a one-dimensional representation of $\widetilde{G}_j^{\triangle}\ltimes U_{j}$, where $G_j^{\triangle}\ltimes U_j$ is regarded as a subgroup of the middle block of the partition ($G_j^{\triangle}$ is identified with the coordinates of $c$).
\end{theorem}
This theorem implies in particular that $V_j(F_v)$ and $Y_j(F_v)$ act trivially on $\theta_{U_k,\psi_j}$, and
when $n=2k=2j$, $\theta$ admits a unique metaplectic $(\psi_{(-1)^{k}},\psi)$-Shalika model (as mentioned in
\S~\ref{subsection:The metaplectic Shalika model of theta}). The case $j=0$ of Theorem~\ref{theorem:intro Jacquet modules} was already proved by Kable \cite[Theorem~5.1]{Kable}, in this case $\delta=\delta_{Q_k}$ and the result can be used to read off the exponents of $\theta$ (see the proof of Theorem~\ref{theorem:local props GJ integral p-adic} in \S~\ref{subsection:Local p-adic Godement--Jacquet integral}).
\begin{proposition}\label{proposition:Fourier coefficients are trivial along v and y}
For any $\varphi$ in the space of $\Theta$, $0<k<n$ and $0\leq j\leq \min(k,n-k)$, $\varphi^{U_k,\psi_j}$ is trivial on $V_j(\Adele)$ and $Y_j(\Adele)$.
\end{proposition}
\begin{proof}
Write the Fourier expansion of $\varphi^{U_k,\psi_j}$ along $\lmodulo{V_j(F)}{V_j(\Adele)}$. Any coefficient with respect to a nontrivial character vanishes, because by virtue of Theorem~\ref{theorem:intro Jacquet modules}, at any $v<\infty$, $V_j(F_v)$ acts trivially on $(\Theta_v)_{U_k,\psi_j}$.
\end{proof}

The next theorem describes the constant term of automorphic forms in the space of $\Theta$, along any unipotent radical of a standard parabolic subgroup. The constant term defines an automorphic representation of the Levi part, and as such, it is often convenient to identify it with a tensor representation. Unfortunately, this is not simple to do for covering groups of $G_n$, because the preimages in the cover, of the direct factors of the Levi part, do not commute (see e.g., \cite{BG,Kable,Tk}). To avoid this problem, at least to some extent, we define exceptional representations of Levi subgroups. This approach was suggested by Bump and Ginzburg \cite[\S~1]{BG}, and was perhaps implicit in \cite[\S~5]{Kable}.
Let $\Delta\subset\Delta_{G_n}$ and $Q=Q_{\Delta}$ be the corresponding standard parabolic subgroup, $Q=M\ltimes U$.
Let $B_M<M$ be the Borel subgroup such that $B_n=B_M\ltimes U$, $B_M=T_n\ltimes N_M$. Denote by $\delta_{B_M}$ the modulus character of $B_M$, as a subgroup of $M$. Also let $W_M$ be the Weyl group of $M$ and $w_M$ its longest element.

Consider a local context first. A genuine character $\chi$ of $C_{\widetilde{T}_n}$ is called
$M$-exceptional if $\chi(a_{\alpha})=|q|$ for all $\alpha\in\Delta$
($a_{\alpha}$ was defined after \eqref{eq:one dim for regular hom}).
The representation $\Ind_{\widetilde{B}_M}^{\widetilde{M}}(\delta_{B_M}^{1/2}\chi')$, where $\chi'$ is an irreducible genuine representation of $\widetilde{T}_n$ corresponding to $\chi$, has a unique irreducible quotient called an exceptional representation of $\widetilde{M}$. This follows from the Langlands Quotient Theorem because $\chi$ belongs to the positive Weyl chamber in $W_M$. Note that the Langlands Quotient Theorem was proved for
covering groups over $p$-adic fields by Ban and Jantzen \cite{BJ}, and over Archimedean fields the proof of Borel and Wallach \cite{BW} is applicable to covering groups. The exceptional representation of $\widetilde{M}$ is isomorphic to the image of the intertwining operator with respect to $w_M$ (over Archimedean fields this follows from \cite{BW}, over $p$-adic fields one can argue as in \cite[Theorem~I.2.9]{KP}).

Globally, we construct exceptional representations as residual representations of Eisenstein series. The discussion in \cite[\S~II]{KP} (see also \cite[\S~3]{BFG}) can be extended to Levi subgroups. For
$\underline{s}=(s_1,\ldots,s_n)\in\C^n$, put
\begin{align*}
\chi_{\underline{s}}(\diag(t_1,\ldots,t_n))=\prod_{i=1}^{n}|t_i|^{s_i}.
\end{align*}
We extend $\chi_{\underline{s}}$ to a right $K$-invariant function on $M(\Adele)$ using the Iwasawa decomposition, and lift it to a function on the cover. Let $f$ be an element of $\Ind_{\widetilde{B}_M(\Adele)}^{\widetilde{M}(\Adele)}(\delta_{B_M}^{1/2}\chi')$, where $\chi'$ corresponds to a global $M$-exceptional character $\chi$ (defined similarly to the local one). Denote by $f_{\underline{s}}$ the
standard section given by $f_{\underline{s}}(m)=\chi_{\underline{s}}(m)f(m)$, $m\in\widetilde{M}(\Adele)$.
Form the Eisenstein series
\begin{align*}
E_{B_M}(m;f,\underline{s})=\sum_{m_0\in \lmodulo{B_M(F)}{M(F)}}f_{\underline{s}}(m_0m),\qquad m\in\widetilde{M}(\Adele).
\end{align*}
Identify $\Sigma_{G_n}$ with the pairs $(i,j)$, $1\leq i\ne j\leq n$, and $\Delta_{G_n}$ with the pairs $(i,i+1)$, $1\leq i<n$. For $\alpha=(i,j)\in\Sigma_{G_n}$, put $\underline{s}_{\alpha}=s_i-s_{j}$.
The global exceptional representation is spanned by the functions
\begin{align*}
m\mapsto\mathrm{Res}_{\underline{s}=\underline{0}}E_{B_{M}}(m;f,\underline{s})=
\lim_{\underline{s}\rightarrow\underline{0}}\prod_{\alpha\in\Delta}\underline{s}_{\alpha}E_{B_{M}}(m;f,\underline{s}).
\end{align*}
It is isomorphic to the restricted direct product of local exceptional representations.
The proof is analogous to \cite[Proposition~II.1.2 and
Theorems~II.1.4 and II.2.1]{KP} and \cite[Proposition~3.1 and Theorem~3.2]{BFG}.

Let $\Theta_M$ (resp., $\theta_{M,v}$) be the global (resp., local over $F_v$) exceptional representation corresponding to
a genuine character, which is a lift of $\delta_{B_M}^{1/4}|_{T_n^2}$ to $C_{\widetilde{T}_n}$. Then
$\Theta_M=\otimes'_v\Theta_{M,v}$. In particular
$\Theta_{G_n}=\Theta$.
\begin{remark}
Over a $p$-adic field, the aforementioned result \cite[Theorem~5.1]{Kable} also implies that
an exceptional representation of $\widetilde{M}$ is isomorphic to a metaplectic tensor of exceptional representations of the preimages of the direct factors of $M$.
\end{remark}
We are ready to state the result regarding the constant term of elements of $\Theta$.
\begin{theorem}\label{theorem:constant term theorem}
Let $\varphi$ belong to the space of $\Theta$.
The function $m\mapsto\varphi^U(m)$ on $\widetilde{M}(\Adele)$ belongs to the space of $\delta_Q^{1/4}\Theta_M$.
\end{theorem}
\begin{remark}
This was proved in \cite{KP} for $Q=B_n$.
\end{remark}
\begin{remark}
If $F$ is a function field, this result follows immediately from the computation of the Jacquet modules of $\theta$ over $p$-adic fields in \cite[Theorem~5.1]{Kable}.
\end{remark}
\begin{remark}\label{remark:the constant term does not vanish}
A priori, the constant term does not vanish identically, because it does not vanish for $N_n$ (this follows from the construction of the exceptional representation in \cite[\S~II]{KP}).
\end{remark}
\begin{remark}
Similar results were obtained in \cite{me7} for the small representation of a cover of $\mathrm{SO}_{2n+1}$ (in the sense of \cite{BFG}), and in \cite{me8} for the analog of this representation for $\mathrm{GSpin}_{2n+1}$.
In the case of $\mathrm{SO}_{2n+1}$, the direct factors of maximal parabolic subgroups do commute in the cover, so that a more precise result could be obtained. For $\mathrm{GSpin}_{2n+1}$ the direct factors commute only for the maximal parabolic subgroup obtained by removing the unique short simple root.
\end{remark}
\begin{proof}
As explained above and keeping the same notation (with $M=G_n$), we can assume
\begin{align*}
\varphi(g)=\mathrm{Res}_{\underline{s}=\underline{0}}E_{B_{n}}(g;f,\underline{s}),
\end{align*}
for some $f$ in the space of $\Ind_{\widetilde{B}_n(\Adele)}^{\widetilde{G}_n(\Adele)}(\delta_{B_n}^{1/2}\chi')$, where
$\chi'$ corresponds to a genuine lift of $\delta_{B_n}^{1/4}|_{T_n^2}$.

Let $\mathcal{W}\subset W$ be with
$G_n=\coprod_{w\in\mathcal{W}}B_nw^{-1}Q$. For $X<Q$ and $w\in\mathcal{W}$, set $X^w=\rconj{w}B_n\cap X$. Define
\begin{align}\label{eq:constant term theorem intertwining operator w}
&f_{\underline{s},w}(m)=M(w,\chi_{\underline{s}})f_{\underline{s}}(m)=\int_{\lmodulo{U^w(F)}{U(\Adele)}}f_{\underline{s}}(w^{-1}um)\, du,\\
\label{eq:constant term theorem Eisenstein series w}
&E_{M^w}(f_{\underline{s},w}(m))=
\sum_{m_0\in\lmodulo{M^w(F)}{M(F)}}f_{\underline{s},w}(m_0m).
\end{align}
Here $m\in\widetilde{M}(\Adele)$ (although the definition is valid also on $\widetilde{G}_n(\Adele)$).
The intertwining operator $M(w,\chi_{\underline{s}})$ is in fact defined by the meromorphic continuation of this integral.
According to M{\oe}glin and Waldspurger \cite[\S~II.1.7]{MW2},
\begin{align}\label{eq:constant term MW sum}
\varphi^{U}(m)=\sum_{w\in\mathcal{W}}\mathrm{Res}_{\underline{s}=\underline{0}}E_{M^w}(f_{\underline{s},w}(m)).
\end{align}
We can take the representatives $\mathcal{W}$ such that for all $w\in\mathcal{W}$,
\begin{align}\label{containment:convenient representative condition}
N_M<\rconj{w}N_n.
\end{align}
Indeed this follows from \cite[2.11]{BZ2} (applied to $W^{T_n,M}$ in their notation, $\rconj{w^{-1}}(M\cap B_n)<B_n$).

\begin{claim}\label{claim:constant term theorem claim poles of intertwining}
$f_{\underline{s},w}$ is holomorphic at $\underline{s}\rightarrow\underline{0}$, except for simple poles
in the variables $\underline{s}_{\alpha}$ with $\alpha\in\Delta_{G_n}$ such that $w\alpha<0$.
\end{claim}
\begin{claim}\label{claim:constant term theorem claim poles of series}
$E_{M^w}(f_{\underline{s},w}(m))$ is holomorphic at $\underline{s}\rightarrow\underline{0}$, except for at most $|\Delta_{G_n}|$ simple poles. There are less than $|\Delta_{G_n}|$ poles unless $w=w_Mw_0$, in which case
$m\mapsto E_{M^w}(f_{\underline{s},w}(m))$ belongs to the space of $\Theta_M$.
\end{claim}
Granted the claims, the result quickly follows: by Claim~\ref{claim:constant term theorem claim poles of series}, any
summand with $w\ne w_Mw_0$ vanishes when we take the residue, and we are left with
$E_{M^{w_Mw_0}}(f_{\underline{s},w_Mw_0}(m))$. Since at most one summand is nonzero, as explained in Remark~\ref{remark:the constant term does not vanish} this summand is nonzero for some data.
\end{proof}
\begin{proof}[Proof of Claim~\ref{claim:constant term theorem claim poles of intertwining}]
Since $f$ is left-invariant under $N_n(\Adele)$ and $U^w<N_n$, we may
rewrite the integration in \eqref{eq:constant term theorem intertwining operator w} over
$\lmodulo{U^w(\Adele)}{U(\Adele)}$. Moreover \eqref{containment:convenient representative condition} implies
\begin{align*}
N_MU^w=N_M\rconj{w}B_n\cap N_MU=\rconj{w}B_n\cap N_n=N_n^w,
\end{align*}
hence $\lmodulo{U^w}{U}=\lmodulo{N_n^w}{N_n}$ so that we can write the integral in
\eqref{eq:constant term theorem intertwining operator w} over $\lmodulo{N_n^w(\Adele)}{N_n(\Adele)}$. Now we proceed as in
\cite[Claim~3.5]{me8}. We can assume that $f$ is a pure tensor. At any finite place, the local intertwining operator is holomorphic for $\underline{s}$ in a small neighborhood of $\underline{0}$, because then the local component of
$\chi'\chi_{\underline{s}}$ belongs to the positive Weyl chamber. Thus the poles are global, in the sense that they appear in the product of local intertwining operators applied to the unramified normalized components of
$f$. Now we use the Gindikin-Karpelevich formula \eqref{eq:Gindikin-Karpelevich formula} to identify the possible poles.

Let $\alpha\in\Sigma_{G_n}^+$ and assume $w\alpha<0$. Recall the function $c_{\alpha}(\cdots)$ (defined before \eqref{eq:Gindikin-Karpelevich formula}). The poles occur in the products
\begin{align*}
\prod_vc_{\alpha}((\chi'\chi_{\underline{s}})_v)=\frac{\zeta(C_{\alpha}+2\underline{s}_{\alpha})}{\zeta(C_{\alpha}+2\underline{s}_{\alpha}+1)},
\end{align*}
where $\zeta$ is the partial Dedekind zeta function of $F$ and $0<C_{\alpha}\in\mathbb{Z}$ is an integer depending
on $\chi'$. As a function of $\underline{s}_{\alpha}$, there is a pole when $\underline{s}_{\alpha}\rightarrow0$ if and only if $C_{\alpha}=1$, this occurs precisely when $\alpha\in\Delta_{G_n}$. We also deduce that the poles are simple.
\end{proof}
\begin{proof}[Proof of Claim~\ref{claim:constant term theorem claim poles of series}]
The poles of $E_{M^w}(f_{\underline{s},w}(m))$ are either poles of $f_{\underline{s},w}(m)$, which were
identified in Claim~\ref{claim:constant term theorem claim poles of intertwining}, or poles incurred by the summation
over $\lmodulo{M^w}{M}$. Because $\lmodulo{U^w}{U}=\lmodulo{N_n^w}{N_n}$, $f_{\underline{s},w}$ is an element in the space of \begin{align*}
\Ind_{\widetilde{B}_n(\Adele)}^{\widetilde{G}_n(\Adele)}(\delta_{B_n}^{1/2}\ \rconj{w}(\chi'\chi_{\underline{s}})).
\end{align*}
Then the restriction of $f_{\underline{s},w}$ to $\widetilde{M}(\Adele)$ belongs to
\begin{align*}
\delta_Q^{1/2}\Ind_{\widetilde{B}_M(\Adele)}^{\widetilde{M}(\Adele)}(
\delta_{B_M}^{1/2}\ \rconj{w}(\chi'\chi_{\underline{s}})).
\end{align*}
Since $N_M<\rconj{w}B_n\cap M$ (by \eqref{containment:convenient representative condition}),
$B_M=\rconj{w}B_n\cap M$ ($w$ cannot conjugate more positive roots from $N_n$ into $M$, because $N_M$ is a maximal unipotent subgroup in $M$). Hence \eqref{eq:constant term theorem Eisenstein series w} is a series over $\lmodulo{B_M}{M}$.
To locate its poles we consider the constant term along $N_M$. Arguing as in \cite[Proposition~II.1.2]{KP}, the constant term is a sum of intertwining operators $M(\omega,\chi_{\underline{s}})f_{\underline{s},w}$, where $\omega$ varies over $W_M$.

Fix $\omega$ and let $\Sigma_M$ (resp., $\Sigma_M^+$) be the set of roots (resp., positive roots) spanned by the roots in $\Delta$. Assumption \eqref{containment:convenient representative condition} implies that $\rconj{w}(\chi'\chi_{\underline{s}})$
belongs to the positive Weyl chamber of $W_M$ (when $\underline{s}\rightarrow\underline{0}$). Therefore, as in the proof of Claim~\ref{claim:constant term theorem claim poles of intertwining} any pole must be located in a quotient of $\zeta$ functions
\begin{align}\label{eq:global expression of poles using zeta function}
\prod_vc_{\alpha}(\rconj{w}(\chi'\chi_{\underline{s}})_v)=\frac{\zeta(C_{w,\alpha}+2\underline{s}_{w^{-1}\alpha})}{\zeta(C_{w,\alpha}+2\underline{s}_{w^{-1}\alpha}+1)}.
\end{align}
Here $\alpha\in\Sigma_M^+$ and $\omega\alpha<0$;
$0<C_{w,\alpha}\in\mathbb{Z}$ depends on $\rconj{w}\chi'$, specifically $(\rconj{w}\chi')_v(a_{\alpha})=q_v^{-C_{w,\alpha}}$.
The quotient has a pole if and only if $C_{w,\alpha}=1$. Since $C_{\alpha}=1$ if and only if $\alpha\in\Delta_{G_n}$,
\begin{align}\label{eq:set of poles series}
\setof{\alpha\in\Sigma_M^+}{C_{w,\alpha}=1}=
\setof{\alpha\in\Delta_{G_n}}{w\alpha\in\Sigma_M^+}.
\end{align}
Thus the number of poles of the series is bounded by the size of \eqref{eq:set of poles series}.
Since the conditions $w\alpha\in\Sigma_M^+$ and $w\alpha<0$ are disjoint, the number of poles of
$E_{M^w}(f_{\underline{s},w}(m))$ is at most $|\Delta_{G_n}|$. The poles of $f_{\underline{s},w}$
are simple, they occur in the variables $\underline{s}_{\beta}$ where $\beta>0$ and $w\beta<0$. A pole of
\eqref{eq:global expression of poles using zeta function} is also simple, it occurs in
$\underline{s}_{w^{-1}\alpha}$ with $\alpha>0$, and because $w(w^{-1}\alpha)=\alpha>0$,
this pole does not appear in the set of poles of $f_{\underline{s},w}$. Thus the poles of
$E_{M^w}(f_{\underline{s},w}(m))$ are simple.

Assuming $|\Delta_{G_n}|$ poles are obtained, we prove $w=w_Mw_0$. In this case we must have
\begin{align*}
\Delta_{G_n}=\setof{\alpha\in\Delta_{G_n}}{w\alpha<0}\amalg
\setof{\alpha\in\Delta_{G_n}}{w\alpha\in\Sigma_M^+},
\end{align*}
and then \eqref{containment:convenient representative condition} immediately implies $w=w_Mw_0$.
Since in this case $\rconj{w}\delta_{B_n}=\delta_{Q}^{-1}\delta_{B_M}$,
the restriction of $f_{\underline{s},w}$ to $\widetilde{M}(\Adele)$ belongs to
$\delta_Q^{1/4}\Ind_{\widetilde{B}_M(\Adele)}^{\widetilde{M}(\Adele)}(\delta_{B_M}^{1/2}\chi'')$, where
$\chi''$ corresponds to a lift of $\delta_{B_M}^{1/4}|_{T_n^2}$, and the definition of $\Theta_M$ implies that
\begin{align*}
m\mapsto\mathrm{Res}_{\underline{s}=\underline{0}}E_{M^{w}}(f_{\underline{s},w}(m))
\end{align*}
belongs to the space of $\delta_Q^{1/4}\Theta_M$.
\end{proof}
\begin{remark}
Takeda \cite[Proposition~2.45]{Tk} computed the constant term along the unipotent radical corresponding to the partition
$(2,\ldots,2)$, for the twisted exceptional representation (his proof does not apply to the non-twisted version, i.e., to our setting). He developed a generalized Weil representation for the preimage of the Levi part, defined a global block-compatible cocycle and was able to provide an explicit realization, whose properties we shall use below (see the proof of Claim~\ref{claim:theorem Shalika Fourier coefficient value of semi-Whittaker on t}). However, it is not expected that a similar realization will apply for a general Levi subgroup: the realization of $\theta_n$ as essentially a Weil representation is only valid for $n=2$.
\end{remark}

Now we describe the Fourier coefficient along the Shalika unipotent subgroup and character.
Assume $n=2k$. The stabilizer of $\psi_k$ in $M_k$ is now $G_k^{\triangle}$.
\begin{claim}
The section $\mathfrak{s}$ is well defined on $G_k^{\triangle}(\Adele)$ and satisfies a global relation similar to
\eqref{eq:formula for s on triangle}, with the Hilbert symbol replaced by the product of local Hilbert symbols at all places.
\end{claim}
\begin{proof}
If $g\in G_k^{\triangle}(\Adele)$, for almost all $v$, $g_v\in G_k^{\triangle}(\mathcal{O}_v)$ and then by
Claim~\ref{claim:h and kappa agree on H cap K}, $\mathfrak{s}_v(g_v)=\kappa_v(g_v)$. Hence $\mathfrak{s}$ is well-defined.
The global relation follows from the local one, since $\mathfrak{s}=\prod_v\mathfrak{s}_v$.
\end{proof}
Henceforth we identify $G_k^{\triangle}(\Adele)$ with its image in $\widetilde{G}_n(\Adele)$ under $\mathfrak{s}$.

Let $\pi$ be any automorphic representation of $\widetilde{G}_n(\Adele)$. We say that $l$ is a global metaplectic Shalika functional on $\pi$, if $l$ satisfies the global analog of \eqref{eq:metaplectic Shalika def}, where
$c\in G_k(\Adele)$, $u\in U_k(\Adele)$ and
$\gamma_{\psi'}$ is now the global Weil factor. As in the case of the Whittaker model, we are interested in concrete realizations of such functionals, e.g., given as integrals. A global integral for the Shalika model of $G_n$ was studied in \cite{JS4,FJ}, involving an integration over $G_k^{\triangle}(\Adele)\times U_k(\Adele)$. For exceptional representations, we will show that
$\varphi^{U_k,\psi_k}$ already satisfies the required properties, i.e., the reductive part of the integration is not needed.
\begin{remark}
Beineke and Bump \cite{BeD} also studied, in the non-metaplectic case, local and global Shalika functionals given solely by
unipotent integration, for a specific representation.
\end{remark}

To study the properties of $\varphi^{U_k,\psi_k}$, we first relate it to the Fourier coefficient of $\Theta$ along $N_n$ and its degenerate character $\psi^{\circ}$ given by
\begin{align*}
\psi^{\circ}(v)=\sum_{i=1}^k\psi((-1)^{k-i}v_{2i-1,2i})
\end{align*}
(called a semi-Whittaker coefficient in \cite{BG}).

Define $w\in G_n(F)$ as follows. For all $1\leq i\leq k$, it has $1$ on the $(2i-1,i)$-th coordinate and
$(-1)^{k-i}$ on the $(2i,k+i)$-th coordinate. Its other entries are zero.
Let $V<U_k$ be the subgroup obtained from $U_k$ by zeroing all upper diagonal entries except
those on rows $i+1,\ldots,k$ of column $k+i$, for all $1\leq i\leq k-1$.
\begin{example}When $k=2$,
\begin{align*}
w=\left(\begin{array}{cccc}1\\&&-1\\&1\\&&&1\end{array}\right),\qquad V=\left\{\left(\begin{array}{cccc}1\\&1&v_1\\&&1\\&&&1\end{array}\right)\right\}.
\end{align*}
\end{example}
We prove:
\begin{lemma}\label{lemma:relating Fourier coefficients - Shalika to the semi-Whittaker}
For any $\varphi$ in the space of $\Theta$,
\begin{align*}
\varphi^{U_k,\psi_k}(g)=\int_{V(\Adele)}\varphi^{N_n,\psi^{\circ}}(wvg)\, dv.
\end{align*}
\end{lemma}
\begin{proof}
For $0\leq l\leq k-1$, define $U^{(l)}<N_n$, a character $\psi^{(l)}$ of $\lmodulo{U^{(l)}(F)}{U^{(l)}(\Adele)}$, $w^{(l)}\in G_n(F)$, and $V^{(l)}<N_n$ as follows. Let
\begin{align*}
&U^{(l)}=\left\{\left(\begin{array}{ccc}z&x_1&x_2\\&I_{k-l}&u\\&&I_{k-l}\end{array}\right):z\in N_{2l}\right\},\\
&\psi^{(l)}(u)=\sum_{i=1}^{l}\psi((-1)^{k-i}z_{2i-1,2i})\psi(\mathrm{tr}(u)).
\end{align*}
Note that $\psi^{(l)}$ is the product of a degenerate character depending only the odd rows of $z$, and the Shalika character on the coordinates of $u$.

The matrix $w^{(l)}$: for $1\leq i\leq l$, has $1$ in the $(2i-1,i)$-coordinate and
$(-1)^{k-i}$ in the $(2i,k+i)$-coordinate, and for $1\leq i\leq k-l$, has $1$ in coordinates
$(2l+i,l+i)$ and $(k+l+i,k+l+i)$, and $0$ elsewhere.

The subgroup $V^{(l)}$: it is obtained from $N_n$ by zeroing all upper diagonal entries except
those in rows $i+1,\ldots,k$ of column $k+i$, for all $1\leq i\leq l$.

\begin{example}When $k=3$,
\begin{align*}
w^{(1)}&=\left(\begin{array}{cccccc}1&0&0&0&0&0\\&&&1\\&1\\&&1\\&&&&1\\&&&&&1\end{array}\right),&
w^{(2)}&=\left(\begin{array}{cccccc}1&0&0&0&0&0\\&&&1\\&1\\&&&&-1\\&&1\\&&&&&1\end{array}\right),\\
V^{(1)}&=\left\{\left(\begin{array}{cccccc}1&&&0\\&1&&v_1\\&&1&v_2\\&&&1\\&&&&1\\&&&&&1\end{array}\right)\right\},&
V^{(2)}&=\left\{\left(\begin{array}{cccccc}1&&&0&0\\&1&&v_1&0\\&&1&v_2&v_3\\&&&1\\&&&&1\\&&&&&1\end{array}\right)\right\}.
\end{align*}
\end{example}

In particular $U^{(0)}=U_k$, $\psi^{(0)}=\psi_k$ (the Shalika character of $U_k$), $w^{(0)}=I_n$, $V^{(0)}=\{I_n\}$; and
$U^{(k-1)}=N_n$, $\psi^{(k-1)}=\psi^{\circ}$, $w^{(k-1)}=w$, $V^{(k-1)}=V$.
\begin{claim}\label{claim:technical claim relating Shalika Fourier coefficient to degenerate Whittaker}
For any $0\leq l\leq k-2$,
\begin{align}\label{eq:identity exchange of roots}
&\int_{V^{(l)}(\Adele)}\int_{\lmodulo{U^{(l)}(F)}{U^{(l)}(\Adele)}}\varphi(uw^{(l)}v)(\psi^{-1})^{(l)}(u)\, du\, dv\\\notag&=
\int_{V^{(l+1)}(\Adele)}\int_{\lmodulo{U^{(l+1)}(F)}{U^{(l+1)}(\Adele)}}\varphi(uw^{(l+1)}v)(\psi^{-1})^{(l+1)}(u)\, du\, dv.
\end{align}
\end{claim}
The lemma follows from the repeated application of this claim.
\end{proof}
\begin{proof}[Proof of Claim~\ref{claim:technical claim relating Shalika Fourier coefficient to degenerate Whittaker}]
We apply the technique of ``exchange of roots", following Ginzburg \cite{G} (see also \cite{GRS3,Soudry5,RGS}).
Identify $U_{k-l}$ with the subgroup of $U^{(l)}$ defined by the coordinates of $u$. Denote an element
of $U_{k-l}$ by
\begin{align*}
u(u_1,u_2,u_3,u_4)=\left(\begin{array}{cccc}1&&u_1&u_2\\&I_{k-l-1}&u_3&u_4\\&&1\\&&&I_{k-l-1}\end{array}\right).
\end{align*}
By definition the restriction of $\psi^{(l)}$ to $U_{k-l}$ is given by $\psi(u_1)\psi(\mathrm{tr}(u_4))$. Set
\begin{align*}
&Y=\{u(0,0,u_3,0)\}<U_{k-l},\quad C=\{u(u_1,u_2,0,u_4)\}<U_{k-l}, 
\\&X=\left\{\left(\begin{array}{ccc}1&e\\&I_{k-l-1}\\&&I_{k-l}\end{array}\right)\right\}.
\end{align*}
Let $U'<U^{(l)}$ be defined by the coordinates of $z$, $x_1$ and $x_2$, and define the following function on
$\widetilde{G}_n(\Adele)$,
\begin{align*}
\varphi'(g)=\int_{\lmodulo{U'(F)}{U'(\Adele)}}\varphi(u'g)(\psi^{-1})^{(l)}(u')\, du'.
\end{align*}
In particular, $\varphi'$ is trivial on $U_{k-l}(F)$. Since $U^{(l)}=U'\rtimes U_{k-l}$, the \lhs\ of \eqref{eq:identity exchange of roots} becomes
\begin{align}\label{eq:identity exchange of roots lhs}
\int_{V^{(l)}(\Adele)}\int_{\lmodulo{U_{k-l}(F)}{U_{k-l}(\Adele)}}\varphi'(uw^{(l)}v)(\psi^{-1})^{(l)}(u)\, du\, dv.
\end{align}

According to Ginzburg, Rallis and Soudry \cite{RGS} (Lemma~7.1, with $X,Y$ and $C$ as above),
\begin{align}\label{eq:identity exchange of roots 1}
&\int_{\lmodulo{U_{k-l}(F)}{U_{k-l}(\Adele)}}\varphi'(u)(\psi^{-1})^{(l)}(u)\, du=
\int_{Y(\Adele)}\int_{\lmodulo{CX(F)}{CX(\Adele)}}\varphi'(uy)(\psi^{-1})^{(l)}(u)\, du\, dy.
\end{align}
Here $\psi^{(l)}$ is extended trivially on $X$.
Put
\begin{align}\label{eq:def of w'}
w'=\left(\begin{array}{cccc}I_{2l+1}\\&&(-1)^{k-l-1}\\&I_{k-l-1}\\&&&I_{k-l-1}\end{array}\right).
\end{align}
Since $w'$ normalizes $U'$ and stabilizes $\psi^{(l)}|_{U'}$, we may replace $\varphi'$ with its left-translate by $w'$. The \rhs\ of \eqref{eq:identity exchange of roots 1} becomes
\begin{align}\label{eq:identity exchange of roots 2}
\int_{Y(\Adele)}\int_{\lmodulo{U''(F)}{U''(\Adele)}}\varphi(uw'y)\psi^{-1}(u)\, du\, dy,\end{align}
where
\begin{align*}
&U''=U'\rtimes \rconj{w'}(CX)=\left\{\left(\begin{array}{ccccc}z&x_1&x_2&x_3&x_4\\&1&u_1&e&u_2\\&&1\\&&&I_{k-l-1}&u_4\\&&&&I_{k-l-1}\end{array}\right):z\in N_{2l}\right\},\\
&\psi(u)=\sum_{i=1}^{l}\psi((-1)^{k-i}z_{2i-1,2i})\psi((-1)^{k-l-1}u_1)\psi(\mathrm{tr}(u_4)).
\end{align*}
Next we show that the $du$-integration can be extended to $U^{(l+1)}$.

Let $U'''<U''$ be obtained from $U''$ by
zeroing out the coordinates of $u_4$. Define
\begin{align*}
M'''=\left\{m(a,b)=\left(\begin{array}{ccc}I_{2l+1}\\&1&a\\&&b\end{array}\right):b\in G_{2(k-l-1)}\right\}
\end{align*}
and denote $A'''=m(a,I_{2(k-l-1)})$, $B'''=m(0,b)$.
The function
\begin{align*}
g\mapsto \int_{\lmodulo{U'''(F)}{U'''(\Adele)}}\varphi(ug)\psi^{-1}(u)\, du
\end{align*}
is well-defined on $\lmodulo{M'''(F)}{M'''(\Adele)}$. In particular, we can consider its Fourier expansion along
the abelian unipotent subgroup $A'''$. There are two orbits under the action of $B'''(F)$, the trivial one and the one corresponding to a character $\psi^{\star}$ depending only on the leftmost coordinate of $a$. However, for each $v<\infty$,
the Jaquet module $(\Theta_v)_{U'''A''',\psi\psi^{\star}}=0$. To see this, first note that by Theorem~\ref{theorem:intro Jacquet modules},
as a vector space $(\Theta_v)_{U'''A''',\psi\psi^{\star}}$ is a quotient of
\begin{align*}
\theta_{2l}\otimes((\theta_{2(k-l)})_{U_{1,1,2(k-l-1)},\psi\psi^{\star}}),
\end{align*}
where $U_{1,1,2(k-l-1)}$ is the unipotent radical of the standard parabolic subgroup
of $G_{2(k-l)}$ with a Levi part isomorphic to $G_1\times G_1\times G_{2(k-l-1)}$ (we restricted the Jacquet module to a representation of the unipotent subgroup to pass from the metaplectic tensor to the standard one, see \cite[Theorem~3.1]{Kable}). However, by \cite[Lemma~2.8]{me11} for any $m\geq3$,
$(\theta_m)_{U_1,\psi}$ is a quotient of $(\theta_m)_{U_2}$ (where on $U_1$, $\psi$ is given by $\psi(\left(\begin{smallmatrix}1&z\\&I_{m-1}\end{smallmatrix}\right))=\psi(z_1)$). Since $\psi^{\star}$ is nontrivial (and $2(k-l)\geq4$),
\begin{align*}
(\theta_{2(k-l)})_{U_{1,1,2(k-l-1)},\psi\psi^{\star}}=0.
\end{align*}
Hence only the trivial character remains and \eqref{eq:identity exchange of roots 2} becomes
\begin{align}\label{eq:identity exchange of roots 3}
\int_{Y(\Adele)}\int_{\lmodulo{U^{(l+1)}(F)}{U^{(l+1)}(\Adele)}}\varphi(uw'y)(\psi^{-1})^{(l+1)}(u)\, du\, dy.
\end{align}
Plugging this into the \lhs\ of \eqref{eq:identity exchange of roots lhs} and using the fact that
\begin{align*}
(w^{(l)})^{-1}Yw^{(l)}V^{(l)}=V^{(l+1)},\qquad w'w^{(l)}=w^{(l+1)}
\end{align*}
leads to the desired equality.
\end{proof}
\begin{theorem}\label{theorem:Shalika model given by Fourier coefficient}
The mapping $\varphi\mapsto\varphi^{U_k,\psi_k}$ on the space of $\Theta(=\Theta_{2k})$ is nonzero and satisfies
\begin{align*}
\varphi^{U_k,\psi_k}(c^{\triangle})=\gamma_{\psi,(-1)^{k}}(\det c)\varphi^{U_k,\psi_k}(1),\qquad\forall c\in G_k(\Adele).
\end{align*}
Consequently, $\Theta$ admits a global metaplectic Shalika model given by a Fourier coefficient. Moreover, the functional $\varphi\mapsto\varphi^{U_k,\psi_k}(1)$ is factorizable.
\end{theorem}
\begin{proof}
We begin with the proof of the formula, the non triviality will be deduced during the computation.
First we show that for $c\in\SL_k(\Adele)$,
\begin{align}\label{eq:theorem Shalika Fourier coefficient SL_k acts trivially}
\varphi^{U_k,\psi_k}(c^{\triangle})=\varphi^{U_k,\psi_k}(1).
\end{align}
 To show this,
because $\varphi$ is invariant on the left by $G_n(F)$, it is enough to consider
\begin{align*}
c=\left(\begin{array}{ccc}1&0&x\\&I_{k-2}&0\\&&1\end{array}\right),\qquad x\in\Adele.
\end{align*}
As in the proof of Proposition~\ref{proposition:Fourier coefficients are trivial along v and y} we
consider the Fourier expansion along the abelian subgroup of matrices $c^{\triangle}$ with $c$ of this form. The nontrivial coefficients
vanish by Theorem~\ref{theorem:intro Jacquet modules}. This proves \eqref{eq:theorem Shalika Fourier coefficient SL_k acts trivially}.

It remains to consider $c=t=\diag(t_1,\ldots,t_k)\in T_k$. To this end we claim the following.
\begin{claim}\label{claim:theorem Shalika Fourier coefficient conjugation of t by w}
\begin{align*}
\rconj{w}t=\prod_{l=1}^{k-1}(t_l,\prod_{i=l+1}^kt_i)_2(\prod_{i=l+1}^kt_i,\prod_{i=l+1}^kt_i)_2
\diag(t_1,t_1,\ldots,t_k,t_k).
\end{align*}
\end{claim}
\begin{claim}\label{claim:theorem Shalika Fourier coefficient value of semi-Whittaker on t}
\begin{align*}
\varphi^{N_n,\psi^{\circ}}(\diag(t_1,t_1,\ldots,t_k,t_k))=\prod_{l=1}^k\gamma_{\psi,(-1)^{k+1-i}}(t_i)\delta_{B_k}(\diag(t_1,\ldots,t_k))\varphi(1).
\end{align*}
\end{claim}
As a corollary, the Fourier coefficient $\varphi^{U_k,\psi_k}$ does not vanish identically on the space of $\Theta$.
Indeed, the arguments in Lemma~\ref{lemma:relating Fourier coefficients - Shalika to the semi-Whittaker} can be repeated in ``the opposite direction" to relate
$\varphi^{N_n,\psi^{\circ}}$ to an integration of $\varphi^{U_k,\psi_k}$. The former is nonzero by Claim~\ref{claim:theorem Shalika Fourier coefficient value of semi-Whittaker on t}, hence $\varphi^{U_k,\psi_k}$ cannot be zero for all $\varphi$.

According to \eqref{eq:Weil factor identities}, for any $1\leq l\leq k-2$,
\begin{align*}
&(t_l,\prod_{i=l+1}^kt_i)_2(\prod_{i=l+1}^kt_i,\prod_{i=l+1}^kt_i)_2
\gamma_{\psi,(-1)^{k+1-l}}(t_l)\gamma_{\psi,(-1)^{k+1-l+1}}(\prod_{i=l+1}^kt_i)
\\&=(t_l,\prod_{i=l+1}^kt_i)_2
\gamma_{\psi,(-1)^{k+1-l}}(t_l)\gamma_{\psi,(-1)^{k+1-l}}(\prod_{i=l+1}^kt_i)=\gamma_{\psi,(-1)^{k+1-l}}(\prod_{i=l}^kt_i).
\end{align*}
The conjugation of $V$ by $t$ multiplies the measure by $\delta_{B_k}^{-1}(\diag(t_1,\ldots,t_k))$. Combining the claims yields the result.

The assertion regarding factorizability follows from the local uniqueness of the metaplectic Shalika model
at all the local components of $\Theta$ (see \S~\ref{subsection:The metaplectic Shalika model of theta}, we only need to know that the dimension of the space of metaplectic Shalika functionals on $\Theta_v$ is at most one, for all $v$).
\end{proof}
\begin{proof}[Proof of Claim~\ref{claim:theorem Shalika Fourier coefficient conjugation of t by w}]
This follows from the local properties of the cocycle of \cite{BLS}. Specifically, write
$t$ in the form $\mathrm{diag}(t_1,t',t_1,t')\in T_n$ with $t'\in T_{k-1}$ and let $w'$ be given by
 \eqref{eq:def of w'} with $l=0$. Using \cite{BLS} (\S~2 Lemma~2, \S~3
Lemmas~3 and 1) we see that
\begin{align*}
\rconj{w'}\mathfrak{s}(t)&=\sigma(w',t)\mathfrak{s}(\rconj{w}t)=(t_1,\det{t'})_2(\det{t'},\det{t'})_2\mathfrak{s}(\mathrm{diag}(t_1,t_1,t',t'))\\&=
(t_1,\det{t'})_2(\det{t'},\det{t'})_2\mathfrak{s}(\mathrm{diag}(t_1,t_1,I_{2n-2}))
\mathfrak{s}(\mathrm{diag}(I_2,t',t')).
\end{align*}
For $l=1$, $w'$ commutes with $\mathfrak{s}(\mathrm{diag}(t_1,t_1,I_{2n-2}))$, and by \eqref{eq:block-compatibility} we can apply the same arguments to compute $\rconj{w'}\mathfrak{s}(\mathrm{diag}(I_2,t',t'))$.
Since $w$ is a product if elements $w'$ (with $l$ varying), we can repeat this for all $l=1,\ldots,k-1$ and obtain the result.
\end{proof}
\begin{proof}[Proof of Claim~\ref{claim:theorem Shalika Fourier coefficient value of semi-Whittaker on t}]
Let $Q=M\ltimes U<G_n$ be the standard parabolic subgroup,
whose Levi part $M$ is isomorphic to $k$ copies of $G_2$.
According to Theorem~\ref{theorem:constant term theorem}, the mapping $m\mapsto\varphi^U(m)$ belongs to the space of $\delta_Q^{1/4}\Theta_M$. The representation $\Theta_M$ is isomorphic to the global Weil representation $\Pi\otimes\widetilde{\varphi}^{-1}$
constructed by Takeda \cite[\S~2.3 and p.~204]{Tk} (this is $\Pi_{\chi}\otimes\widetilde{\varphi}_P^{-1}$ in his notation with the unitary character $\chi=1$, see \cite[2.26]{Tk}). The mapping $\widetilde{\varphi}$ was defined in \cite[\S~A.1, p.~261]{Tk} to correct a global block-compatibility issue. If $m\in M(\Adele)$,
$\widetilde{\varphi}(m)=m\widehat{S}(m)$, where $\widehat{S}=\prod_v(\widehat{S})_v(m_v)$ and at each place $v$,
\begin{align*}
(\widehat{S})_v(m)=\frac{\prod_{i=1}^k\mathfrak{s}_v(m_i)}{\mathfrak{s}_v(m)},\qquad m=\diag(m_1,\ldots,m_k)\in M(F_v).
\end{align*}
Therefore $\varphi^{N_n,\psi^{\circ}}$ is the application
of a Fourier coefficient corresponding to $(N_M,\psi^{\circ}|_{N_M})$, to an automorphic form
in the space of $\delta_Q^{1/4}(\Pi\otimes\widetilde{\varphi}^{-1})$.

For $d=\diag(t_1,t_1,\ldots,t_k,t_k)$, $\widehat{S}(d)=1$ because for all $v$,
$\prod_{i=1}^k\mathfrak{s}_v(t_i^{\triangle})=\mathfrak{s}_v(t)$.
Hence $\widetilde{\varphi}(d)=d$. Also $\delta_Q^{1/4}(d)=\delta_{B_k}(\diag(t_1,\ldots,t_k))$.
Lastly, following the definitions and formulas from \cite[Proposition~2.55 and (2.56)]{Tk} we see that
for any automorphic form $\phi$ in the space of $\Pi$,
$\phi^{N_M,\psi^{\circ}}(d)=\prod_{l=1}^k\gamma_{\psi,(-1)^{k+1-i}}(t_i)\phi(1)$.
\end{proof}

\begin{corollary}\label{corollary:metaplectic Shalika functional in Archimedean places}
Let $\theta(=\theta_{2k})$ be an exceptional representation over an Archimedean field (in fact, over any local field). Then
$\theta$ admits a metaplectic $(\psi_{(-1)^{k}},\psi)$-Shalika functional.
\end{corollary}
\begin{proof}
The functional $\varphi\mapsto\varphi^{U_k,\psi_k}(1)$ defines a metaplectic $(\psi_{(-1)^{k}},\psi)$-Shalika functional
on each of the local pieces of $\Theta$. Now the result follows from
Theorem~\ref{theorem:Shalika model given by Fourier coefficient} and a standard globalization argument applied to
the exceptional representation.
\end{proof}
\section{A Godement--Jacquet type integral}\label{section:A Godement Jacquet integral}
\subsection{Local theory}\label{subsection:Local p-adic Godement--Jacquet integral}
Let $F$ be a local field and $\pi\in\Alg{G_k}$ be irreducible. Set $n=2k$ and $\theta=\theta_{n}$ ($\theta_n$ was defined in \S~\ref{subsection:The exceptional representations}).
Regard $G_k$ as a subgroup of $G_n$ via the embedding $g\mapsto\diag(g,I_k)$.

We define the following local Godement--Jacquet type integral.
Let $f$ be a matrix coefficient of $\pi$, $\mathscr{S}\in\mathscr{S}(\theta,\psi_{(-1)^{k}},\psi)$,
$\mathscr{S}'\in\mathscr{S}(\theta,\psi_{(-1)^{k+1}},\psi^{-1})$ 
and $s\in\C$. Define
\begin{align}\label{eq:local GJ integral p-adic}
Z(f,\mathscr{S},\mathscr{S}',s)=\int_{G_k}f(g)\mathscr{S}(\mathfrak{s}(g))\mathscr{S}'(\mathfrak{s}(g))|\det{g}|^{s-k/2}\, dg.
\end{align}
Note that the actual choice of section $\mathfrak{s}$ here does not matter, because
$\mathscr{S}$ and $\mathscr{S}'$ are both genuine. Henceforth we omit it from the notation.

\begin{theorem}\label{theorem:local props GJ integral p-adic}Assume $F$ is a $p$-adic field.
\begin{enumerate}[leftmargin=*]
\item\label{item:p-adic integral conv}
The integral \eqref{eq:local GJ integral p-adic} is absolutely convergent for $\Re(s)\gg0$, the domain of convergence depends only on the representations.
\item\label{item:p-adic integral made const}One can choose data $(f,\mathscr{S},\mathscr{S}')$ such that $Z(f,\mathscr{S},\mathscr{S}',s)$ is absolutely convergent and equals
$1$, for all $s$.
\item \label{item:p-adic integral meromorphic}The integral has a meromorphic continuation to a rational function in $\C(q^{-s})$.
\item \label{item:p-adic integral unramified}
When all data are unramified,
\begin{align*}
Z(f,\mathscr{S},\mathscr{S}',s)=\frac{L(2s,\mathrm{Sym}^2,\pi)}{L(2s+1,\wedge^2,\pi)}.
\end{align*}
\item \label{item:p-adic integrals holomorphic at s=1/2 for some cases}
Assume that $\pi$ satisfies the following condition: if $\varepsilon$ is a normalized exponent of $\pi$ along a standard parabolic subgroup $Q$ (see \S~\ref{subsection:representations} for the definition), then
$\delta_Q^{1/2}|\varepsilon|$ lies in the open cone spanned by the positive roots in the Levi part of $Q$. In particular, this condition holds when $\pi$ is tempered. Then $Z(f,\mathscr{S},\mathscr{S}',s)$ is holomorphic at $\Re(s)=1/2$.
\end{enumerate}
\end{theorem}
\begin{proof}
The first three assertions have already been proved in \cite{me11}, in a slightly different form. By virtue of Claim~\ref{claim:general bound on Shalika function} the integral is bounded by the zeta integral of Godement and Jacquet,
\begin{align*}
\int_{G}f(g)\phi(g)|\det|^{r}(g)dg,
\end{align*}
where $\phi\in \mathcal{S}(F_{k\times k})$ is positive and $r\in\R$ depends only on $\theta$ (i.e., not on
$\mathscr{S}$ and $\mathscr{S}'$). This integral is absolutely convergent for $r>r_0$, where $r_0$ is independent of the specific coefficient $f$ (\cite[p.~30]{GJ}).

Applying Claim~\ref{claim:richness of restriction of Shalika model to Q} to $\theta$, one can choose
elements $\mathscr{S}$ and $\mathscr{S}'$ such that the mapping
$g\mapsto \mathscr{S}(g)\mathscr{S}'(g)$ is an arbitrary function in
$\ind_{G_k^{\triangle}U_k}^{Q_n}(1)\isomorphic C^{\infty}_c(G_k)$. This implies \eqref{item:p-adic integral made const}, when we take some $f$ which does not vanish at the identity.

We turn to the meromorphic continuation.
For $a,b\in G_k$, put $m=\diag(a,b)\in M_k$. Then for any $g\in G_k$,
\begin{align*}
\mathscr{S}(gm)\mathscr{S}'(gm)=&(\mathscr{S}\mathscr{S}')(\diag(g,I_k)m)\\&
=\gamma_{\psi}(\det b)\gamma_{\psi}^{-1}(\det b)
(\mathscr{S}\mathscr{S}')(\diag(b^{-1}ga,I_k))\\
&=(\mathscr{S}\mathscr{S}')(b^{-1}ga).
\end{align*}
Additionally for $u\in U_k$,
\begin{align*}
\mathscr{S}(gu)\mathscr{S}'(gu)
=\psi(gu)\psi^{-1}(gu)(\mathscr{S}\mathscr{S}')(g)=
(\mathscr{S}\mathscr{S}')(g).
\end{align*}
Hence in its domain of absolute convergence \eqref{eq:local GJ integral p-adic} satisfies
\begin{align*}
Z(mf,mu\mathscr{S},mu\mathscr{S}',s)=
\delta_{Q_k}^{1/2-s/k}(m)Z(f,\mathscr{S},\mathscr{S}',s).
\end{align*}
Therefore it is a nonzero element of
\begin{align}\label{homspace:for s}
\Hom_{M_k}((\theta\otimes\theta')_{U_k},\delta_{Q_k}^{1/2}|\det|^{-s}\pi^{\vee}\otimes|\det|^{s}\pi).
\end{align}
In \cite[Claim~4.6]{me11} we proved that this space is at most one-dimensional, outside of a finite set of
values of $q^{-s}$. We mention that the proof utilized the structure of the
twisted Jacquet modules of $\theta$ along $U_k$.

Now the one-dimensionality result combined with the first two assertions imply \eqref{item:p-adic integral meromorphic},
by virtue of Bernstein's continuation principle (in \cite{Banks}). 

Assume all data are unramified. We prove \eqref{item:p-adic integral unramified}.
The function $f$ is bi-$G_k(\mathcal{O})$-invariant and by Claim~\ref{claim:Shalika or H functional are determined on the torus}
\eqref{claim Shalika or H determined:K invariance left and right}, for any $k_1,k_2\in G_k(\mathcal{O})$,
\begin{align*}
(\mathscr{S}\mathscr{S}')(k_1gk_2)=(\mathscr{S}\mathscr{S}')(g).
\end{align*}
Note that on its own, $\mathscr{S}$ is not bi-$\kappa(G_k(\mathcal{O}))$-invariant because elements of
$\kappa(\diag(I_k,G_k(\mathcal{O})))$ and $\mathfrak{s}(\diag(G_k,I_k))$ do not commute in $\widetilde{G}_n$. Nonetheless, the integrand is bi-$G_k(\mathcal{O})$-invariant.
Hence we can write
\eqref{eq:local GJ integral p-adic} using the Cartan decomposition
$\coprod_{\lambda\in\Z^k_{\geq}}G_k(\mathcal{O})\varpi^{\lambda}G_k(\mathcal{O})$ and
Lemma~\ref{lemma:unramified normalized Shalika function of theta} implies that the integrand vanishes unless
$\lambda\in2\Z^k_+$ (see \S~\ref{section:The metaplectic unramified Shalika formula} for the notation).

Let $t_{\pi}
\in G_k(\C)$ be the Satake parameters of $\pi$.
For $\lambda\in\Z^k_+$, the formulas of Macdonald \cite[Chapter V \S~2 (2.9) and \S~3 (3.4)]{M} ($f(\varpi^{\lambda})$ is $\omega_s(\pi^{-\lambda})$ in his notation) show
\begin{align*}
f(\varpi^{\lambda})=\mathrm{vol}(G_k(\mathcal{O})\varpi^{\lambda}G_k(\mathcal{O}))^{-1}\delta^{-1/2}_{B_k}(\varpi^{\lambda})P_{\lambda}(t_{\pi};q^{-1}).
\end{align*}
Here $P_{\lambda}$ is the Hall-Littlewood polynomial in the Satake parameters,
see \cite[Chapter III \S~2]{M} for the definition.
Using Lemma~\ref{lemma:unramified normalized Shalika function of theta} and
$\delta^{1/2}_{B_n}(t_{\lambda})|\det{\varpi^{\lambda}}|^{-k/2}=\delta^{1/2}_{B_k}(\varpi^{\lambda})$, the integral becomes
\begin{align*}
\sum_{\lambda\in2\Z^k_+}P_{\lambda}(t_{\pi};q^{-1})q^{-|\lambda|s},\qquad|\lambda|=\lambda_1+\ldots+\lambda_k.
\end{align*}
Since $P_{\lambda}(t_{\pi};q^{-1})q^{-|\lambda|s}=P_{\lambda}(t_{\pi|\det|^s};q^{-1})$,
the required formula follows immediately from the formal identity of \cite[Chapter III \S~5, example 2]{M}:
\begin{align*}
\sum_{\lambda\in2\Z^k_+}P_{\lambda}(x_1,\ldots,x_k;r)=\frac{\prod_{i<j}(1-rx_ix_j)}{\prod_{i\leq j}(1-x_ix_j)}.
\end{align*}

Finally consider \eqref{item:p-adic integrals holomorphic at s=1/2 for some cases}. Assume $\pi$ satisfies the assumption on the exponents. We prove that the integral is absolutely convergent, uniformly on compact sets, if $r=\Re(s)\geq 1/2-\epsilon_0$, for some $\epsilon_0>0$ (depending on the exponents of $\pi$). Let $\Delta\subset\Delta_{G_k}$ and $P=P_{\Delta}$ be the corresponding standard parabolic subgroup, $P=M\ltimes U$. Let $\zeta>0$. Let $C_M^+$ be the set of $\varpi^{\lambda}\in C_M$ such that $|\alpha(t)|<\zeta$ for all $\alpha\in\Delta_{G_k}\setdifference\Delta$.
Let $\mathcal{K}$ be a principal congruence subgroup on which the integrand is bi-$\mathcal{K}$. For the proof of convergence, we may replace $K$ in the Cartan decomposition with $\mathcal{K}$, and we see that it is enough to prove the convergence of the following sum, for each $\Delta$,
\begin{align*}
\sum_{\varpi^{\lambda}\in C_{M}^+}|f\mathscr{S}\mathscr{S}'\delta_{B_k}^{-1}|
(\varpi^{\lambda})q^{-|\lambda|(r-k/2)}.
\end{align*}
Here the modulus character $\delta_{B_k}^{-1}$ is bounding $\mathrm{vol}(\mathcal{K}\varpi^{\lambda}\mathcal{K})$.

This sum can be bounded using the exponents of the representations. Let $P^{\wedge}$ be the standard parabolic subgroup of $G_n$, whose Levi part is isomorphic to $M\times G_k$.
According to \cite[Lemma~6.2]{JR2}, if $\zeta$ is sufficiently small with respect to $\varphi$, $\mathscr{S}$ is bounded on $C_M^+$ by the exponents of $\theta$ along $P^{\wedge}$. Kable proved that the normalized Jacquet module of $\theta$ along the unipotent radical of $P^{\wedge}$ is an irreducible representation with a central character $\omega$ such that $|\omega|=\delta_{P^{\wedge}}^{-1/4}$ \cite[Theorem~5.1]{Kable}. The normalized exponent of $\theta$ along $P^{\wedge}$ is thus $\delta_{P^{\wedge}}^{-1/4}$. Since
on $C_M$, $|\det|^{-k/4}\delta_{B_k}^{-1/4}=\delta_{P^{\wedge}}^{-1/4}$ (on the \rhs\ $C_M$ is embedded in $M$), we obtain (up to a constant)
\begin{align*}
\sum_{\varpi^{\lambda}\in C_{M}^+}|\delta_{B_k}^{-1/2}f|(\varpi^{\lambda})q^{-|\lambda|r}.
\end{align*}
This is bounded in the proper right half-plane, given our assumption on $\pi$. 
\end{proof}
\begin{remark}\label{remark:local gamma factor}
As we have seen in the proof, the local zeta integral can be regarded as an element of \eqref{homspace:for s},
which according to \cite[Claim~4.6]{me11} is at most one-dimensional, outside of a finite set of
values of $q^{-s}$. This is the uniqueness property underlying the standard definition of the local
gamma factor of Rankin--Selberg integrals. One typically considers a ``similar zeta integral", absolutely convergent in a left half-plane, which also belongs to \eqref{homspace:for s}, then the proportionality factor between these integrals is the gamma factor. The problem is to find the similar integral. It should be related to a global construction, because the product of local gamma factors, over all places (including the Archimedean ones) should be $1$, and this usually follows from a global functional equation.
\end{remark}

\begin{theorem}\label{theorem:local props GJ integral archimedean}
Assume $F$ is an Archimedean field.
\begin{enumerate}[leftmargin=*]
\item\label{item:archimedean integral conv}
The integral \eqref{eq:local GJ integral p-adic} is absolutely convergent for $\Re(s)\gg0$, the domain of convergence depends only on the representations.
\item \label{item:archimedean integral made non-zero} For each $r>0$ one can choose data $(f,\mathscr{S},\mathscr{S}')$ such that $Z(f,\mathscr{S},\mathscr{S}',s)$ is absolutely convergent and nonzero for $s$ in the strip $\{s\in\C:|\Im(s)|\leq r\}$.
\item \label{item:archimedean integral meromorphic}
For fixed $(f,\mathscr{S},\mathscr{S}')$ the function $Z(f,\mathscr{S},\mathscr{S}',s)$ has a meromorphic continuation in $s\in\C$.
\end{enumerate}
\end{theorem}

\begin{proof}
As in the $p$-adic case, by Claim~\ref{claim:general bound on Shalika function} the convergence is reduced to the convergence of a Godement--Jacquet type integral, which was proved in \cite[Theorem 8.7]{GJ}.

Consider \eqref{item:archimedean integral made non-zero}. Let $f$ be a matrix coefficient with $f(I_k)=1$. We choose a compact neighborhood $\mathcal{V}$ of the identity in $G_k$, such that both $f$ and $|\det|^{s-k/2}$ take values in $\{z\in\C:\arg(z)<\pi/8\}$, on $\mathcal{V}$ and whenever $|\Im(s)|\leq r$. We will show that, possibly taking a smaller $\mathcal{V}$, one can choose $\mathscr{S}$ and $\mathscr{S}'$ such that $\mathscr{S}\mathscr{S}'(g)$ is supported in $\mathcal{V}$, is equal to $1$ at the identity, and takes values in $\{\arg(z)<\pi/4\}$. We see that the integrand is supported in $\mathcal{V}$, equals $1$ at the identity, and in its support takes values in $\{\Re(z)>0\}$, for $|\Im(s)|\leq r$. Hence, the corresponding integral over $G_k$ converges for any such $s$ and is nonzero.

To find $\mathscr{S}$ and $\mathscr{S}'$ with the above properties, we imitate the argument used in the proof of Claim~\ref{claim:richness of restriction of Shalika model to Q}. In fact, this argument shows that for any $\varphi$ in the representation space of $\theta$ and any $\phi\in\mathcal{S}(F_{k\times k})$, $\mathscr{S}_{\theta(\phi)\varphi}(d_g)=\widehat{\phi}(g)\mathscr{S}_\varphi(d_g)$, where $d_g=\mathfrak{s}(\diag(g,I_k))$. We choose $\varphi$ such that $\mathscr{S}_\varphi(\mathfrak{s}(I_n))=1$, then possibly making $\mathcal{V}$ smaller we can achieve that $\mathscr{S}_\varphi(d_g)$ takes values in $\{\arg(z)<\pi/8\}$ for $g\in\mathcal{V}$. We take $\phi$ such that $\widehat{\phi}\geq0$, $\phi(I_k)=1$, and has support in $\mathcal{V}$, then put
$\mathscr{S}=\mathscr{S}_{\theta(\phi)\varphi}$. The function $\mathscr{S}'$ is chosen similarly.

Now we prove \eqref{item:archimedean integral meromorphic}. As explained in the proof of Claim~\ref{claim:general bound on Shalika function}, we may replace $\mathscr{S}$ by a finite sum $\sum_j\mathscr{S}_j\phi_j$ with metaplectic Shalika functions $\scrS_j$ and $\phi_j\in\mathcal{S}(F_{k\times k})$. Hence it is enough to prove meromorphic continuation for an integral of the form
\begin{align*}
\int_{G_k} f(g)\scrS\scrS'(g)\phi(g)|\det(g)|^{s-k/2}\,dg ,\qquad \phi\in\mathcal{S}(F_{k\times k}).
\end{align*}
We use the integration formula
\begin{align*}
\int_{G_k}\vartheta(g)\,dg = \int_{K}\int_{K}\int_\R\int_{\R_{>0}^{k-1}} \vartheta(k_1a_xk_2)J(x)\, d(x_1,\ldots,x_{k-1})\, dx_k \, dk_1\, dk_2,
\end{align*}
where $K$ is the maximal compact subgroup of $G_k$ (the orthogonal group over $\R$, the unitary group if $F=\C$) 
and for $x=(x_1,\ldots,x_k)\in\R^k$,
\begin{align*}
&a_x=\diag(e^{-(x_1+\cdots+x_k)},e^{-(x_2+\cdots+x_k)},\ldots,e^{-x_k}),\\
&J(x)=\prod_{1\leq i<j\leq k}\sinh\left(x_i+\cdots+x_{j-1}\right),
\end{align*}
for $F=\R$, and $J(x)$ has to be replaced by $J(x)^2$ if $F=\C$.
Since $J(x)$ is a finite sum of terms of the form $e^{p\cdot x'}$ with $p\in\Z^{k-1}$ and $x'=(x_1,\ldots,x_{k-1})$, it suffices to show that every integral of the form
\begin{align}\label{int:archimedean integral for cont}
 \int_{K}\int_{K}\int_\R\int_{\R_{>0}^{k-1}}f\scrS\scrS'\phi(k_1a_xk_2)
 e^{-(s+s_1)x_1}e^{-(2s+s_2)x_2}\cdots e^{-(ks+s_k)x_k}d(\cdots)
\end{align}
with $s_1,\ldots,s_k\in\R$ (coming from the exponent $k/2$ in $|\det(g)|^{s-k/2}$ and from $e^{p\cdot x'}$) converges absolutely for $\Re(s)\gg0$ and extends to a meromorphic function in $s\in\C$.

Note that $f(g)=\xi^{\vee}(\pi(g)\xi)$ for some $\xi\in V$ and $\xi^{\vee}\in V^\vee$, where $V$ is the space of $\pi$, so that
\begin{align*}
f(k_1a_xk_2)=\pi^\vee(k_1^{-1})\xi^{\vee}(\pi(a_x)\pi(k_2)\xi).
\end{align*}
Also $k_1\mapsto\pi^{\vee}(k_1^{-1})\xi^{\vee}$ and $k_2\mapsto\pi(k_2)\xi$ are smooth.
Hence Theorem~\ref{thm:AsymptoticExpansionSmoothMatrixCoefficientsArchimedean} can be applied with $\pi^{\vee}(k_1^{-1})\xi^{\vee}$ instead of $\xi^{\vee}$ and $\pi(k_2)\xi$ instead of $\xi$. Furthermore
$\scrS$ (resp. $\scrS'$) is of the form $\scrS(g)=l(\theta(g)\varphi)$ (resp. $\scrS'(g)=l'(\theta(g)\varphi')$) with $\varphi$ (resp. $\varphi'$) in the space of $\theta$ and a metaplectic Shalika functional $l$ (resp. $l'$). Then \eqref{eq:metaplectic Shalika def} implies
\begin{align*}
\scrS\scrS'(k_1a_xk_2) = \scrS\scrS'(a_x\diag(k_2,k_1^{-1})),
\end{align*}
and  we can apply Theorem~\ref{thm:AsymptoticExpansionShalikaFunctionsArchimedean} with $\theta(\mathfrak{s}(\diag(k_2,k_1^{-1})))\varphi$ instead of $\varphi$ (similarly for $\varphi'$).

We split the $dx_k$-integration in \eqref{int:archimedean integral for cont} into $x_k\geq0$ and $x_k\leq0$.
First assume $x_k\geq0$. By the above considerations and Theorems~\ref{thm:AsymptoticExpansionSmoothMatrixCoefficientsArchimedean} and \ref{thm:AsymptoticExpansionShalikaFunctionsArchimedean}, $f\scrS\scrS'(k_1a_xk_2)$ can be written as a finite sum of terms of the form
$$ c(x_{\overline{I}};k_1,k_2)x_I^\alpha e^{-z\cdot x_I} $$
where $I\subset\{1,\ldots,k\}$, $\alpha\in\N^I$, $z\in\C^I$ with $\Re z_j\geq\Lambda_j=\Lambda_{\pi,j}+\Lambda_{\theta,j}+\Lambda_{\theta',j}$ for all $j\in I$ (see \eqref{eq:DefinitionLambdaPiSmoothMatrixCoefficient} for the definition of $\Lambda_{\pi,j}$), and $c(x_{\overline{I}};k_1,k_2)$ is smooth in $x_{\overline{I}}$, $k_1$ and $k_2$, and satisfies
$$ |c(x_{\overline{I}};k_1,k_2)| \leq Ce^{-D_{\overline{I}}\cdot x_{\overline{I}}},\qquad \forall x_{\overline{I}}\in\R_{>0}^{\overline{I}},k_1,k_2\in K. $$
Here the numbers $D_j\in\R$ can be chosen arbitrarily large. The resulting integral is
\begin{align*}
 \int_{K}\int_{K}\int_{\R_{>0}^k}c(x_{\overline{I}};k_1,k_2)x_I^\alpha e^{-z\cdot x_I}\phi(k_1a_xk_2)
 \Big(\prod_{j=1}^ke^{-(js+s_j)x_j}\Big)d(\cdots).
\end{align*}
Substitute $y_j=e^{-x_j}$ and obtain
\begin{align*}
 \int_{K}\int_{K}\int_0^1\cdots\int_0^1c'(y_{\overline{I}};k_1,k_2)\Big(\prod_{j\in I} y_j^{js+s_j+z_j}\log^{\alpha_j}y_j \Big)\Big(\prod_{j\notin I}y_j^{js+s_j}\Big)
\phi(k_1b_yk_2)\,d(\cdots),
\end{align*}
where $b_y=\diag(y_1\cdots y_k,y_2\cdots y_k,\ldots,y_k)$ and $c'(e^{-x_{\overline{I}}};k_1,k_2)=c(x_{\overline{I}};k_1,k_2)$ satisfies
$$ |c'(y_{\overline{I}};k_1,k_2)| \leq Cy_{\overline{I}}^{D_{\overline{I}}}, \qquad \forall y_{\overline{I}}\in(0,1)^{\overline{I}},k_1,k_2\in K. $$
The function $\phi(k_1b_yk_2)$ together with all its $y$-derivatives is bounded for all $k_1,k_2\in K$ and $y_1,\ldots,y_k\in(0,1)$. Hence, the integral
\begin{align*}
B_{c'}(y)=\int_{K}\int_{K} c'(y_{\overline{I}};k_1,k_2)\phi(k_1b_yk_2)\,dk_1\,dk_2
\end{align*}
converges absolutely and defines a smooth function in $y_1,\ldots,y_k\in(0,1)$ such that for all differential operators $\mathcal{D}$ in $y_I$ with constant coefficients we have
\begin{align*}
|\mathcal{D}B_{c'}(y)| \leq C_{\mathcal{D}}y_{\overline{I}}^{M_{\overline{I}}}, \qquad \forall y\in(0,1)^k.
\end{align*}
This implies that the integral
$$ \int_0^1\cdots\int_0^1 B_{c'}(y)\Big(\prod_{j\in I}y_j^{js+s_j+z_j}\log^{\alpha_j}y_j\Big)\Big(\prod_{j\notin I}y_j^{js+s_j}\Big)dy_1\cdots dy_k $$
converges absolutely whenever $\Re s>-\frac{\Re\Lambda_j+s_j+1}{j}\geq-\frac{\Re z_j+s_j+1}{j}$ for $j\in I$ and $\Re s>-\frac{D_j+s_j+1}{j}$ for $j\notin I$. By the usual technique writing $(s+1)y^s=\frac{d}{dy}y^{s+1}$ and integrating by parts it is easy to see that the integral further extends to a meromorphic function in the region where $\Re s>-\frac{D_j+s_j+1}{j}$ for $j\notin I$ (the region where the integrals over $y_{\overline{I}}$ always converge absolutely). Since the $D_j$ can be chosen arbitrarily large, this shows the meromorphic extension to $s\in\C$.

Now assume $x_k\leq0$. Let $x'=(x_1,\ldots,x_{k-1},0)$ with $x_1,\ldots,x_{k-1}\geq0$ and $x_k\leq0$. By Theorems~\ref{thm:AsymptoticExpansionSmoothMatrixCoefficientsArchimedean} and \ref{thm:AsymptoticExpansionShalikaFunctionsArchimedean} and Remark~\ref{rem:ExpansionsForXkleq0} the product $f\scrS\scrS'(k_1a_xk_2)$ can be written as a finite sum of terms of the form
$$ c(x'_{\overline{J}},x_k;k_1,k_2)(x'_J)^\alpha e^{-z'\cdot x'_J} $$
where $J\subset\{1,\ldots,k-1\}$, $\overline{J}=\{1,\ldots,k-1\}\setdifference J$, $\alpha\in\N^J$, $z'\in\C^J$ with $\Re z'_j\geq\Lambda_j=\Lambda_{\pi,j}+\Lambda_{\theta,j}+\Lambda_{\theta',j}$ for all $j\in J$, and $c(x'_{\overline{J}},x_k;k_1,k_2)$ is smooth in $x'_{\overline{J}}$, $x_k$, $k_1$ and $k_2$, and satisfies
$$ |c(x'_{\overline{J}},x_k;k_1,k_2)| \leq Ce^{-D'_{\overline{J}}\cdot x'_{\overline{J}}}e^{-Nx_k}, \qquad \forall x_{\overline{I}}\in\R_{>0}^{\overline{I}},x_k\leq0,k_1,k_2\in K. $$
Here the numbers $D'_j\in\R$ can be chosen arbitrarily large and $N\in\R$ depends on $D'_j$. The resulting integral is
\begin{multline*}
 \int_{K}\int_{K}\int_{-\infty}^0\int_{\R_{>0}^{k-1}}c(x'_{\overline{J}},x_k;k_1,k_2)(x'_J)^\alpha e^{-z'\cdot x'_J}\phi(k_1a_xk_2) \Big(\prod_{j=1}^ke^{-(js+s_j)x_j}\Big)\,d(\cdots).
\end{multline*}
Again substitute $y_j=e^{-x_j}$ and obtain
\begin{align*}
& \int_{K}\int_{K}\int_1^\infty\int_0^1\cdots\int_0^1c'(y'_{\overline{J}},y_k;k_1,k_2)\\&\times\Big(\prod_{j\in J}y_j^{js+s_j+z_j}\log^{\alpha_j}y_j\Big)\Big(\prod_{j\notin J}y_j^{js+s_j}\Big)
\phi(k_1b_yk_2)\,d(\cdots),
\end{align*}
where $b_y=\diag(y_1\cdots y_k,y_2\cdots y_k,\ldots,y_k)$ and $c'(e^{-x'_{\overline{J}}},e^{-x_k};k_1,k_2)=c(x'_{\overline{J}},x_k;k_1,k_2)$ satisfies
\begin{align*}
|c'(y'_{\overline{J}},y_k;k_1,k_2)| \leq C(y'_{\overline{J}})^{D'_{\overline{J}}}y_k^N, \qquad \forall y'_{\overline{J}}\in(0,1)^{\overline{J}},y_k\in(1,\infty),k_1,k_2\in K.
\end{align*}
Note that $y_k$ is a single coordinate in $b_y$ and $\phi\in\mathcal{S}(F_{k\times k})$. This implies that the function $\phi(k_1b_yk_2)$ together with all its $y'$-derivatives is bounded by a constant times $y_k^{-L}$ with $L\in\R$ arbitrarily large. Then $B_{c'}(y)$ (defined as above)
converges absolutely and defines a smooth function in $y_1,\ldots,y_{k-1}\in(0,1)$, $y_k\in(1,\infty)$ such that for all differential operators $\mathcal{D}$ in $y'_J$ with constant coefficients we have
$$ |\mathcal{D}B_{c'}(y)| \leq C_{\mathcal{D}}(y'_{\overline{J}})^{D'_{\overline{J}}}y_k^{N-L}, \qquad \forall y'\in(0,1)^{k-1},y_k\in(1,\infty). $$
This implies that the integral
\begin{align*}
\int_1^\infty\int_0^1\cdots\int_0^1 B_{c'}(y)\Big(\prod_{j\in J} y_j^{js+s_j+z_j}\log^{\alpha_j}y_j\Big)\Big(\prod_{j\notin J}y_j^{js+s_j}\Big)\,d(\cdots)
\end{align*}
converges absolutely whenever $\Re s>-\frac{\Re\Lambda_j+s_j+1}{j}\geq-\frac{\Re z_j+s_j+1}{j}$ for $j\in J$, $\Re s>-\frac{D_j+s_j+1}{j}$ for $j\in \overline{J}$ and $\Re s<\frac{L-N-s_k-1}{k}$. Since $L$ is arbitrarily large the last condition is superfluous and we can proceed as before.
\end{proof}

\subsection{Global theory}\label{section:global theory}
Let $F$ be a global field.
Let $\pi$ be a cuspidal (irreducible) automorphic representation of $G_{k}(\Adele)$. We construct a global integral representation.

Put $n=2k$ and let $\varphi$ and $\varphi'$ be a pair of automorphic forms belonging to the space of the global exceptional representation $\Theta$ of $\widetilde{G}_{n}(\Adele)$ (defined in \S~\ref{subsection:The exceptional representations}). The Fourier coefficient $\varphi^{U_k,\psi_k}(g)$ is a global Shalika functional on $\Theta$. By restriction it is a function
on $\widetilde{G_k}(\Adele)$, where $G_k$ is embedded in $G_n$ in the top left block of $M_k$. Then $g\mapsto\varphi^{U_k,\psi_k}(g){\varphi'}^{U_k,\psi_k^{-1}}(g)$ is well defined on $G_k(\Adele)$ and trivial on $G_k(F)$.

Let $f$ be a matrix coefficient of $\pi$, i.e.,
\begin{align*}
f(g)=\int_{\lmodulo{C_{G_k}(\Adele)G_{k}(F)}{G_{k}(\Adele)}}v(hg)v^{\vee}(h)\, dh,
\end{align*}
where $v$ (resp. $v^{\vee}$) is a cusp form in the space of $\pi$ (resp. $\pi^{\vee}$). Also let $s\in\C$. Consider the integral
\begin{align}\label{int:global integral}
Z(f,\varphi^{U_k,\psi_k},{\varphi'}^{U_k,\psi_k^{-1}},s)=\int_{G_{k}(\Adele)}f(g)\varphi^{U_k,\psi_k}(g){\varphi'}^{U_k,\psi_k^{-1}}(g)|\det{g}|^{s-k/2}\, dg.
\end{align}

\begin{theorem}\label{theorem:properties of the Global GJ integral}The integral $Z(f,\varphi^{U_k,\psi_k},{\varphi'}^{U_k,\psi_k^{-1}},s)$ satisfies the following properties.
\begin{enumerate}[leftmargin=*]
  \item\label{thm:global part 1} It is absolutely convergent for $\Re(s)\gg0$.
  \item\label{thm:global part 2} For factorizable data $f=\prod_vf_v$, $\varphi$ and $\varphi'$, write
  $\varphi^{U_k,\psi_k}=\prod_v\mathscr{S}_v$ and ${\varphi'}^{U_k,\psi_k^{-1}}=\prod_v\mathscr{S}'_v$.
  In the domain of absolute convergence, the integral is equal to a product
\begin{align*}
\prod_vZ(f_v,\mathscr{S}_v,\mathscr{S}'_v,s).
\end{align*}
Moreover, for any sufficiently large finite set of places $S$ of $F$,
\begin{align*}
Z(f,\varphi^{U_k,\psi_k},{\varphi'}^{U_k,\psi_k^{-1}},s)=\prod_{v\in S}Z(f_v,\mathscr{S}_v,\mathscr{S}'_v,s)
\frac{L^S(2s,\mathrm{Sym}^2,\pi)}{L^S(2s+1,\wedge^2,\pi)}.
\end{align*}
  \item\label{thm:global part 3}For a suitable choice of data, the integral is not identically zero (as a function of $s$).
  \item\label{thm:global part 4} It admits a meromorphic continuation to the complex plane.
\end{enumerate}\end{theorem}
\begin{proof}
We merely have to establish convergence: the integral is Eulerian by virtue of
Theorem~\ref{theorem:Shalika model given by Fourier coefficient}; the other properties follow immediately from
the corresponding local results of Theorems~\ref{theorem:local props GJ integral p-adic} and
\ref{theorem:local props GJ integral archimedean}, and the properties of
$L^S(s,\mathrm{Sym}^2,\pi)$ and $L^S(s,\wedge^2,\pi)$.

As in \cite[\S~12]{GJ}, first we show that for $r\gg0$,
\begin{align}\label{eq:convergence of integral with double global Shalika funcitonals for real s large}
\int_{G_{k}(\Adele)}|\varphi^{U_k,\psi_k}(g){\varphi'}^{U_k,\psi_k^{-1}}(g)|\cdot|\det{g}|^{r-k/2}\, dg<\infty.
\end{align}
This implies the absolute convergence of $Z(f,\varphi^{U_k,\psi_k},{\varphi'}^{U_k,\psi_k^{-1}},s)$, because $|f(g)|\leq C_f|\det g|^{r_0}$ for some $C_f>0$ independent of $g$, and a fixed $r_0$ depending only on $\pi$ (when $\pi$ is unitary, $r_0=0$).
It is enough to consider factorizable data
$\varphi$ and ${\varphi'}$. 
Then we need to bound
\begin{align}\label{eq:global product to bound}
\prod_v\int_{G_{k}(F_v)}|\mathscr{S}_{v}(g)\mathscr{S}'_{v}(g)|\cdot|\det{g}|_v^{r-k/2}\, d_vg.
\end{align}
At any $v$, the integral is finite by Theorem~\ref{theorem:local props GJ integral p-adic}~\eqref{item:p-adic integral conv} and Theorem~\ref{theorem:local props GJ integral archimedean}~\eqref{item:archimedean integral conv}.
Let $S$ be a finite set of places such that for $v\notin S$, all data are unramified. Arguing as in the proof of
Theorem~\ref{theorem:local props GJ integral p-adic}~\eqref{item:p-adic integral unramified}, for $v\notin S$,
\begin{align*}
\int_{G_{k}(F_v)}|\mathscr{S}_{v}(g)\mathscr{S}'_{v}(g)|\cdot|\det{g}|_v^{r-k/2}\, d_vg=
\sum_{\lambda\in2\Z^k_+}\mathrm{vol}(G_k(\mathcal{O}_v)\varpi_v^{\lambda}G_k(\mathcal{O}_v))\delta_{B_k}^{1/2}(\varpi_v^{\lambda})q_v^{-|\lambda|r}.
\end{align*}
Since $\mathrm{vol}(G_k(\mathcal{O}_v)\varpi_v^{\lambda}G_k(\mathcal{O}_v))\leq\delta_{B_k}^{-1}(\varpi_v^{\lambda})$, this sum is bounded by
\begin{align*}
\prod_{i=1}^k(1-q_v^{-2r-2i+k+1})^{-1}.
\end{align*}
According to Weil \cite[Chapter~VII \S~1, Proposition~1]{We2} the product of these factors over all the finite places of $F$ is finite
for $r\geq(k+1)/2$, and then
\begin{align*}
\prod_{v\notin S}\int_{G_{k}(F_v)}|\mathscr{S}_{v}(g)\mathscr{S}'_{v}(g)|\cdot|\det{g}|_v^{r-k/2}\, dg<\infty.
\end{align*}
Therefore \eqref{eq:global product to bound} is finite, when we take
$r\gg(k+1)/2$ to guarantee the convergence of the integrals also at $v\in S$.
\end{proof}
The next theorem will be used in \S~\ref{section:The co-period} to compute the co-period. It relies on the following assumption on $\pi$ and $\Theta$.
\begin{assumption}\label{Assumption:Archimedean result for co-period}
At each Archimedean place $v$ of $F$,
there is a neighborhood of $1/2$ such that the local zeta integral is holomorphic for all data, i.e.,
matrix coefficients of $\pi_v$ and local metaplectic Shalika functions.
\end{assumption}
As we show in the proof below, if $\pi$ is unitary with a trivial central character, the analogous assumption at $p$-adic places is valid. This gap will probably be closed once the exponents of $\theta_n$ are determined over Archimedean fields; for $n=2$ this can be done directly using the realization of $\theta_2$ as a Weil representation and the assumption holds.
\begin{theorem}\label{theorem:unitary cuspidal pi pole is from global part}
Assume that $\pi$ is unitary with a trivial central character, and Assumption~\ref{Assumption:Archimedean result for co-period} holds.
\begin{enumerate}[leftmargin=*]
  \item\label{item:global pole comes from the L function}The global integral has a pole at $s=1/2$ if and only if
  $L^S(s,\mathrm{Sym}^2,\pi)$ has a pole at $s=1$.
  \item\label{item:around global pole}For any data $(f,\varphi^{U_k,\psi_k},{\varphi'}^{U_k,\psi_k^{-1}})$, there is a punctured neighborhood around $1/2$
  such that the integral $Z(kf,k\varphi^{U_k,\psi_k},k{\varphi'}^{U_k,\psi_k^{-1}},s)$ is holomorphic for all $k\in K$ and $s$ in this neighborhood.
  \end{enumerate}
\end{theorem}
\begin{proof}
The first part follows from Theorem~\ref{theorem:properties of the Global GJ integral}~\eqref{thm:global part 2}, assuming we can control the integrals at the places in $S$. One can choose data, for which these do not contribute poles or zeros in a neighborhood of $1/2$, by Theorem~\ref{theorem:local props GJ integral p-adic}~\eqref{item:p-adic integral made const}
and Theorem~\ref{theorem:local props GJ integral archimedean}~\eqref{item:archimedean integral made non-zero}. Hence
if $L^S(s,\mathrm{Sym}^2,\pi)$ has a pole at $s=1$, the integral will have a pole at $s=1/2$, for some choice of data.
In the other direction, we need to show that the integrals at the places of $S$ cannot contribute any pole. At Archimedean places this is guaranteed by Assumption~\ref{Assumption:Archimedean result for co-period}. At a $p$-adic place $v\in S$ we can use Theorem~\ref{theorem:local props GJ integral p-adic}~\eqref{item:p-adic integrals holomorphic at s=1/2 for some cases}: since $\pi_v$ is a local component of a unitary cuspidal automorphic representation of $G_k(\Adele)$, it satisfies the prescribed condition on the exponents, by virtue of the results of Luo, Rudnick and Sarnak \cite{LRS} towards Ramanujan type bounds (see \cite[p.~221]{CKPS}).

For the second part, first note that if $f$ is a pure tensor, $k$-translations leave most of its components unchanged. The global partial $L$-functions are meromorphic, so one can find a punctured neighborhood of $1/2$ where their quotient is holomorphic, and we only have to deal with a finite number of local integrals. At the Archimedean places we again resort to Assumption~\ref{Assumption:Archimedean result for co-period}. For the remaining $p$-adic places we use Theorem~\ref{theorem:local props GJ integral p-adic}~\eqref{item:p-adic integrals holomorphic at s=1/2 for some cases} and its proof, to choose a neighborhood of $1/2$ depending only on the exponents of $\pi_v$, at which the local integrals are holomorphic.
\end{proof}
\begin{remark}
We mention that over a function field Lafforgue \cite{Lafforgue1,Lafforgue2} (following \cite{Drinfeld2,Drinfeld1}) proved the Generalized Ramanujan Conjecture for $G_k(\Adele)$, so $\pi_v$ is already tempered.
\end{remark}

\section{The co-period}\label{section:The co-period}

\subsection{The co-period integral}\label{section:The co-period integral}
Let $F$ be a global field and $\pi$ be a cuspidal self-dual automorphic representation of $G_k(\Adele)$ with a trivial central character. Put $n=2k$ and denote by
$A^+$ the subgroup of id\`{e}les of $F$ whose finite components are
trivial and Archimedean components are equal and positive.

Let $\rho$ be a smooth complex-valued function on $G_n(\Adele)\times M_k(\Adele)$ satisfying the following properties: for $g\in G_n(\Adele)$,
the function $g\mapsto\rho(g,I_n)$ is right $K$-finite; for any $m,m_1\in M_k(\Adele)$ and $u\in U_k(\Adele)$,
$\rho(mug,m_1)=\delta_{Q_k}^{1/2}(m)\rho(g,m_1m)$; and the function $m\mapsto\rho(g,m)$ is a cusp form in the space of $\pi\otimes\pi^{\vee}$.
The standard section corresponding to $\rho$ is defined by $\rho_s(muk,m_1)=\delta_{Q_k}^{s/k}(m)\rho(muk,m_1)$,
for any $k\in K$ and $s\in\C$. The function $\rho_s$ belongs to the space of
\begin{align*}
\Ind_{Q_k(\Adele)}^{G_n(\Adele)}(\delta_{Q_k}^{1/2+s/k}\pi\otimes\pi^{\vee}).
\end{align*}
For simplicity, denote $\rho_s(g)=\rho_s(g,I_n)$.
Form the Eisenstein series
\begin{align*}
E(g;\rho,s)=\sum_{g_0\in \lmodulo{Q_k(F)}{G_n(F)}}\rho_{s}(g_0g),\qquad g\in G_n(\Adele).
\end{align*}
Let $E_{1/2}(g;\rho)=\mathrm{Res}_{s=1/2}E(g;\rho,s)$ denote the residue of the series at $s=1/2$. The residue here is in the sense $\lim_{s\rightarrow1/2}(s-1/2)E(g;\rho,s)$. If we let $\rho$ vary, then $E_{1/2}(g;\rho)$ is nontrivial if and only if $L(\pi\times\pi,s)$ has a pole at $s=1$; equivalently, $L^S(\pi\times\pi,s)$ has a pole at $s=1$ for any sufficiently large finite set of places $S$ (\cite[1.5 and 3.16]{JS1}).

Also let $\varphi$ and $\varphi'$ be a pair of automorphic forms belonging to the space of the global exceptional representation $\Theta$ of $\widetilde{G}_{n}(\Adele)$. The central character of $\Theta$ is trivial on $\mathfrak{s}(C_{G_n}^2(\Adele))$. 
The function $g\mapsto\varphi(g)\varphi'(g)$ is well defined and moreover, it is left-invariant on a subgroup
$C'<C_{G_n}(\Adele)$ of finite index, which contains $C_{G_n}^2(\Adele)$ and is also contained in $G_k^{\triangle}(\Adele)$ (see \cite[pp.~159-160]{BG}).
Specifically, let $S'$ be a finite set of places of $F$ such that for all $v\notin S'$,
$g\mapsto\varphi(g)\varphi'(g)$ is right $K_v$-invariant, then
\begin{align*}
C'=\setof{z I_n}{z\in F^*\Adele^{*2}\prod_{v\notin S'}\mathcal{O}_v^*}.
\end{align*}
The global co-period integral is
\begin{align}\label{eq:global co-period}
CP(E_{1/2}(\blank;\rho),\varphi,\varphi')=
\int_{\lmodulo{C'G(F)}{G(\Adele)}}E_{1/2}(g;\rho)\varphi(g)\varphi'(g)\, dg.
\end{align}
This integral was defined by Shunsuke Yamana and the first named author (albeit using a slightly different quotient for the integration domain).
One may replace $E_{1/2}(\blank;\rho)$ with any complex-valued function on the quotient,
such that the integral is absolutely convergent. As long as $\varphi$ and $\varphi'$ are fixed, this
does not change the integration domain (in particular, e.g., the co-period is a linear function in the first variable).

Exactly as in \cite{me7}, since the cuspidal support of $E_{1/2}(\blank;\rho)$ is $\delta_{Q_k}^{-1/(2k)}\pi\otimes\pi^{\vee}$ and
the only cuspidal exponent of $\Theta$ is $\delta_{B_n}^{-1/4}$
\cite[Proposition II.1.2, Theorems II.1.4 and II.2.1]{KP}, and using \cite[Lemma~I.4.1]{MW2}, one proves the following claim:
\begin{claim}\label{claim:co period is absolutely convergent}
Integral~\eqref{eq:global co-period} is absolutely convergent.
\end{claim}
For the function $\rho$, and with $C'$ selected as described above (depending
on $\varphi$ and $\varphi'$) define
\begin{align*}
f_{\rho}(g)=\int_{\lmodulo{C'G_k^{\triangle}(F)}{G_k^{\triangle}(\Adele)}}\rho(I_n,c\ \diag(g,I_k))\, dc,\qquad g\in G_k(\Adele).
\end{align*}
As noted above, $C'<G_k^{\triangle}(\Adele)$. Then $f_{\rho}$
is a matrix coefficient of $\pi$. Since the integrand is left-invariant on $C_{G_n}(\Adele)$, taking
a different $C'$ only multiplies $f_{\rho}$ by some nonzero constant.

\begin{theorem}\label{theorem:co-period}Assume $F$ is a function field or Assumption~\ref{Assumption:Archimedean result for co-period} holds.
The co-period~\eqref{eq:global co-period} satisfies the following properties.
\begin{enumerate}[leftmargin=*]
\item\label{item:co-period thm part 1} $CP(E_{1/2}(\blank;\rho),\varphi,\varphi')=\int_{K}\mathrm{Res}_{s=1/2}Z(f_{k\rho},k{\varphi}^{U_k,\psi_k},k{\varphi'}^{U_k,\psi_k^{-1}},s)\, dk$.
\item\label{item:co-period thm part 2} It is nonzero for some choice of data $(\rho,\varphi,\varphi')$ if and only if
\begin{align*}\mathrm{Res}_{s=1/2}Z(f_{\rho_1},\varphi_1^{U_k,\psi_k},{\varphi_1'}^{U_k,\psi_k^{-1}},s)\end{align*}
does not vanish for some matrix coefficient $f_{\rho_1}$ of $\pi$ and
automorphic forms $\varphi_1,\varphi_1'$ in the space of $\Theta$.
\item \label{item:co-period thm part 3}The co-period is nonzero if and only if $L^S(s,\mathrm{Sym}^2,\pi)$ has a pole at $s=1$.
\end{enumerate}
\end{theorem}
\begin{remark}
Part~\eqref{item:co-period thm part 3} was obtained by Shunsuke Yamana and the first named author, it is included here because we can also deduce it from Theorem~\ref{theorem:unitary cuspidal pi pole is from global part}.
\end{remark}
\begin{proof}
We apply the truncation operator of Arthur \cite{A2,A1} to $E(g;\rho,s)$ as in, e.g.,
\cite{JR,Jng,GRS7,GJR2,me7}. For a real number $d>1$, let $ch_{>d}:\R_{>0}\rightarrow\{0,1\}$ be the characteristic function of $\R_{>d}$ and set $ch_{\leq d}=1-ch_{>d}$. Also extend $\delta_{Q_k}$ to $G_n(\Adele)$ via the Iwasawa decomposition.
The truncation operator is given by
\begin{align*}
\Lambda_dE(g;\rho,s)=E(g;\rho,s)-\sum_{g_0\in\lmodulo{Q_k(F)}{G_n(F)}}E^{U_k}(g_0 g;\rho,s)ch_{>d}(\delta_{Q_k}(g_0 g)),
\end{align*}
where $E^{U_k}$ is the constant term of the series along $U_k$. Since $\pi$ is cuspidal,
$E^{U_k}(g;\rho,s)=\rho_s(g)+M(w,s)\rho_s(g)$. Here
$M(w,s)$ is the global intertwining operator corresponding to a Weyl element $w$ taking $U_k$ to $U_k^-$. Write
\begin{align*}
\Lambda_dE(g;\rho,s)=\mathcal{E}_1^d(g;s)-\mathcal{E}_2^d(g;s),
\end{align*}
an equality between meromorphic functions in $s$, where
\begin{align*}
&\mathcal{E}_1^d(g;s)=\sum_{g_0\in\lmodulo{Q_k(F)}{G_n(F)}}\rho_s(g_0 g)ch_{\leq d}(\delta_{Q_k}(g_0 g)),\\
&\mathcal{E}_2^d(g;s)=\sum_{g_0\in\lmodulo{Q_k(F)}{G_n(F)}}M(w,s)\rho_s(g_0 g)ch_{>d}(\delta_{Q_k}(g_0 g)).
\end{align*}
\begin{claim}\label{claim:conv global co-periods}
Fix $d>1$. The integral obtained from \eqref{eq:global co-period} by replacing $E_{1/2}(\blank;\rho)$ with $\mathcal{E}_j^d(\blank;s)$ and taking the absolute value
of $\rho_s$ ($M(w,s)\rho_s$ for $j=2$), $\varphi$ and $\varphi'$, is finite for $\Re(s)\gg0$.
\end{claim}
It follows that for fixed $d>1$ and $\Re(s)\gg0$, $CP(\Lambda_dE(\blank;\rho,s),\varphi,\varphi')$ is absolutely convergent and
\begin{align}\label{eq:truncation formula 1}
CP(\Lambda_dE(\blank;\rho,s),\varphi,\varphi')=
CP(\mathcal{E}_1^d(\blank;s),\varphi,\varphi')-CP(\mathcal{E}_2^d(\blank;s),\varphi,\varphi').
\end{align}
The heart of the argument is in the following result.
\begin{proposition}\label{proposition:technical unfolding truncation}
Assume $d$ is sufficiently large.
The integrals $CP(\mathcal{E}_j^d(\blank;s),\varphi,\varphi')$ admit meromorphic continuation to the plane.
There exist complex-valued functions $\alpha_0(s)$, $\alpha_1(s,d)$, $\alpha_2(s)$ and $\alpha_3(d,s)$ such that
\begin{align*}
&CP(\mathcal{E}_1^d(\blank;s),\varphi,\varphi')=\frac{d^{s/k}}{sk}
\alpha_0(s)+\alpha_1(d,s),\\
&CP(\mathcal{E}_2^d(\blank;s),\varphi,\varphi')=\frac{d^{-s/k}}{sk}\alpha_2(s)+\alpha_3(d,s).
\end{align*}
The function $\alpha_0(s)$ is entire; for $j>0$ the functions $\alpha_j(\cdots)$ are meromorphic in $s$ (for fixed $d$ if $j=1,3$);
\begin{align*}
&\mathrm{Res}_{s=1/2}\alpha_1(d,s)=\int_{K}\mathrm{Res}_{s=1/2}Z(f_{k\rho},k\varphi^{U_k,\psi_k},k{\varphi'}^{U_k,\psi_k^{-1}},s)\, dk,
\end{align*}
$\alpha_2(s)$ has at most a simple pole at $s=1/2$, and
\begin{align*}
&\lim_{d\rightarrow\infty}\mathrm{Res}_{s=1/2}\alpha_3(d,s)=0.
\end{align*}
Furthermore, when $F$ is a number field, $\alpha_0(s)=\alpha_0$ is a constant, and if $F$ is a function field, we in fact have $CP(\mathcal{E}_2^d(\blank;s),\varphi,\varphi')=0$.
\end{proposition}
Exactly as argued in \cite[pp.~14-15]{me7}, the properties of
the truncated series imply that when we take the residue in \eqref{eq:truncation formula 1},
\begin{align*}
CP(\Lambda_dE_{1/2}(\blank;\rho),\varphi,\varphi')=
\mathrm{Res}_{s=1/2}CP(\mathcal{E}_1^d(\blank;s),\varphi,\varphi')-\mathrm{Res}_{s=1/2}CP(\mathcal{E}_2^d(\blank;s),\varphi,\varphi').
\end{align*}
According to Labesse and Waldspurger \cite[Proposition~4.3.3]{LW}, $\Lambda_dE_{1/2}(g;\rho)$ is bounded by a sum
$E_{1/2}(g;\rho)+\vartheta(g,d;\rho)$, where $\vartheta(g,c;\rho)$ is uniformly rapidly decreasing in $d$ and rapidly decreasing in $g$. Since $|\Lambda_dE(\blank;\rho,s)|\leq|E(\blank;\rho,s)|$, Claim~\ref{claim:co period is absolutely convergent} and the Dominated Convergence Theorem imply that we can take the limit $d\rightarrow\infty$ under the integral sign on the \lhs\ of the last equality. Then according to the proposition,
\begin{align*}
CP(E_{1/2}(\blank;\rho),\varphi,\varphi')=\int_{K}\mathrm{Res}_{s=1/2}Z(f_{k\rho},k\varphi^{U_k,\psi_k},k{\varphi'}^{U_k,\psi_k^{-1}},s)\, dk.
\end{align*}
This completes part~\eqref{item:co-period thm part 1} of the theorem.

The $dk$-integral has a pole at $s=1/2$ for some choice of data $(\rho,\varphi,\varphi')$ if and only if
$\mathrm{Res}_{s=1/2}Z(\cdots)\ne0$, for some choice of data $(f_{\rho_1},\varphi_1,\varphi_1')$. This follows exactly as in \cite[Theorem~3.2]{GJR2} (see also \cite[Proposition~2]{JR} and \cite{Jng}), their arguments are general and apply to our setting. The proof of the second part of the theorem is complete.
The third part follows immediately from the second and
Theorem~\ref{theorem:unitary cuspidal pi pole is from global part}~\eqref{item:global pole comes from the L function}.
\end{proof}
\begin{proof}[Proof of Claim~\ref{claim:conv global co-periods}]
Collapsing the summation of the Eisenstein series into the integral and using the Iwasawa decomposition we obtain
\begin{align*}
&\int_K\int_{\lmodulo{C'M_k(F)}{M_k(\Adele)}}|\rho_s|(mk)ch_{\leq d}(\delta_{Q_k}(m))\delta_{Q_k}^{r/k-1/2}(m)\\
&\times\int_{\lmodulo{U_k(F)}{U_k(\Adele)}}|\varphi\varphi'|(umk)\, du \, dm\, dk,
\end{align*}
where $r=\Re(s)$.
This is finite for $r\gg0$, because $\varphi$ and $\varphi'$ have moderate growth, and $m\mapsto\rho(m)$ is a cusp form.
\end{proof}
\begin{proof}[Proof of Proposition~\ref{proposition:technical unfolding truncation} - number field case]
First consider $CP(\mathcal{E}_1^d(\blank;s),\varphi,\varphi')$.
We take $s$ in the domain of absolute convergence, where our integral manipulations are justified.
Collapsing the summation and using the Iwasawa decomposition,
\begin{align*}
CP(\mathcal{E}_1^d(\blank;s),\varphi,\varphi')=
&\int_{K}\int_{\lmodulo{C'M_k(F)}{M_k(\Adele)}}\rho_s(mk)ch_{\leq d}(\delta_{Q_k}(m))\\
&\times\int_{\lmodulo{U_k(F)}{U_k(\Adele)}}\varphi(umk)\varphi'(umk)\, du\  \delta_{Q_k}^{-1}(m)\, dm \, dk.
\end{align*}
Next we appeal to the results of \S~\ref{section:Fourier coefficients} and replace $\varphi$ with its Fourier expansion along $U_k$.
In the notation of \S~\ref{section:Fourier coefficients}, the integral becomes a sum of integrals
$\sum_{j=0}^k\mathrm{I}_j$, where
\begin{align*}
\mathrm{I}_j=&\int_{K}\int_{\lmodulo{C'\mathrm{St}_{n,k}(\psi_j)(F)}{M_k(\Adele)}}\rho_s(mk)ch_{\leq d}(\delta_{Q_k}(m))
\varphi^{U_k,\psi_j}(mk){\varphi'}^{U_k,\psi_j^{-1}}(mk)\delta_{Q_k}^{-1}(m)\, dm \, dk.
\end{align*}
This formal manipulation is justified using the fact that $\varphi$ has uniform moderate growth. See \cite[p.~17]{me7} for the complete argument leading to the sum of integrals $\mathrm{I}_j$, the identities here are very similar (and simpler).
By virtue of Proposition~\ref{proposition:Fourier coefficients are trivial along v and y}, the integrals $\mathrm{I}_j$ with
$0<j<k$ vanish, because we may introduce an inner integration $\int_{\lmodulo{V_j(F)}{V_j(\Adele)}}\rho_s(vg)dv$
along a unipotent subgroup $V_j$ of $M_k$, which vanishes because $m\mapsto\rho(m)$ is a cusp form.

Consider $\mathrm{I}_0$. Then $\psi_0=1$ and $\mathrm{St}_{n,k}(\psi_0)=M_k$. For a subgroup $H<M_k(\Adele)$, let
$H^1=\setof{h\in H}{|\delta_{Q_k}(h)|=1}$. Since $C'M_k(F)<M_k(\Adele)^1$ and
\begin{align*}
M_k(\Adele)=M_k(\Adele)^1\times \setof{\diag(tI_k,I_k)}{t\in A^+},
\end{align*}
we obtain
\begin{align*}
\mathrm{I}_0=&\int_{K}\int_{\lmodulo{C'M_k(F)}{M_k(\Adele)^1}}\int_{\R_{>0}}
\rho_s\varphi^{U_k}{\varphi'}^{U_k}(\diag(tI_k,I_k)mk)\\&\times ch_{\leq d}(t^{k^2})\delta_{Q_k}^{-1}(\diag(tI_k,I_k))\ t^{-1}dt\, dm \, dk.
\end{align*}
Here we identified $A^+$ with $\R_{>0}$, $dt$ is the standard Lebesgue measure on $\R$ and the constant normalizing the measure $dm$ on the \rhs\ is
omitted. For $m\in M_k(\Adele)^1$,
\begin{align}\label{eq:rho s emits tI_k}
&\rho_s(\diag(tI_k,I_k)mk)=
\delta_{Q_k}^{1/2+s/k}(\diag(tI_k,I_k))\rho(mk)
.
\end{align}
By virtue of Theorem~\ref{theorem:constant term theorem}, the function
$\varphi^{U_k}$ on $\widetilde{M}_k(\Adele)$ belongs to the space of
$\delta_{Q_k}^{1/4}\Theta_{M_k}$, where $\Theta_{M_k}$ corresponds to a character which is
a lift of $(\delta_{B_k}^{1/4}\otimes\delta_{B_k}^{1/4})|_{T_n^2}$ (the chosen Borel subgroup of $M_k$ is
$B_k\times B_k$). Therefore
\begin{align}\label{eq:varphi U emits tI_k}
&\varphi^{U_k}{\varphi'}^{U_k}(\diag(tI_k,I_k)mk)=\delta_{Q_k}^{1/2}(\diag(tI_k,I_k))\varphi^{U_k}{\varphi'}^{U_k}(mk),
\end{align}
Combining \eqref{eq:rho s emits tI_k} and \eqref{eq:varphi U emits tI_k}
we obtain an inner integral
\begin{align}\label{int:inner dt integral}
\int_{0<t^{k^2}\leq d}t^{sk-1}\, dt=\frac{d^{s/k}}{sk},
\end{align}
multiplied by a $dmdk$-integral, which is independent of $s$ and $d$. Therefore $\mathrm{I}_0=\frac{d^{s/k}}{sk}\alpha_0$.

Regarding $\mathrm{I}_k$, $\psi_k$ is the Shalika character and $\mathrm{St}_{n,k}(\psi_k)=G_k^{\triangle}$. According to
Theorem~\ref{theorem:Shalika model given by Fourier coefficient},
\begin{align*}
\varphi^{U_k,\psi_k}{\varphi'}^{U_k,\psi_k^{-1}}(c^{\triangle}mk)=
\varphi^{U_k,\psi_k}{\varphi'}^{U_k,\psi_k^{-1}}(mk),\qquad\forall c\in G_k(\Adele).
\end{align*}
Therefore we may introduce the inner integration given by the formula for $f_{\rho}$ and
\begin{align*}
\mathrm{I}_k=&\int_{K}\int_{G_k(\Adele)}f_{k\rho}(g)\varphi^{U_k,\psi_k}{\varphi'}^{U_k,\psi_k^{-1}}(\diag(g,I_k)k)|\det{g}|^{s-k/2}ch_{\leq d}(|\det{g}|^k)\, dg \, dk.
\end{align*}
Consider the inner integral, for a fixed $k\in K$. Without the function $ch_{\leq d}$, it is the integral
$Z=Z(f_{k\rho},k\varphi^{U_k,\psi_k},k{\varphi'}^{U_k,\psi_k^{-1}},s)$ given by \eqref{int:global integral}. Denote by $Z_{\leq d}$ (resp. $Z_{>d}$) the integral obtained when the integrand is multiplied by $ch_{\leq d}$ (resp. $ch_{>d}$). Then for $\Re(s)\gg0$, $Z=Z_{\leq d}+Z_{>d}$. The following claim is the main observation needed to complete the proof.
\begin{claim}\label{claim:nice props of Z > d}
Fix $d$, sufficiently large. For any $s$, $Z_{>d}$ is absolutely convergent, uniformly for $s$ in compact subsets of $\C$. Hence $Z_{>d}$ has a holomorphic continuation to the plane. Moreover for any fixed $s$, $Z_{>d}=0$ as $d\rightarrow\infty$.
\end{claim}
The proof will be given below.
Assume that $d$ is large enough. By Theorem~\ref{theorem:properties of the Global GJ integral}~\eqref{thm:global part 4},
$Z$ admits a meromorphic continuation to $\C$, and therefore the claim implies that $Z_{\leq d}$ has a meromorphic continuation to the plane. Furthermore, because $Z_{>d}$ is holomorphic, we deduce that for any $s_0$,
$\mathrm{Res}_{s=s_0}Z=\mathrm{Res}_{s=s_0}Z_{\leq d}$, and since the \lhs\ is independent of $d$, the residues of
$Z_{\leq d}$ are also independent of $d$.

Because $K$ is compact, $\mathrm{I}_k$ also has a meromorphic continuation.
It follows that
\begin{align*}
\mathrm{Res}_{s=1/2}\mathrm{I}_k=&\mathrm{Res}_{s=1/2}\int_{K}Z_{\leq d}(f_{k\rho},k\varphi^{U_k,\psi_k},k{\varphi'}^{U_k,\psi_k^{-1}},s)\, dk.
\end{align*}
According to Theorem~\ref{theorem:unitary cuspidal pi pole is from global part}~\eqref{item:around global pole} and its proof, at $s=1/2$ the integral $Z$ has at most a simple pole, and there is a punctured neighborhood of $1/2$ where the integral is holomorphic for all $k\in K$. The same applies to $Z_{\leq d}$, hence by virtue of Cauchy's Residue Theorem and Fubini's Theorem we can take the residue under the integral sign and obtain
\begin{align*}
&\int_{K}\mathrm{Res}_{s=1/2}Z_{\leq d}(f_{k\rho},k\varphi^{U_k,\psi_k},k{\varphi'}^{U_k,\psi_k^{-1}},s)\, dk\\
&=\int_{K}\mathrm{Res}_{s=1/2}Z(f_{k\rho},k\varphi^{U_k,\psi_k},k{\varphi'}^{U_k,\psi_k^{-1}},s)\, dk.
\end{align*}
The \rhs\ is independent of $d$. Hence $\mathrm{I}_k$ qualifies as the function $\alpha_1(s,d)$.
In addition, now that we proved that both $\mathrm{I}_0$ and $\mathrm{I}_k$ admit meromorphic continuations,
the same holds for $CP(\mathcal{E}_1^d(\blank;s),\varphi,\varphi')$.

Next we handle $CP(\mathcal{E}_2^d(\blank;s),\varphi,\varphi')$. We may proceed as above and obtain
a sum of two integrals $\mathrm{I}_0+\mathrm{I}_k$, where $\rho_s$ and $ch_{\leq d}$ are replaced by
$M(w,s)\rho_s$ and $ch_{>d}$. Now $\mathrm{I}_0$ takes the form
\begin{align*}
\mathrm{I}_0=&\int_{K}\int_{\lmodulo{C'M_k(F)}{M_k(\Adele)^1}}\int_{\R_{>0}}
M(w,s)\rho_s\varphi^{U_k}{\varphi'}^{U_k}(\diag(tI_k,I_k)mk)\\&\times ch_{>d}(t^{k^2})\delta_{Q_k}^{-1}(\diag(tI_k,I_k))\ t^{-1}dt\, dm \, dk.
\end{align*}
Since
\begin{align*}
&M(w,s)\rho_s(\diag(tI_k,I_k)mk)=
\delta_{Q_k}^{1/2-s/k}(\diag(tI_k,I_k))M(w,s)\rho_s(mk),
\end{align*}
the inner integral in this case becomes
\begin{align*}
\int_{d<t^{k^2}<\infty}t^{-sk-1}\, dt.
\end{align*}
Thus
\begin{align*}
\mathrm{I}_0=&\frac{d^{-s/k}}{sk}\int_{K}\int_{\lmodulo{C'M_k(F)}{M_k(\Adele)^1}}
M(w,s)\rho_s\varphi^{U_k}{\varphi'}^{U_k}(mk)\, dm\, dk.
\end{align*}
The $dmdk$-integral still depends on $s$, because of the intertwining operator, but it is a meromorphic function
in $s$ and if it has a pole at $s=1/2$, this pole is simple. Hence the constant $\alpha_0$ is replaced by a function $\alpha_2(s)$ with the stated properties.

Regarding $\mathrm{I}_k$, with the above notation
\begin{align*}
\mathrm{I}_k=&\int_{K}Z_{>d}(f_{kM(w,s)\rho_s},k\varphi^{U_k,\psi_k},k{\varphi'}^{U_k,\psi_k^{-1}},-s)\, dk.
\end{align*}
Note that since $m\mapsto\rho(g,m)$ is a cusp form, the matrix coefficient $f_{M(w,s)\rho_s}$ is a meromorphic function in $s$ and
$\mathrm{Res}_{s=1/2}f_{M(w,s)\rho_s}=f_{\mathrm{Res}_{s=1/2}M(w,s)\rho_s}$.
The convergence properties of $Z_{>d}$ stated above, imply that $\mathrm{I}_k$ admits a meromorphic continuation whose poles are contained in the set of poles of $M(w,s)$, and also
\begin{align*}
\mathrm{Res}_{s=1/2}\mathrm{I}_k
=\int_{K}Z_{>d}(f_{\mathrm{Res}_{s=1/2}kM(w,s)\rho_s},k\varphi^{U_k,\psi_k},k{\varphi'}^{U_k,\psi_k^{-1}},-s)\, dk.
\end{align*}
Hence when we let $d\rightarrow\infty$, $\mathrm{Res}_{s=1/2}\mathrm{I}_k=0$.
\end{proof}
\begin{proof}[Proof of Claim~\ref{claim:nice props of Z > d}]
Consider the function $\phi(g)=|\varphi^{U_k,\psi_k}{\varphi'}^{U_k,\psi_k^{-1}}(\diag(g,I_k))|$ on $G_k(\Adele)$.
The global matrix coefficient $f_{\rho}$ is bounded, therefore it is enough to establish the stated convergence properties for an integral of the form
\begin{align*}
\int_{G_k(\Adele)}\phi(g)ch_{>d}(|\det g|)|\det g|^r\, dg=
\int_{G_k(\Adele)^1}\int_{\R_{>d^{1/k}}}\phi(gt)t^{kr-1}\, dt\, dg,\quad r\in\R.
\end{align*}
But according to \eqref{eq:convergence of integral with double global Shalika funcitonals for real s large},
for $d>1$, this integral is finite even for $r\gg 0$.
\end{proof}
\begin{proof}[Proof of Proposition~\ref{proposition:technical unfolding truncation} - function field case]
We argue as in the case of a number field and use similar notation, but the proof is greatly simplified.
We write $CP(\mathcal{E}_1^d(\blank;s),\varphi,\varphi')$ as a sum of integrals, the only
surviving ones are $\mathrm{I}_0$ and $\mathrm{I}_k$. We can then write $\mathrm{I}_0$ in the form
\begin{align*}
\mathrm{I}_0=&\int_{K}\int_{\lmodulo{(C_{G_k}(\Adele)C'M_k(F))}{M_k(\Adele)}}\int_{\lmodulo{C_{G_k}(F)}{C_{G_k}(\Adele)}}
\rho_s\varphi^{U_k}{\varphi'}^{U_k}(zmk)\\&\times ch_{\leq d}(\delta_{Q_k}(zm))\delta_{Q_k}^{-1}(zm)\, dz\, dm \, dk.
\end{align*}
Since over a function field, a cusp form on $G_k(\Adele)$ is compactly supported modulo
$C_{G_k}(\Adele)G_k(F)$, $\mathrm{I}_0$ becomes a linear combination of integrals of the form
\begin{align}\label{int:inner dt integral for function field}
\int_{\lmodulo{F^*}{\Adele^*}}\rho_s\varphi^{U_k}{\varphi'}^{U_k}(\diag(tI_k,I_k))ch_{\leq d}(t^{k^2})\delta_{Q_k}^{-1}(\diag(tI_k,I_k))\, dt.
\end{align}
Set $(\Adele^*)^1=\setof{a\in\Adele^*}{|a|=1}$ and identify
$\lmodulo{(\Adele^*)^1}{\Adele^*}$ with $\R_{>0}$. Since \eqref{eq:rho s emits tI_k} and \eqref{eq:varphi U emits tI_k} still hold, the integrand is invariant on $\setof{\diag(tI_k,I_k)}{t\in(\Adele^*)^1}$. Thus we
see that \eqref{int:inner dt integral for function field} equals \eqref{int:inner dt integral} multiplied by
the volume of $\lmodulo{F^*}{(\Adele^*)^1}$ (which is finite).

Next note that according to Claim~\ref{claim:general bound on Shalika function},
the projection of the support of $\varphi^{U_k,\psi_k}$ on $G_k(\Adele)$ is contained in the support of some $\phi\in\mathcal{S}(\Adele_{k\times k})$.
Hence if $g\in G_k(\Adele)$ satisfies $|\det{g}|^k>d$ for a large enough $d$ (depending only on
$\varphi$), $\varphi^{U_k,\psi_k}{\varphi'}^{U_k,\psi_k^{-1}}(\diag(g,I_k))=0$. This means that the integrals $Z_{>d}$ vanish identically. We can conclude that $\mathrm{I}_k$ satisfies the stated properties and moreover, $CP(\mathcal{E}_2^d(\blank;s),\varphi,\varphi')=0$, because in the corresponding summands $\mathrm{I}_0$ and $\mathrm{I}_k$ appears $ch_{>d}$.
\end{proof}

\bibliographystyle{alpha}
\def\cprime{$'$} \def\cprime{$'$}

\end{document}